\documentclass[11pt]{amsart}
\usepackage{setspace}
\usepackage{xcolor}
\usepackage{makecell}
\usepackage{amssymb,amscd,amsthm,verbatim,amsmath,color,fancyhdr,mathrsfs,amsfonts,amssymb,commath, graphicx,bbm}
\usepackage[bookmarks=false]{hyperref}
\hypersetup{colorlinks,allcolors=blue}
\usepackage[english]{babel}
\usepackage{parskip}
\usepackage{enumitem}
\usepackage[letterpaper,top=2cm,bottom=2cm,left=3cm,right=3cm,marginparwidth=1.75cm]{geometry}

\newcommand{\R}{\mathbb{R}}

\newcommand{\N}{\mathbb{N}}
\newcommand{\E}[2]{\mathbb{E}_{#1}\left[#2\right]}
\newcommand{\of}[1]{\left(#1\right)}
\newcommand{\off}[1]{\left[#1\right]}

\DeclareMathOperator{\Id}{\mathrm{Id}}

\newcommand{\Was}[1]{\mathbb{W}_{#1}}

\newtheorem{corollary}[]{Corollary}

\newcommand{\cN}{\mathcal{N}}
\DeclareMathOperator*{\argmin}{argmin}

\newcommand{\cF}{\mathcal{F}}
\newcommand{\cP}{\mathcal{P}}

\newcommand{\cU}{\mathcal{U}}
\newcommand{\cO}{\mathcal{O}}

\newcommand{\Ent}{\mathrm{Ent}}

\newcommand{\kl}[2]{H\left(#1 | #2\right)}
\newcommand{\Tr}{\text{Tr}}

\newcommand{\Tan}[1]{\text{Tan}(#1)}
\newcommand{\Schro}{\text{Schr\"{o}dinger}}

\newcommand{\Exp}[1]{\mathrm{E}_{#1}}
\newcommand{\ac}{\cP_2^{\text{ac}}}
\newcommand{\WasDiff}{\nabla_{\mathbb{W}}}
\newcommand{\eps}{\varepsilon}

\newtheorem{lemma}{Lemma}
\newtheorem{remark}{Remark}

\newtheorem{definition}{Definition}

\newtheorem{theorem}{Theorem}
\newtheorem{assumption}{Assumption}

\newcommand{\SB}[3]{\text{SB}_{#2}^{#3}(#1)}
\newcommand{\GA}[3]{\text{GA}_{#2}^{#3}(#1)}
\newcommand{\LD}[3]{\text{LD}_{#2}^{#3}(#1)}

\newcommand{\SBstatic}[2]{\pi_{#1,#2}}

\newcommand{\SBdynam}[2]{S_{#1,#2}}
\newcommand{\LDstatic}[2]{\ell_{#1,#2}}

\newcommand{\BMstatic}[2]{\mu_{#1,#2}}

\newcommand{\BPbase}[2]{\mathcal{B}_{#1,#2}}
\newcommand{\LBPbase}[2]{\mathcal{L}_{#1,#2}}

\newcommand{\opt}[3]{S_{#1}^{#2}\of{#3}}

\newcommand{\commentout}[1]{}

\title[Iterated Schr\"odinger bridge]{Iterated Schr\"odinger bridge approximation\\to Wasserstein Gradient Flows}

\author{Medha Agarwal}
\address{Medha Agarwal\\ Department of Statistics \\ University of Washington\\ Seattle WA 98195, USA\\ {Email: medhaaga@uw.edu}}
\author{Zaid Harchaoui}
\address{Zaid Harchaoui\\ Department of Statistics \\ University of Washington\\ Seattle WA 98195, USA\\ {Email: zaid@uw.edu}}
\author{Garrett Mulcahy}
\address{Garrett Mulcahy\\ Department of Mathematics \\ University of Washington\\ Seattle WA 98195, USA\\ {Email: gmulcahy@uw.edu}}
\author{Soumik Pal}
\address{Soumik Pal\\ Department of Mathematics \\ University of Washington\\ Seattle WA 98195, USA\\ {Email: soumik@uw.edu}}

	

\thanks{This research is partially supported by the following grants. All the authors are supported by NSF grant DMS-2134012. Additionally, Pal is supported by NSF grant DMS-2133244; Harchaoui is supported by NSF grant CCF-2019844, DMS-2023166, DMS-2133244. Mulcahy is supported by the National Science Foundation Graduate
Research Fellowship Program under Grant No.\ DGE-2140004. Thanks to PIMS Kantorovich Initiative for facilitating this collaboration supported through a PIMS PRN and the NSF Infrastructure grant DMS 2133244. The authors are listed in alphabetical order.}

\date{\today}

\begin{document}

\begin{abstract}
% We introduce a novel discretization scheme for Wasserstein gradient flows that involves successively computing Schr\"{o}dinger bridges with the same marginals. This is different from both the forward/geodesic approximation and the backward/Jordan-Kinderlehrer-Otto (JKO) approximations. Since functions of the Schr\"odinger bridges can be effectively computed from data, this may be useful in statistical estimation when measures are known via a finite number of particles. A crucial part of our proof is a low temperature approximation of the Schr\"odinger bridge between the same marginals by the joint distribution of a corresponding Langevin diffusion running in stationarity. Our results give precise conditions under which the analysis of \cite{sander_22} on the self-attention mechanism of a variant of the Transformer neural network architecture, called Sinkformer by the authors, works.  
We introduce a novel discretization scheme for Wasserstein gradient flows that involves successively computing Schr\"{o}dinger bridges with the same marginals. This is different from both the forward/geodesic approximation and the backward/Jordan-Kinderlehrer-Otto (JKO) approximations. The proposed scheme has two advantages: one, it avoids the use of the score function, and, two, it is amenable to particle-based approximations using the Sinkhorn algorithm.
%Since functions of the Schr\"odinger bridges are amenable to particle-based approximations, the proposed discretization scheme can be used to consistently approximate Wasserstein gradient flows. %Give an overview, relate to, Theorem 1-2-4-5. 
% A  crucial part of our our proof is a low temperature approximation of the Schr\"odinger bridge between the same marginals by the joint distribution of a corresponding Langevin diffusion running in stationarity. 
%The approach hinges upon a inequality showing that the joint distribution of a stationary Langevin diffusion is close to the Schr\"{o}dinger bridge with equal marginals in terms of relative entropies, and with an explicit dependence on a new notion of Fisher information. 
Our proof hinges upon showing that relative entropy between the Schr\"{o}dinger bridge with the same marginals at temperature $\eps$ and the joint distribution of a stationary Langevin diffusion at times zero and $\eps$ is of the order $o(\eps^2)$ with an explicit dependence given by Fisher information.
Owing to this inequality, we can show, using a triangular approximation argument, that the interpolated iterated application of the Schr\"{o}dinger bridge approximation converge to the Wasserstein gradient flow, for a class of gradient flows, including the heat flow. The results also provide a probabilistic and rigorous framework for the convergence of the self-attention mechanisms in transformer networks to the solutions of heat flows, first observed in the inspiring work~\cite{sander_22} in machine learning research.   
\end{abstract}

\maketitle

\section{Introduction} \label{sec:introduction}
Let $\cP_2(\R^d)$ denote the space of square-integrable Borel probability measures on $\R^d$. Throughout this paper we will let $\rho \in \cP_2(\mathbb{R}^{d})$ refer to both the measure and the its Lebesgue density, whenever it exists. $\cP_2(\R^d)$ can be turned into a metric space equipped with the Wasserstein-$2$ metric \cite[Chapter 7]{ambrosio2005gradient}. We will refer to this metric space as simply the Wasserstein space. For any pair of probability measures $(\rho_1, \rho_2)$ in the Wasserstein space, let $\Pi(\rho_1, \rho_2)$ denote the set of couplings, i.e., joint distributions, with marginals $(\rho_1, \rho_2)$. Further, define the relative entropy, also known as Kullback-Leibler (KL) divergence, between $\rho_1$ and $\rho_2$ as
\begin{align}
    H(\rho_1|\rho_2) := \begin{cases}
        \int \log\left(\frac{d\rho_1}{d\rho_2}\right) d\rho_1 &\text{if $\rho_1 << \rho_2$} \\
        +\infty &\text{otherwise}
    \end{cases}
\end{align}

% A goal of widespread practical applications is to approximate the continuous time curve $(\rho_t; t\in [0,T])$ for some $0<T<\infty$ in an iterative discrete-time manner. The geodesic approximation of $(\rho_t; t\in [0,T])$ for any $t\in[0,T]$ is given by $\rho_{t+\eps} \approx (\Id + \eps v_t)_\#\rho_t$. This approximation is analogous to the explicit Euler updates in the Euclidean space. While the theoretical convergence properties of such geodesic approximations can be ascertained by drawing from the rich optimization literature in general Polish spaces, the practical application of such schemes are limited. For one thing, evaluating the velocity field for each particle in a finite particle system can be prohibitive in a general setup where the velocity field is not analytically available. 
%In this paper, we propose an iterative scheme for approximating AC curves in the Wasserstein space using Schrödinger bridges with same marginals.
%Let $\SBstatic{\rho}{\eps} \in \Pi(\rho,\rho)$ 

To understand our results let us start with an example that was first described in \cite{sander_22}.  
The static Schr\"{o}dinger bridge (\cite{schroLeonard13}) at temperature $\eps$ with both marginals $\rho$ is the solution to the following optimization problem
\begin{align}\label{eq:schrodinger}
    \SBstatic{\rho}{\eps}  &:= \argmin_{\pi \in \Pi(\rho,\rho)} H(\pi|\BMstatic{\rho}{\eps}),
\end{align}
where $\BMstatic{\rho}{\eps} \in \cP(\mathbb{R}^{d} \times \mathbb{R}^{d})$ has density
\begin{align}\label{eq:base-for-SB}
\BMstatic{\rho}{\eps}(x,y) := \rho(x)\frac{1}{(2\pi\eps)^{d/2}}\exp\left(-\frac{1}{2\eps}\|x-y\|^{2}\right).    
\end{align}
If $(X,Y) \sim \SBstatic{\rho}{\eps}{}$, define the \textbf{barycentric projection} \cite{pooladian2022entropic} as a function from $\R^d$ to itself given by 
\begin{align}\label{eq:bary-proj}
    \BPbase{\rho}{\eps}(x) &= \Exp{\SBstatic{\rho}{\eps}{}}[Y|X=x].
\end{align}
Now define a map from $\cP_2(\R^d)$ to itself given by the following pushforward
\begin{equation}\label{eq:SBheatflow}
\SB{\rho}{\eps}{}:= \of{2\Id - \BPbase{\rho}{\eps}}_\#\rho.
\end{equation}
We call this map $\SB{\cdot}{\eps}{}$ to emphasize the central role of the $\Schro$ bridge in its definition. That is, if $(X,Y) \sim \SBstatic{\rho}{\eps}$, then $\SB{\rho}{\eps}{}$ is the law of the random variable $2X - \Exp{\SBstatic{\rho}{\eps}{}}[Y|X]$. If $\eps \approx 0$, then $\SBstatic{\rho}{\eps}{}$ should closely approximate the quadratic cost optimal transport plan from $\rho$ to itself \cite{leo-sb-to-kp12}. That is, one would expect $\Exp{\SBstatic{\rho}{\eps}{}}[Y|X]$ to be approximately close to $X$ itself (see, for instance, \cite[Corollary 1]{pooladian2022entropic}). Therefore, one expects $2X - \Exp{\SBstatic{\rho}{\eps}{}}[Y|X] \approx X$, as well. Hence, the pushforward should not alter $\rho$ by much.

% (Move this to later!) We can associate with $\SBstatic{\rho}{\eps}$ a collection of functions $(f_{\rho,\eps},\eps > 0)$ called \textbf{entropic potentials}, which are given by
% \begin{align*}
%     \SBstatic{\rho}{\eps}(x,y) &= \frac{1}{(2\pi\eps)^{d/2}} \exp\left(\frac{1}{\eps}f_{\rho,\eps}(x)+\frac{1}{\eps}f_{\rho,\eps}(y)-\frac{1}{2\eps}\|y-x\|^{2}\right)\rho(x)\rho(y).
% \end{align*}
% There is also the following connection between the entropic potential and barycentric projection
% \begin{align*}
%     x - \nabla f_{\eps,\rho}(x) = \BPbase{\rho}{\eps}(x).
% \end{align*}

What happens if we iterate this map and rescale time with $\eps$? That is, fix some $\rho_0 \in \cP_2(\R^d)$ and a $T >0$. Let $N_\eps = \lfloor T \eps^{-1}\rfloor$. With $\rho_0^\eps:=\rho_0$, define iteratively
\[
\rho_k^\eps:=\SB{\rho_{k-1}^\eps}{\eps}{}, \quad k\in [N_\eps]:=\{1,2,\ldots, N_\eps\}.
\]
Define a continuous time piecewise constant interpolant $\left( \rho^\eps(t)= \rho^\eps_{\lfloor t/\eps \rfloor},\; t \in [0,T] \right)$. The question is whether, as $\eps\rightarrow 0+$, this curve on the Wasserstein space has an absolutely continuous limit? If so, can we write down its characterizing continuity equation?  
 
Before we give the answer, consider the simulation in Figure \ref{fig:gaussian_mix_eps0.1}. 
The figure considers a particle system approximation to the above iterative scheme. We push $n=500$ particles, initially distributed as a mixtures of two normal distributions, $\rho_0 := 0.5 \,\cN(-2, 1) + 0.5\,\cN(2, 1)$, via the dynamical system for $T=5$ and step size $\eps = 0.1$. At every step the Schr\"odinger bridge is approximated by the (matrix) Sinkhorn algorithm \cite{cuturi2013sinkhorn} applied to the current empirical distribution of the particles and the corresponding barycenter is approximated from this estimate. The pushforward is then applied to the $n$ particles themselves. This gives a time evolving process of empirical distributions that can be thought of as a discrete approximation of our iterative scheme. One can see that the histogram gets smoothed out gradually and ends in a shape that appears to be a Gaussian distribution. In fact, this process of smoothing is reminiscent of the heat flow starting from $\rho_0$. 
In the case of this simple starting distribution, the heat flow $(\rho_t, t\geq 0)$ is known analytically. The curves of the heat flow are superimposed as continuous curves. The broken curves are the exact $\rho^\eps(\cdot)$ which, in this case, can also be computed. The histograms and the two curves all closely match with one another.

%In fact, the barycentric projection for Gaussian marginals is also known analytically via \cite{janati2020}. Consequently, we also plot the true gradient flow and the piecewise constant interpolation of $(\rho_k^\eps, k\in[T\eps^{-1}])$. Note the discrete-time approximation of heat flow defined using barycentric projections closely approximates the true flow. The details on explicit calculations of $\rho_k^\eps$ are included in Section~\ref{sec:example_entropy}.

%Figure~\ref{fig:gaussian_mix_eps0.1} plots the histogram of the particles pushed by the map $x \mapsto 2x - \hat T_{\hat \rho_k^\eps}(x)$ at intervals of one unit time where $\hat T_{\hat \rho_k^\eps}$ is the empirical conditional mean of discrete Schrödinger bridge with equal marginals $\hat \rho_k^\eps$. 

\begin{figure}
    \centering
    \includegraphics[width=0.99\linewidth]{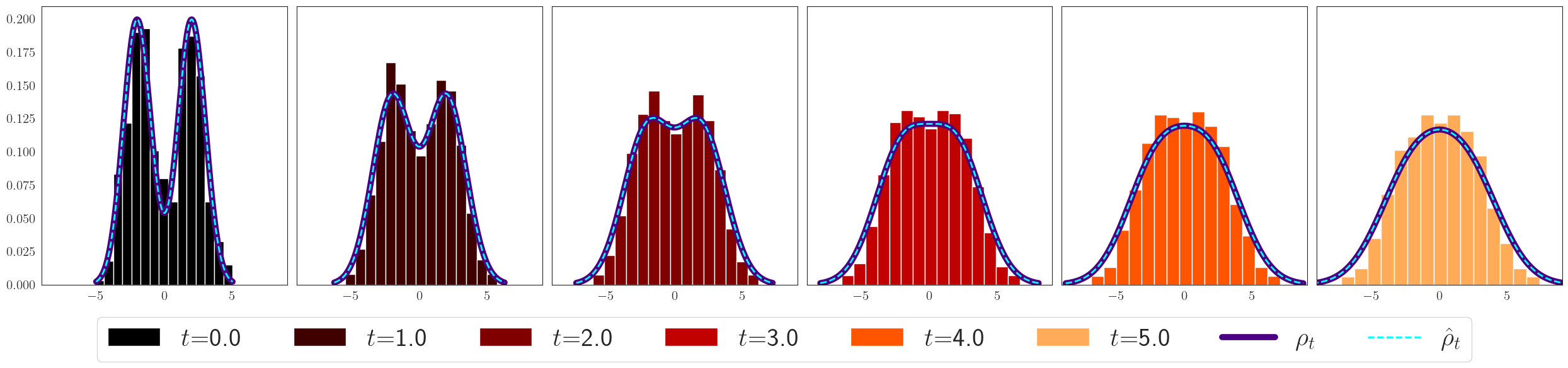}
    \caption{Histograms of $n=500$ particles at increments of one time unit. 
    Here $(\rho_t, t\in[0,T])$ is the true gradient flow and $(\hat \rho_t, t\in[0,T])$ is the piecewise constant interpolation of the analytically known scheme $(\rho_k^\eps, k\in[N_{\eps}])$.}
    \label{fig:gaussian_mix_eps0.1}
\end{figure}

 One of the results that we prove in this paper shows that, under appropriate assumptions, $\left( \rho^\eps(t),\; t \in [0,T] \right)$ indeed converges to the solution of the heat equation $\frac{\partial}{\partial t} \rho_t = \frac{1}{2}\Delta \rho_t$, with initial condition $\rho_0$. In a different context and language, this result was first observed in \cite[Theorem 1]{sander_22}. However, their proof is incomplete. See the discussion below. The statement of Theorem \ref{corollary:entropy-with-sb} in this paper lays down sufficient conditions under which this convergence happens. 

 It is well known from \cite[Theorem 5.1]{jko98} that the heat flow is the Wasserstein gradient flow of (one half) entropy. We prove a generalization of the above result where a similar discrete dynamical system converges, under assumptions, to other Wasserstein gradient flows that satisfy an evolutionary PDE of the type \cite[eqn. (11.1.6)]{ambrosio2005gradient}. The general existence and uniqueness of gradient flows (\textit{curves of maximal slopes}) in Wasserstein space are quite subtle \cite[Chapter 11]{ambrosio2005gradient}. However, evolutionary PDEs of the type \cite[eqn. (11.1.6)]{ambrosio2005gradient} include the classical case of the gradient flow of entropy as the heat equation and that of the Kullback-Leibler divergence (also called relative entropy) as the Fokker-Planck equation \cite[Theorem 5.1]{jko98}.  Our iterative approximating schemes always iterate a function of the barycentric projection of the Schr\"{o}dinger bridge with the \textit{same} marginals. The marginal itself changes with the iteration. Let us give the reader an intuitive idea how such schemes work.

%In spirit of the above iterative scheme, we propose a general explicit SB scheme for approximating AC curves. The iterative scheme considered by \cite{sander_22} is more grounded in practical applications because of computational methods and resources available to us for estimating barycentric projecting using discrete Schrödinger bridges \cite{knight2008sinkhorn}. 

 %\SP{Did \cite{sander_22} actually ``show'' this? Or is this our interpretation of what they attempted? If I understand correctly, they had particles, a more complicated expression, and a partial non-rigorous argument. This clean statement is our contribution. The comparison with \cite{sander_22} has to be carefully done in detail at the end of the Introduction. We should say that our statement here is inspired by their attempt.}

 %\GM{That is correct, the statement of \cite[Theorem 1]{sander_22} is just that the ``plug-in'' of the limit velocity field to the continuity equation gives the heat equation.}

 %\SP{The above para should be rewritten and moved back where we compare our work to \cite{sander_22}.}

The foundations of all these results can be found in Theorem \ref{thm:ld-sb-general-g} which shows that the two point joint distribution of a stationary Langevin diffusion is a close approximation of the Schrödinger bridge with equal marginals.
More concretely, let $\rho = e^{-g}$ be some element in $\cP_2(\R^d)$ for which the Langevin diffusion process with stationary distribution $\rho$ exists and is unique in law. This diffusion process is governed by the stochastic differential equation 
\begin{equation}\label{eq:langevin_sde}
    dX_t = -\frac{1}{2}\nabla g(X_t)dt + dB_t\,,\quad X_0 \sim \rho,
\end{equation}
where $(B_t, t \geq 0)$ is the standard $d$-dimensional Brownian motion. Therefore, for any $\eps>0$, $\LDstatic{\rho}{\eps} := \text{Law}(X_0,X_{\eps})$ is an element in $\Pi(\rho, \rho)$. Then, Theorem \ref{thm:ld-sb-general-g} shows that, under appropriate conditions on $\rho$, the symmetrized relative entropy, i.e., the Jensen-Shannon divergence, between $\LDstatic{\rho}{\eps}$ and the Schr\"odinger bridge $\SBstatic{\rho}{\eps}$ is $o(\eps^2)$. Although this level of precision is sufficient for our purpose, it is shown via explicit computations in \eqref{eq:gaussiancompute} that when $\rho$ is Gaussian this symmetrized relative entropy is actually $O(\eps^4)$, an order of magnitude tighter. 

In any case, one would expect that, under suitable assumptions, the barycentric projection
\begin{equation}\label{eq:barycenterapprox}
\BPbase{\rho}{\eps}(x) \approx \Exp{\LDstatic{\rho}{\eps}{}}[Y|X=x] \approx x - \frac{\eps}{2} \nabla g(x)= x + \frac{\eps}{2} \nabla \log \rho(x).  
\end{equation}
The quantity $\nabla \log \rho$ is called the score function. This is a reformulation of \cite[Theorem 1]{sander_22} which states that the difference between the two sides converges to zero in $L^2(\rho)$. 
In other words, one would expect that, under suitable assumptions,
\begin{equation}\label{eq:sbldapprox}
\SB{\rho}{\eps}{}:= \of{2\Id - \BPbase{\rho}{\eps}}_\#\rho \approx \of{\Id - \frac{\eps}{2} \nabla \log \rho(x)}_\#\rho.
\end{equation}
We claim this is approximately the solution of the heat equation at time $\epsilon$, starting from $\rho$. 

To see this, recall that absolutely continuous curves (see \cite[Definition 1.1.1]{ambrosio2005gradient}) $\left(\rho_t,\; t\ge 0 \right)$ on the Wasserstein space are characterized as weak solutions \cite[eqn.\ (8.1.4)]{ambrosio2005gradient} of continuity equations
\begin{align}\label{eq:continuity-equation}
   \frac{\partial}{\partial t} \rho_t + \nabla \cdot \of{v_t \rho_t} = 0\,.
\end{align}
Here $v_t: \R^d \to \R^d$  is the time-dependent Borel velocity field of the continuity equation whose properties are given by \cite[Theorem 8.3.1]{ambrosio2005gradient}. For any absolutely continuous curve, if $v_t$ belongs to the tangent space at $\rho_t$ (\cite[Definition 8.4.1]{ambrosio2005gradient}), \cite[Proposition 8.4.6]{ambrosio2005gradient} states the following approximation result for Lebesgue-a.e. $t \geq 0$
\begin{align}\label{eq:ags8-4-6}
    \lim\limits_{\eps \to 0} \frac{1}{\eps}\Was{2}(\rho_{t+\eps},(\Id+\eps v_t)_{\#}\rho_t) = 0. 
\end{align}
For example, if $\rho$ is the Wasserstein gradient flow of a function $\cF: \cP_2(\R^d) \to (-\infty, \infty]$ satisfying \cite[eqn. (11.1.6)]{ambrosio2005gradient}, $v_t= - \nabla \frac{\delta \mathcal{F}}{\delta \rho_t}$ where $\frac{\delta \mathcal{F}}{\delta \rho_t}$ is the first variation of $\cF$ (see \cite[Section 10.4.1]{ambrosio2005gradient}) at $\rho_t$. 
For the heat flow, which is the gradient flow of (half) entropy, $\frac{\delta \mathcal{F}}{\delta \rho_t}=\frac{1}{2}\log \rho_t$ and $v_t=-\frac{1}{2}\nabla \log \rho_t$. Comparing \eqref{eq:sbldapprox} with \eqref{eq:ags8-4-6} explains the previous claim. 

The claimed approximation in \eqref{eq:sbldapprox} is what we call \textit{one-step approximation}. In a somewhat more general setting, this is proved in Theorem \ref{thm:one_step_convergence}. 
Note that a crucial advantage we gain in the Schr\"odinger bridge scheme that is absent in Langevin diffusion or the continuity equation is that we no longer need to compute the score function $\nabla \log \rho$. Instead it suffices to be able to draw samples from $\rho$ and compute the approximate Schr\"odinger bridge, say by the Sinkhorn algorithm \cite{cuturi2013sinkhorn}, as done in Figure \ref{fig:gaussian_mix_eps0.1}.  

To summarize, the flow of ideas in the paper goes as follows. First, we show that the Schr\"odinger bridge map $\SB{\rho}{\eps}{}$, defined in its full generality in \eqref{sb-step}, is an approximation the pushforward of $\rho$ by the tangent vector of the gradient flow \eqref{eq:ags8-4-6}. This is known to be close to the gradient flow curve itself \cite[Proposition 8.4.6]{ambrosio2005gradient}. So one may expect that an iteration of this idea, where an absolutely continuous curve approximated by iterated pushforwards such as \eqref{eq:ags8-4-6} can be, under suitable assumptions, approximated by iterated Schr\"odinger bridge operators $\SB{\cdot}{\eps}{}$. One has to be mindful that the cumulative error remains controlled as we iterate, but in the best of situations this might work. 
And, in fact, it does for the heat flow, under suitable conditions, as proved in Theorem \ref{corollary:entropy-with-sb}.

Now, a natural question is how general is this method? Can we approximate other gradient flows, say that of Kullback-Leibler with respect to a log-concave density? What about absolutely continuous curves that are not gradient flows? The case of KL is partially settled in this paper in Theorem \ref{corollary:kl-with-sb}.
Let us state our proposed Schr\"{o}dinger Bridge scheme in somewhat general notation.
Consider a suitably nice functional $\mathcal{F}: \cP_{2}(\mathbb{R}^{d}) \to \mathbb{R} \cup \{+\infty\}$ and fix an initial measure $\rho_0 \in \cP_{2}(\mathbb{R}^{d})$. Suppose the gradient flow of $\cF$ can be approximated by a sequence of explicit Euler iterations $\left(\tilde{\rho}_\eps(k),\; k \in \N \right)$, where, starting from $\tilde{\rho}_\eps(0)=\rho_0$, iteratively define 
\begin{equation}\label{eq:explicit_euler_Intro}
\tilde{\rho}_\eps(k)=\left(\Id + \eps \nabla u_{k-1} \right)_{\#}\tilde{\rho}_\eps(k-1)
\end{equation}
with $-\nabla u_{k-1}$ being the Wasserstein gradient of $\cF$ at $\tilde{\rho}_\eps(k-1)$. Such an approximation scheme will be referred to as an explicit Euler iteration scheme and conditions must be imposed for it to converge to the underlying gradient flow. See Theorem \ref{thm:S_uniform_convergence} below for sufficient conditions. 

Ignoring technical details, our goal is to approximate each such explicit Euler step via a corresponding Schr\"odinger bridge step. This is accomplished by the introduction of what we call a \textit{surrogate measure}. For simplicity, consider the first step $\tilde{\rho}_\eps(1)=\left(\Id + \eps \nabla u_0 \right)_{\#}\rho_0$. 
Suppose that there is some $\theta \in \mathbb{R} \setminus \{0\}$ such that $\int e^{2\theta u_0} < \infty$, define the \textbf{surrogate measure} $\sigma_0$ to be one with the density
\begin{align}\label{eq:surrogate}
    \sigma_0(x) &:= \exp\left(2\theta u_0(x)- \Lambda_0(\theta)\right), \quad \text{where}\quad \Lambda_0(\theta):=\log \int e^{2\theta u_0(y)}dy.
\end{align}
As explained in \eqref{eq:barycenterapprox}, if we consider the Schr\"odinger bridge at temperature $\eps$ with both marginals to be $\sigma_0$, then
%\begin{align*}
    $\BPbase{\sigma_0}{\eps}(x) \approx x+ \frac{\eps}{2}\nabla \log \sigma_0(x) = x + \theta \eps\nabla u_0(x)$.
%\end{align*}
If we now define one step in the SB scheme as
\begin{align}\label{sb-step}
    \tag{SB} \SB{\rho_0}{\eps}{} &= \left(\left(1-\frac{1}{\theta}\right)\Id + \frac{1}{\theta} \BPbase{\sigma_0}{\eps}\right)_{\#}\rho_0.
\end{align}
Observing that, as $\eps \downarrow 0$, we recover that $\SB{\rho}{\eps}{} \approx (\Id + \eps \nabla u_0)_{\#} \rho$.
Notice that the choice of $\theta$ does not play a role except to guarantee that the surrogate $\sigma_0$ in \eqref{eq:surrogate} is integrable. For the case of the heat flow, $u_0=-\frac{1}{2}\log \rho_0$. Hence, if we pick $\theta=-1$, $\sigma_0=\rho_0$ and then the scheme \eqref{eq:SBheatflow} becomes a special case of \eqref{sb-step}.

Let us consider the sequence of iterates of \eqref{sb-step}, interpolated in continuous time for some $0 < T < \infty$,
\[
\rho^\eps((k+1)/\eps):= \SB{\rho^\eps(k/\eps)}{\eps}{}, \quad k \in [N_\eps], \quad \text{where}\; N_\eps:=\lfloor T/\eps \rfloor. 
\]
Theorems \ref{corollary:entropy-with-sb} and \ref{corollary:kl-with-sb} list sufficient conditions under which such a scheme converges uniformly to the gradient flow curve as $\eps \rightarrow 0+$ for specific choices of functional $\cF$. 

We should point out that the primary reason why we need the assumptions in Theorems \ref{corollary:entropy-with-sb} and \ref{corollary:kl-with-sb} is to guarantee that the errors go to zero uniformly over a growing number of iterations. Much of the technical work in this paper goes to produce a simple set of sufficient conditions to ensure this. Consequently, our theorems are specialized to the case of gradient flows of geodesically convex regular functions. But the scope of Schr\"odinger bridge scheme is potentially larger. As shown in Section \ref{sec:example_reverse_flows} via explicit calculations, the Schr\"odinger bridge scheme may be used to approximate even time-reversed gradient flows. The corresponding PDEs, such as the backward heat equation/ Fokker-Planck, are ill-posed and care must be taken to implement any approximating scheme. Interestingly, if one replaces \eqref{eq:SBheatflow} by $\SB{\rho}{\eps}{}:= \of{\BPbase{\rho}{\eps}}_\#\rho$, then starting from a Gaussian initial distribution, and for all small enough $\eps$, the iterates are shown to converge to the time-reversed gradient flow of one-half entropy. 

%Finally, since the explicit Euler scheme is not guaranteed to approximate an underlying gradient flow curve in the Wasserstein space, the preferred alternative method is the renowned JKO scheme 

In Section \ref{sec:one-step} we also generalize our Schr\"odinger bridge scheme \eqref{sb-step} by picking the surrogate measure \textit{depending on} $\eps$. This generalization is necessary to generalize the Schr\"odinger bridge scheme to cover tangent vectors that are not necessarily of the gradient type, but are approximated by a sequence of gradient fields. This generalization also covers the case of the renowned JKO scheme, \cite{jko98}, which can be likened to the implicit Euler discretization scheme of the gradient flow curve.  This, however, introduces an additional layer of complexity. 

%where the surrogate measure changes with the iteration \textit{and} the temperature parameter $\eps$. 

%However, the gain from this generalization is that errors in iterations can be controlled much better due to exisiting bounds from \cite[Chapter 4]{ambrosio2005gradient}. Hence,  Section \ref{sec:schro-bridge-scheme} deals with the explicit and the implicit schemes in a unified manner. 

Some of the technical challenges we overcome in this paper involve proving results about Schr\"odinger bridges at low temperatures when the marginals are not compactly supported. The latter is a common assumption in the literature, see for example \cite{pal2019difference,pooladian2022entropic, chiarini2022gradient,sander_22}. However, the existence of Langevin diffusion requires an unbounded support. Much of this work is done in Section \ref{sec:sb-ld} which may be of independent interest. For example, we extend the expansion of entropic OT cost function from \cite[Theorem 1.6]{conforti21deriv} for the case of same marginals in \eqref{eq:CTextension}. In all this analysis an unexpected mathematical object appears repeatedly, the \textit{harmonic characteristic} (\cite{vonrenesse-conf18}) defined below in \eqref{eq:harmoniccharacteriestic}. It appears in the computation of Fisher information of the Langevin and the Schr\"odinger bridges (see \eqref{eq:almost-bound-rxyeps}) and in the geodesic approximation of the heat flow (Lemma \ref{lem:heat-flow-sb-eps2}).
% and the Fokker-Planck flow (Lemma \ref{lem:kl-sb-eps2}). 
The role of this object merits a deeper study.

\textbf{Discussion on \cite{sander_22}.} 
The inspiring work~\cite{sander_22} is one of the motivations of our investigations. The authors of~\cite{sander_22} propose and analyze a modification to the self-attention mechanism from the Transformer artificial neural network architecture in which the self-attention matrix is enforced to be doubly-stochastic. Since the Sinkhorn algorithm is applied to enforce double stochasticity, the architecture is called the Sinkformer. A major contribution of \cite{sander_22} is the keen observation that self-attention mechanisms can be understood as determining a dynamical system whose particles are the input tokens (possibly after a suitable transformation). See also \cite{geshkovski2024emergence} for other work in this area. Specifically, \cite[Proposition 2,3]{sander_22} observes that, as a regularization parameter vanishes, the associated particle system under doubly-stochastic self-attention evolves according to the heat flow, and this property naturally follows from the doubly-stochastic constraints. 

Theorems~\ref{thm:ld-sb-general-g}-\ref{corollary:kl-with-sb} can be understood as probabilistic approaches to~\cite[Theorem 1]{sander_22}, as well as a finer treatment to key steps of the mathematical analysis of the convergence to the heat flow equation. The first statement of \cite[Theorem 1]{sander_22} concerns the convergence of the rescaled gradients of entropic potentials for the $\Schro$ bridge with the same marginals. Consider the so-called entropic (or, Schr\"odinger) potentials $(f_{\eps}, \eps > 0)$ as defined later on in \eqref{eq:SB-density}. In this language, \cite[Theorem 1]{sander_22} states that $\eps^{-1} \nabla f_{\eps} \to -\nabla \log \rho$ in $L^{2}(\rho)$, as $\eps \rightarrow 0+$. We are not able to prove this particular convergence; not least because we work in a slightly different setting, i.e., with measures fully supported on $\mathbb{R}^{d}$ whereas \cite{sander_22} assumes compact support (although, see the discussion after the proof of Theorem \ref{thm:ld-sb-general-g} for a similar result). Instead our argument relies on approximating the Schr\"odinger bridge by the Langevin diffusion which are undefined for measures with compact support. Moreover, in our paper we develop a general machinery that demonstrate the ability of the Schr\"odinger bridge scheme to approximate gradient flows of other functionals of interest beyond entropy.

\section{Approximation of Schr\"{o}dinger Bridge by Langevin Diffusion}\label{sec:sb-ld}
\subsection{Preliminaries}
In this section we fix a measure $\rho \in \cP_{2}(\mathbb{R}^{d})$ with density given by $\rho = e^{-g}$ for some function $g: \mathbb{R}^{d} \to \mathbb{R} \cup \{+\infty\}$. Here, we follow the standard convention in optimal transport of using the same notation to refer to an absolutely continuous measure and its density. We will state assumptions on $g$ later. Let $\cP_{2}^{ac}(\mathbb{R}^{d})$ denote the collection of probability measures on $\mathbb{R}^{d}$ with density and finite second moments. 
% The important fact to keep in mind is that $\rho$ is playing the role of the surrogate measure defined in (\ref{eq:surrogate}). 
The aim of this section is to prove that when the two marginals are the same, then the $\Schro$ bridge and the stationary Langevin diffusion with that marginal are very close in the sense made precise in Theorem \ref{thm:one_step_convergence}.
% As alluded to in the Introduction, the $\Schro$ bridge and Langevin diffusion are essential objects in the \eqref{sb-step} scheme. 

% Before presenting the argument for the convergence of one step in the scheme outlined in (\ref{sb-step}), we introduce notation as well as some essential results from the literature. 

We start by collating standard results about the Langevin diffusion and $\Schro$ bridge and introducing the notions from stochastic calculus of which we will make use. The relevant notation is all summarized in Table \ref{table:notation-sumary}. 

% Our proof relies the fact that, 
%\SP{Proof of what? we haven't established this fact yet.} \GM{Is this any better? Just wanting to explain why we're doing this preliminaries/notation section}
% at small time scale, the stationary Langevin diffusion with stationary measure equal to $\rho$ provides a very close approximation (in the sense of relative entropy) to the $\eps$-static $\Schro$ bridge with marginals equal to $\rho$. 

Let $C^{d}[0,\infty)$ denote the set of continuous paths from $[0,\infty)$ to $\mathbb{R}^{d}$, equipped with the locally uniform metric. Denote the coordinate process by $(\omega_t, t \geq 0)$. Endow this space with a filtration satisfying the \textit{usual conditions} \cite[Chapter 1, Definition 2.25]{karatshreve91}. For $x \in \mathbb{R}^{d}$, let $W_{x}$ denote the law on $C^{d}[0,\infty)$ of the standard $d$-dimensional Brownian motion started from $x$. 
% Given a continuous local martingale $(M_{t},t \geq 0)$, denote its quadratic variation by $(\langle M \rangle_{t}, t \geq 0)$ and let $\mathcal{E}(M)$ denote the stochastic exponential of $M$, that is $\mathcal{E}(M)_{t} = \exp\left(M_{t}-\frac{1}{2}\langle M\rangle_t\right)$ \cite[Chapter IV]{revuz2004continuous}. Given a stopping time $T$, denote the stopped martingale by $M^{T} = (M_{T \wedge t}, t \geq 0)$. 
All stochastic integrals are It\^{o} integrals.

Let $g \in C^2(\mathbb{R}^{d})$ and assume that $(Y_t, t \geq 0)$, the stationary Langevin diffusion with stationary measure equal to $\rho = e^{-g}$, exists. That is, its law on $C^{d}[0,\infty)$ is a weak solution to (\ref{eq:langevin_sde}) with initial condition $Y_0 \sim \rho$. Let $Q$ denote the law of this process on $C^{d}[0,\infty)$, and let $Q_x$ denote the law of the Langevin diffusion started from $x$. As in the Introduction, for $\eps > 0$, set $\LDstatic{\rho}{\eps} = \text{Law}(Y_0,Y_\eps)$. 
% The transition kernels for $(Y_t, t \geq 0)$ can be explicitly computed by a method outlined in \cite[Chapter VIII, Proposition 3.1]{revuz2004continuous}, which we summarized below. 
Under suitable assumptions on $g$ (see Assumption \ref{assumption:LD_SB_relative_entropy} below), $(Y_t, t \geq 0)$ exists with a.s.\ infinite explosion time \cite[Theorem 2.2.19]{royer-lsi}, and by \cite[Lemma 2.2.21]{royer-lsi} for each time $t > 0$ the laws of $Q_x$ and $P_x$ restricted to $C^{d}[0,t]$ are mutually absolutely continuous with Radon-Nikodym derivative 
\begin{align}\label{eq:rn-LD-BM}
    \left.\frac{dQ_x}{dW_x}\right|_{C^{d}[0,t]}(\omega) &= \sqrt{\frac{\rho(\omega_t)}{\rho(x)}}\exp\left(-\int_{0}^{t} \left(\frac{1}{8}\|\nabla g(\omega_s)\|^2-\frac{1}{4}\Delta g(\omega_s)\right) ds\right). 
\end{align}
The expression inside the integral of the $\exp(\cdot)$ is a distinguished quantity in the Langevin diffusion literature (see \cite{LKreener}). Denote it by 
\begin{align}\label{eq:harmoniccharacteriestic}
    \cU(x) &:= \frac{1}{8}\|\nabla g(x)\|^2 - \frac{1}{4}\Delta g(x). 
\end{align}
This function is called the harmonic characteristic in \cite[eqn.\ (2)]{vonrenesse-conf18}, which provides a heuristic physical interpretation for $\cU$ and $\nabla \cU$ (the latter of which appears in Theorem \ref{thm:ld-sb-general-g}) as the ``mean acceleration'' of the Langevin bridge. This quantity also appears in \eqref{eq:intacc} as the second derivative in time of a quantity of interest, giving another way to see it as mean acceleration.

Let $(p_t(\cdot,\cdot), t > 0)$ and $(q_t(\cdot,\cdot), t > 0)$ denote the transition kernels for the standard $d$-dimensional Brownian motion and the Langevin diffusion, respectively. By \eqref{eq:rn-LD-BM}
% By our construction of $\mathcal{E}(M)$,
% In the notation of \cite[Chapter VIII, Proposition 3.1]{revuz2004continuous} set $f = -\frac{1}{2}g$ and $F = \frac{1}{8}\|\nabla g\|^2 - \frac{1}{4}\Delta g$ to get for $t > 0$
\begin{align*}
    \frac{q_{t}(x,y)}{p_{t}(x,y)} &= \sqrt{\frac{\rho(y)}{\rho(x)}}\Exp{W_{x}}\left[\exp\left(-\int_{0}^{t} \cU(\omega_s) ds\right)\middle| \; \omega_t = y\right].
\end{align*}
Notice that $\Exp{W_x}[\cdot|\omega_t = y]$ is expectation with respect to the law of the Brownian bridge from $x$ to $y$ over $[0,t]$. 
% Due to the distinguished role that the expectation with respect to the Brownian bridge will play in our analysis, we 
Define the following function
\begin{align}\label{eq:bridge-exp}
    c(x,y,\eps) := -\log\left(\Exp{W_{x}}\left[\exp\left(-\int_{0}^{\eps}\cU(\omega_s) ds\right)\middle| \; \omega_\eps = y\right]\right). 
\end{align}
Recall that $\LDstatic{\rho}{\eps}$ is joint density of $(Y_0,Y_{\eps})$ under Langevin; it is now equal to
\begin{align}\label{eq:LD-density}
    \LDstatic{\rho}{\eps}(x,y) &= \rho(x)q_{\eps}(x,y) = \sqrt{\rho(x)\rho(y)}p_{\eps}(x,y)\exp(-c(x,y,\eps)) \\
    &= \sqrt{\rho(x)\rho(y)}\frac{1}{(2\pi\eps)^{d/2}}\exp\left(-\frac{1}{2\eps}\|x-y\|^2 -c(x,y,\eps)\right).
\end{align}

Let $(G_t, t \geq 0)$ denote the corresponding Langevin semigroup and $L$ its extended infinitesmal generator (see \cite[Chapter VII, Section 1]{revuz2004continuous} for definitions). That is, we have for $t \geq 0$ and $f: \mathbb{R}^{d} \to \mathbb{R}$ bounded and measurable that
\begin{align*}
    G_{t}f(x) := \Exp{Q}[f(\omega_t)|\omega_0 = x].
\end{align*}
Letting $C^{2}(\mathbb{R}^{d})$ denote the collection of all twice continuously differentiable functions, by \cite[Chapter VIII, Proposition 3.4]{revuz2004continuous} we have for $f \in C^{2}(\mathbb{R}^{d})$ that
\begin{align}\label{eq:LD-generator}
    Lf &= \frac{1}{2}\Delta f + \left(-\frac{1}{2}\nabla g\right) \cdot \nabla f. 
\end{align}

We now define the static $\Schro$ bridge with marginals equal to $\rho$, denoted by $\SBstatic{\rho}{\eps}$. 
% The $\SBstatic{\rho}{\eps}$ that we have defined thus far is called the ``static'' $\Schro$ bridge; there is a corresponding process called the ``dynamic'' $\Schro$ bridge for which we now delineate the stochastic machinery.
% We now similarly delineate the stochastic machinery for the $\Schro$ 
%\SP{PLease write Schrodinger everywhere as Schr\"{o}dinger. You may make a macro.} \GM{Added macros for $\Schro$ and $\Ito$} 
% bridge, which is more cumbersome owing to the fact that it is a time inhomogeneous Markov process. 
For an excellent and detailed account we refer readers to \cite{schroLeonard13} and references therein; we develop only the notions that we use within our proof. Under mild regularity assumptions on $\rho$ (for instance, finite entropy and a moment condition), it is known that for each $\eps > 0$, $\SBstatic{\rho}{\eps}$ admits an $(f,g)$-decomposition \cite[Theorem 2.8]{schroLeonard13}. As we consider the equal marginal $\Schro$ bridge, we can in fact insist that $f = g$. Thus, there exists $f_{\eps}: \mathbb{R}^{d} \to \mathbb{R}$ such that $\SBstatic{\rho}{\eps}$ admits a Lebesgue density of the form
\begin{align}\label{eq:SB-density}
    \SBstatic{\rho}{\eps}(x,y) &= \frac{1}{(2\pi\eps)^{d/2}}\exp\left(\frac{1}{\eps}\left(f_{\eps}(x)+f_{\eps}(y)-\frac{1}{2}\|x-y\|^2\right)\right)\rho(x)\rho(y). 
\end{align}
We refer to the $(f_{\eps}, \eps > 0)$ as entropic potentials. Note that, in general, entropic potentials are unique up to a choice of constant; however, our convention of symmetry forces that each $f_{\eps}$ is indeed unique. Moreover, combining (\ref{eq:base-for-SB}) and (\ref{eq:SB-density}), the $\Schro$ cost can be written in terms of the entropic potentials as 
%\SP{The notation for the final expectation is inconsistent. You should write $E_{\rho}$. Please change all expectation notation to $\mathrm{E}$.}\GM{Fixed and created new macro $\Exp{\mu}$}
\begin{align}\label{eq:schro-cost-potents}
    H(\SBstatic{\rho}{\eps}|\BMstatic{\rho}{\eps}) &= \Exp{\SBstatic{\rho}{\eps}}\left[\log \frac{d\SBstatic{\rho}{\eps}}{d\BMstatic{\rho}{\eps}}\right] = \Exp{\rho}\left[\frac{2}{\eps}f_{\eps}(X)\right]+\Ent(\rho).
\end{align}

% [Move to after discussion when it can be compared to stated results]
% An essential fact that we make use of in our argument is the following expansion of the $\Schro$ cost about $\eps = 0$ developed in \cite[Theorem 1.6]{conforti21deriv} under more stringent regularity assumptions on $\rho$
% \begin{align}\label{eq:ct-thm1-6}
%     H(\SBstatic{\rho}{\eps}|\BMstatic{\rho}{\eps}) &= \Ent(\rho) + \frac{\eps}{8}I(\rho) + o(\eps). 
% \end{align}

As outlined in \cite[Proposition 2.3]{schroLeonard13}, there is a corresponding dynamic $\Schro$ bridge, obtained from $\SBstatic{\rho}{\eps}$ (the ``static'' $\Schro$ bridge) by interpolating with suitably rescaled Brownian bridges. More precisely, let $R_{\eps}$ denote the law of $(\sqrt{\eps}B_t, t \in [0,1])$ on $C^{d}[0,1]$, where $(B_t, t \geq 0)$ is the reversible $d$-dimensional Brownian motion. Then, the law of the $\eps$-dynamic $\Schro$ bridge on $C^{d}[0,1]$ is given by $\SBdynam{\rho}{\eps}$ where
\begin{align}\label{eq:dynSB-defn}
    \frac{d\SBdynam{\rho}{\eps}{}}{dR_{\eps}}(\omega) = a^{\eps}(\omega_0)a^{\eps}(\omega_1),
\end{align}
where $a^{\eps}(x) := \exp\left(\frac{1}{\eps}f_{\eps}(x)+\log \rho(x)\right)$. Letting $(P_t, t \geq 0)$ denote the (standard) heat semigroup, we define
\begin{align}\label{eq:SB-psi}
    \psi^{\eps}(t,x) &:= \log E_{R_{\eps}}\left[a^{\eps}(\omega_1)|\omega_t = x\right] = \log E_{R_{\eps}}\left[a^{\eps}(\omega_{1-t})|\omega_0 = x\right] = \log P_{\eps(1-t)}a^{\eps}(x),
\end{align}
where the second identity follows from the Markov property of $R_{\eps}$. 

% Further, from \cite[page 2]{leonard2011stochastic}, $\SBdynam{\rho}{\eps}$ is the law of $(Y_t, t  \in [0,1])$ solving the following SDE on $C^{d}[0,1]$
% \begin{align}\label{eq:dynSB-SDE}
%     Y_{t} &= \int_{0}^{t} \eps \nabla \psi^{\eps}(s,Y_{s}) ds + \sqrt{\eps} B_{t}.
% \end{align}
% We denote the time-indexed family of extended generators of the $\eps$-dynamic $\Schro$ bridge by $(L_{t}^{\eps}, t \in [0,1], \eps > 0)$. By \cite[Equation (6)]{leonard2011stochastic} for $u \in C^2(\mathbb{R}^{d})$ the extended generator is equal to
% \begin{align}\label{eq:SB-gener}
%     L_{t}^{\eps}u(x) &:= \eps \nabla \psi^{\eps}(t,x) \cdot \nabla u(x) + \frac{\eps}{2}\Delta u(x).
% \end{align}

% Lastly, consider the curve of probability measures given by the time marginals of the dynamic $\Schro$ bridge. 
Fix $\eps > 0$, set $\rho_{s}^{\eps} := (\pi_{s})_{\#}\SBdynam{\rho}{\eps}$ and call $(\rho_{s}^{\eps}, s \in [0,1])$ the $\eps$-\textbf{entropic interpolation} from $\rho$ to itself, explicitly
\begin{align}\label{eq:entropic-inter}
    \rho_{t}^{\eps} = P_{\eps t}a^{\eps}P_{\eps(1-t)}a^{\eps} = \exp(\psi^{\eps}(t,\cdot)+\psi^{\eps}(1-t,\cdot)).
\end{align}
By \cite[Proposition 2.3]{schroLeonard13}, we construct a random variable with law $\rho_{t}^{\eps}$ as follows. Let $(X_{\eps},Y_{\eps}) \sim \SBstatic{\rho}{\eps}$ and $Z \sim N(0,\Id)$ be independent to the pair, then
\begin{align}\label{eq:rv-entropic-inter}
    X_{t}^{\eps} = (1-t)X_{\eps}+t Y_{\eps}+\sqrt{\eps t(1-t)}Z \sim \rho_{t}^{\eps}.
\end{align}
Thanks to an analogue of the celebrated Benamou-Brenier formula for entropic cost (\cite[Theorem 5.1]{gentil2017analogy}), the entropic interpolation is in fact an absolutely continuous curve in the Wasserstein space. From \cite[Equation (2.14)]{conforti21deriv} the entropic interpolation satisfies the continuity equation \eqref{eq:continuity-equation} with $(\rho_t^{\eps},v_{t}^{\eps})_{t \in [0,1]}$ given by
\begin{align}\label{eq:cont-eq-entropic-inter}
    v_{t}^{\eps} = \frac{\eps}{2}\left(\nabla \psi^{\eps}(t,\cdot)-\nabla \psi^{\eps}(1-t,\cdot)\right).
\end{align}
Moreover, under mild regularity assumptions on $\rho$, we obtain the following entropic variant of Benamou-Brenier formulation \cite[Corollary 5.8]{gentil2017analogy} (with appropriate rescaling)
% \GM{follow language of GLR-- in fact, should just cite CT here}. 
% By \cite[Theorem 5.1]{glr-analogyOT-SB} [Find where missing $\eps$ should be]
% \begin{align}\label{eq:nonsymm-ent-cost-BB}
%     H(\SBstatic{\rho}{\eps}{}|\BMstatic{\rho}{\eps}{}) &= \Ent(\rho) + \inf\limits_{(\rho_t,v_t)} \left(\frac{\eps}{2}\int_{0}^{1} \|v_t+\frac{1}{2}\nabla \log \rho_t\|_{L^{2}(\rho_t)}^2dt\right)
% \end{align}
% Similarly, \cite[Corollary 5.8]{glr-analogyOT-SB} gives
\begin{align}\label{eq:symm-ent-cost-BB}
    H(\SBstatic{\rho}{\eps}{}|\BMstatic{\rho}{\eps}{}) &= \inf\limits_{(\rho_t,v_t)} \left(\frac{1}{2\eps}\int_{0}^{1} \|v_t\|_{L^{2}(\rho_t)}^2 dt+\frac{\eps}{8}\int_{0}^{1} I(\rho_t)dt\right),
\end{align}
where $I(\cdot)$ is the Fisher information, defined for $\sigma \in \cP_{2}^{ac}(\mathbb{R}^{d})$ by $I(\sigma) = \int \|\nabla \log \sigma\|^{2} d\sigma$. The infimimum in \eqref{eq:symm-ent-cost-BB} is taken over all $(\rho_t,v_t)_{t \in [0,1]}$ satisfying $\rho_0 = \rho_1 = \rho$ and \eqref{eq:continuity-equation}. The optimal selection is given in \eqref{eq:cont-eq-entropic-inter}. That is, the optimal selection is the entropic interpolation. For an AC curve $(\rho_t, t \in [0,1])$, we will call the quantity $\int_{0}^{1} I(\rho_{t}) dt$ the integrated Fisher information. 

Lastly, for the same dynamic marginal $\Schro$ bridge, we make the following observation about the continuity at $\eps = 0$ of the integrated Fisher information along the entropic interpolation. 
% \GM{This holds in the different marginal case too by an identical argument. Should we just prove that result here? I think this would also prove the CT expansion under second moment finite entropy for $\mathbb{R}^{d}$ (relaxing their compact support assumption)}
\begin{lemma}\label{lem:integrated-fisher-continuity}
Let $\rho \in \cP_{2}(\mathbb{R}^{d})$ satisfy $-\infty <\Ent(\rho) < +\infty$ and $I(\rho) <+\infty$, then
    $\lim\limits_{\eps \downarrow 0} \int_{0}^{1} I(\rho_{t}^{\eps}) dt = I(\rho)$. 
% In fact, the convergence is at the rate $0 \leq I(\rho) \leq \int_{0}^{1} I(\rho_{t}^{\eps})dt \leq \frac{1}{6}\eps I(\rho)^2$.
\end{lemma}
\begin{proof}
% Since $\rho$ has density, 
By \cite[Theorem 3.7(3)]{leo-sb-to-kp12}, 
% as $\eps \downarrow 0$ $(\SBdynam{\rho}{\eps}, \eps > 0)$ converges weakly to the displacement interpolation from $\rho$ to itself.
% solution to the quadratic cost dynamic Monge Kantorovich problem with marginals equal to $\rho$. 
for each $s \in [0,1]$,
% the continuous mapping theorem gives weak convergence 
$\rho_s^\eps$ converges weakly to $\rho$ as $\eps \downarrow 0$. By the lower semicontinuity of Fisher information with respect to weak convergence \cite[Proposition 14.2]{bobkov-fisher-22} and Fatou's Lemma,
\begin{align*}
    I(\rho) &= \int_{0}^{1} I(\rho) dt \leq \int_0^1 \liminf\limits_{\eps \downarrow 0} I(\rho_{t}^{\eps}) dt \leq \liminf\limits_{\eps \downarrow 0} \int_0^1 I(\rho_{t}^{\eps}) dt.
\end{align*}
Lastly, by \eqref{eq:symm-ent-cost-BB} and the submoptimality of the constant curve $\rho_t = \rho$, $v_t = 0$ for $t \in [0,1]$
\begin{align*}
    \frac{1}{2\eps}\int_{0}^{1} \|v_{t}^{\eps}\|_{L^{2}(\rho_t^{\eps})}^2 dt + \frac{\eps}{8}\int_{0}^{1}I(\rho_{t}^{\eps})dt \leq \frac{\eps}{8}I(\rho) \Rightarrow \int_{0}^{1} I(\rho_{t}^{\eps})dt \leq I(\rho).
\end{align*}
Altogether, it follows
\begin{align*}
    0 \leq \limsup\limits_{\eps \downarrow 0} \left( I(\rho)- \int_0^1 I(\rho_{t}^{\eps}) dt\right) &= I(\rho) - \liminf\limits_{\eps \downarrow 0} \int_0^1 I(\rho_{t}^{\eps}) dt \leq I(\rho) - I(\rho) = 0.
\end{align*}
% To compute the rate, fix $\eps > 0$ and let $(X_{\eps},Y_{\eps}) \sim \SBstatic{\rho}{\eps}$. Let $Z \sim N(0,1)$ be independent to $(X_{\eps},Y_{\eps})$ and for $t \in [0,1]$ set
% \begin{align*}
%     X_{t}^{\eps} &= tX_{\eps} + (1-t)Y_{\eps} + \sqrt{\eps t(1-t)} Z.
% \end{align*}
% Then $\text{Law}(X_{t}^{\eps}) = \rho_{t}^{\eps}$. By independence of $Z$, apply the Stam inequality \cite[(15.7)]{bobkov-fisher-22} to obtain
% \begin{align*}
%     \frac{1}{I(X_{t}^{\eps})} &\geq \frac{1}{I(tX_{\eps}+(1-t)Y_{\eps})}+\frac{1}{\sqrt{\eps t(1-t)}Z} =  \frac{1}{I(tX_{\eps}+(1-t)Y_{\eps})}+\eps t(1-t)
% \end{align*}
% By the convexity of Fisher information \cite[(15.1)]{bobkov-fisher-22}, $I(tX_{\eps}+(1-t)Y_{\eps}) \leq I(\rho)$. Moreover, adding an independent Gaussian decreases Fisher information \cite[Corollary 15.3]{bobkov-fisher-22}, which combines with the previous observation to imply $I(X_{t}^{\eps}) \leq I(\rho)$ for all $t \in [0,1]$. Hence, the above inequality implies
% \begin{align*}
%     \frac{1}{I(\rho_{t}^{\eps})} &\geq \frac{1}{I(\rho)}+\eps t(1-t).
% \end{align*}
% Rearranging the expression gives
% \begin{align*}
%     I(\rho_{t}^{\eps}) &\leq \frac{1}{1+\eps t(1-t)I(\rho_{t}^{\eps})}I(\rho).
% \end{align*}
\end{proof}

\subsection{Langevin Approximation to $\Schro$ Bridge}
Our main result of the section is the relative entropy bound between $\LDstatic{\rho}{\eps}$ and $\SBstatic{\rho}{\eps}$ given in Theorem \ref{thm:ld-sb-general-g}.

\begin{assumption}\label{assumption:LD_SB_relative_entropy}
Let $\rho = e^{-g}$ for $g: \mathbb{R}^{d} \to \mathbb{R}$, and recall that $\cU = \frac{1}{8}\|\nabla g\|^2-\frac{1}{4}\Delta g$. Assume that
\begin{itemize}
    \item[(1)] $\rho$ is subexponential (\cite[Definition 2.7.5]{vershynin-hdp}), $g \in C^{3}(\mathbb{R}^{d})$, $\cU$ is bounded below, $g \to \infty$ as $\|x\| \to \infty$.
    \item[(2)] There exists $C > 0$ and $n \geq 1$ such that $\|\nabla \cU(x)\|^2 \leq C(1+\|x\|^n)$.
%\SP{What is the role of the cited theorem?}\GM{fixed}
    
    \item[(3)] $-\infty <\Ent(\rho) < +\infty$ and $I(\rho) < +\infty$.
\end{itemize}
\end{assumption}
Note that Assumption \ref{assumption:LD_SB_relative_entropy} (3) ensures the existence $\SBstatic{\rho}{\eps}$ and that \eqref{eq:symm-ent-cost-BB} holds \cite[(Exi),(Reg1),(Reg2)]{gentil2017analogy}.

These assumptions are, for instance, satisfied by the multivariate Gaussian $N(\mu,\Sigma)$ with $\mu \in \mathbb{R}^{d}$, $\Sigma \in \mathbb{R}^{d \times d}$ positive definite. This provides an essential class of examples as explicit computations of the $\Schro$ bridge between Gaussians are known \cite{janati2020}. In this instance, $g(x) = \frac{1}{2}(x-\mu)^{T}\Sigma^{-1}(x-\mu)+\frac{1}{2}\log((2\pi)^d \det \Sigma)$ and thus
\begin{align*}
    \cU(x) = \frac{1}{8}(x-\mu)^{T}\Sigma^{-2}(x-\mu)-\frac{1}{4}\Tr(\Sigma^{-1}) &\text{ and } \nabla \cU(x) = \frac{1}{4}\Sigma^{-2}(x-\mu). 
\end{align*}
Importantly, for $X \sim N(\mu,\Sigma)$ this establishes that $\nabla \cU(X) \sim N(0,\frac{1}{16}\Sigma^{-3})$ and thus $\Exp{}\|\nabla \cU(X)\|^{2} = \frac{1}{16} \Tr(\Sigma^{-3})$. More generally, when $g$ is polynomial such that $\int_{\mathbb{R}^{d}} e^{-g} < +\infty$ these assumptions are also satisfied. 
% is $ = \sum_{j=0}^{n} a_j \abs{x}^{j}$, $n \geq 2$ \GM{this won't work when $n=1$ as not $C^3$}, with $a_j \in \mathbb{R}$ and $a_n > 0$ 

Additionally, we remark that the class of functions $g$ satisfying Assumption \ref{assumption:LD_SB_relative_entropy} is stable under the addition of functions $h \in C^{3}(\mathbb{R}^{d})$ such that $h$ and all of its derivatives up to and including order $3$ are bounded (e.g.\ $h \in C_{c}^{3}(\mathbb{R}^{d})$). It is also stable under the addition of linear and quadratic functions, so long as the resulting function has a finite integral when exponentiated. 
%\SP{Don't the linear and quadratic parts keep Gaussians invariant?}\GM{yes-- I was trying to say that if we just have $g$ satisfying the assumptions, then adding a linear or quadratic function wouldn't cause a problem}

% \GM{totally good to move this wherever it makes sense-- putting it here initially because we need the integrability to use DCT in Theorem }
We begin with a preparatory lemma establishing some important integrability poperties we make use of in the proof of Theorem \ref{thm:ld-sb-general-g}.
\begin{lemma}\label{lem:integ-entropic-inter}
Let $\rho = e^{-g} \in \cP_{2}(\mathbb{R}^{d})$ satisfying Assumption \ref{assumption:LD_SB_relative_entropy}.
\begin{itemize}
    \item[(1)] For each $\eps > 0$, the densities along the entropic interpolation $(\rho_{s}^{\eps}, s \in [0,1])$ is uniformly subexponential, meaning that for each $T > 0$ there is a constant $K > 0$ such that for $X_{s}^{\eps} \sim \rho_{s}^{\eps}$ we have $\sup\limits_{\eps \in (0,T)}\sup\limits_{s \in [0,1]} \|X_{s}^{\eps}\|_{\psi_{1}} \leq K$, where $\|.\|_{\psi_{1}}$ is the subexponential norm defined in \cite[Definition 2.7.5]{vershynin-hdp}. %\SP{Uniformly subexponential.}
    %\GM{added}
    \item[(2)] The following integrability conditions hold:
    \begin{itemize}
        \item[(a)] The function $\cU = \frac{1}{8}\|\nabla g\|^{2}-\frac{1}{4}\Delta g$ is in $L^{2}(\rho_{s}^{\eps})$ for all $s \in [0,1]$.
        \item[(b)] The function from $[0,1] \times C^{d}[0,1]$ to $\mathbb{R}$ given by $(t,\omega) \mapsto \cU(\omega_t)$ is in $L^{2}(dt \otimes d\SBdynam{\rho}{\eps})$. Similarly, $(t,\omega) \mapsto \nabla \cU(\omega_t)$ is in $L^{2}(dt \otimes d\SBdynam{\rho}{\eps})$.
        \item[(c)] For any $T > 0$, $\sup\limits_{\eps \in (0,T)} \int_{0}^{1} \Exp{\SBdynam{\rho}{\eps}}[\|\nabla \cU(\omega_s)\|^2]ds < +\infty$.
    \end{itemize}
    % The function $\cU \in L^{2}(\rho_{s}^{\eps})$ for all $s \in [0,1]$ and $(t,\omega) \mapsto \cU(\omega_t)$ is in $L^{2}(dt \otimes d\SBdynam{\rho}{\eps})$. Moreover, $(t,\omega) \mapsto \nabla \cU(\omega_t)$ is in $L^{2}(dt \otimes d\SBdynam{\rho}{\eps})$, and for any $T > 0$, $\sup\limits_{\eps \in (0,T)} \int_{0}^{1} \Exp{\SBdynam{\rho}{\eps}}[\|\nabla \cU(\omega_s)\|^2]ds < +\infty$.
\end{itemize}
\end{lemma}
\begin{proof}
Fix $\eps > 0$ and $s \in [0,1]$. Let $(X,Y)$ be random variables with joint law $\SBstatic{\rho}{\eps}$. Let $Z \sim N(0,\Id)$ be independent of $(X,Y)$. Then $X_{s}^{\eps}$ as given in \eqref{eq:rv-entropic-inter} has law $\rho_{s}^{\eps}$.
% Define $X^{\eps}_{s}$ by 
% \begin{align*}
%     X^{\eps}_{s} &= (1-s)X + sY + \sqrt{\eps s(1-s)}Z,
% \end{align*}
% by \cite[Proposition 2.3]{schroLeonard13} the law of $X^{\eps}_{s}$ has density $\rho_{s}^{\eps}$. 
As $\|\cdot\|_{\psi_{1}}$ is a norm, observe that for any $T > 0$
\begin{align*}
    \sup\limits_{\eps \in (0,T], s \in [0,1]}\|X_{s}^{\eps}\|_{\psi_{1}} &\leq \sup\limits_{\eps \in (0,T], s \in [0,1]} \left((1-s)\|X\|_{\psi_{1}}+s\|Y\|_{\psi_{1}}+\sqrt{\eps s(1-s)}\|Z\|_{\psi_{1}}\right) < +\infty 
    %\leq \|X\|_{\psi_{1}}+\sqrt{\eps}\|Z\|_{\psi_{1}},
\end{align*}
establishing the uniform upper bound on the subexponential norm. 
% As $X,Y,Z$ are all subexponential, $X^{\eps}_{s}$ is subexponential as well. 

% Let $M = \sup\limits_{x \in B(0,R)} \|\nabla \cU(x)\|^2$, then
It then follows that
\begin{align*}
    \sup\limits_{\eps \in (0,T], s \in [0,1]}\Exp{}[\|\nabla \cU(X^{\eps}_{s})\|^2] &\leq \sup\limits_{\eps \in (0,T], s \in [0,1]}C(1+\Exp{}[\|X_{s}^{\eps}\|^n]) <+\infty. 
    % &\leq C + C'\left(s^{n}\Exp{}\|X\|^n + (1-s)^{n}\Exp{}\|Y\|^{n} + (\eps s(1-s))^{n/2}\Exp{}[\|Z\|^n]\right) \\
    % &\leq C + C'\left(2 \Exp{}\|X\|^n+ \eps^{n/2} \Exp{}\|Z\|^n\right),
\end{align*}
% which has a uniform upper bound once we consider $\eps \in (0,T]$ as $X,Y \sim \rho$ and $Z$ are subexponential. 
This observation also establishes that $(t,\omega) \mapsto \nabla \cU(\omega_t)$ is in $L^{2}(dt \otimes d\SBdynam{\rho}{\eps})$. 
% For the remaining claim, again fix $s \in [0,1]$ and $\eps > 0$.
By assumption on $\nabla \cU$, $\cU$ is also of uniform polynomial growth, and thus the same argument establishes the remaining claims. 
% Observe for $y \in \mathbb{R}^{d}$ that
% \begin{align*}
%     \abs{\cU(y)} &\leq \abs{\cU(0)} + \int_{0}^{1} \abs{\langle \nabla \cU(ty), y \rangle} dt 
%     % leq \abs{\cU(0)} + \int_{0}^{1} \frac{1}{2}\|\cU(ty)\|^2 dt + \frac{1}{2}\|y\|^2
%     \leq \abs{\cU(0)}+\frac{1}{2}\left(M+\frac{C}{n+1}\|y\|^n\right)+\frac{1}{2}\|y\|^2. 
% \end{align*}
% As each $\rho^{\eps}_{s}$ is subexponential, squaring both sides and taking expectation completes the lemma. Moreover, that $(t,\omega) \mapsto \cU(\omega_t)$ is in $L^{2}(dt \otimes \SBdynam{\rho}{\eps})$ follows from this observation along with the uniform bound on the subexponential norm previously developed.
\end{proof}

\begin{theorem}\label{thm:ld-sb-general-g}
    For $\rho \in \cP_2(\R^d)$ satisfying Assumption \ref{assumption:LD_SB_relative_entropy}(1,3), for all $\eps > 0$
    \begin{align}\label{eq:rel-ent-ld-sb}
        H(\LDstatic{\rho}{\eps} | \SBstatic{\rho}{\eps})+H(\SBstatic{\rho}{\eps} | \LDstatic{\rho}{\eps}) \leq \frac{1}{2}\eps^2 \left(I(\rho) - \int_{0}^{1}I(\rho_t^\eps) dt\right)^{1/2}\left(\int_{0}^{1} \Exp{\rho_t^\eps}\|\nabla \cU\|^2 dt\right)^{1/2}.
    \end{align}
    % Under the additional assumption that either 
    % \begin{itemize}
    %     \item[(A)] 
    %     $\limsup\limits_{\eps \downarrow 0} \int_{0}^{1} \Exp{\rho_{t}^{\eps}}[\|\nabla \cU\|^2]dt < +\infty$, or,
    %     \item[(B)] $\cU \in C^{1}$ and $\|\nabla \cU(x)\|^2 \leq a\|x\|^2 + b$,
    % \end{itemize}
    % we have
    In particular, additionally under Assumption \ref{assumption:LD_SB_relative_entropy}(2), 
    % \ref{lem:integrated-fisher-continuity} we have that
    \[
    \lim\limits_{\eps \downarrow 0} \frac{1}{\eps^2}\left(H(\LDstatic{\rho}{\eps} | \SBstatic{\rho}{\eps})+H(\SBstatic{\rho}{\eps} | \LDstatic{\rho}{\eps})\right) = 0.
    \]
\end{theorem}
Before proving the theorem, we pause to comment on the sharpness of the inequality presented in \eqref{eq:rel-ent-ld-sb}. As the $\Schro$ bridge is explicitly known for Gaussian marginals \cite{janati2020}, consider the case $\rho = N(0,1)$. Define two matrices
\begin{align*}
    \Sigma_1 =  \begin{pmatrix}
        1 & e^{-\eps/2} \\ e^{-\eps/2} & 1
    \end{pmatrix}, \Sigma_2 = \begin{pmatrix}
        1 & \frac{1}{2}(\sqrt{\eps^2+4}-\eps) \\ \frac{1}{2}(\sqrt{\eps^2+4}-\eps) & 1 
    \end{pmatrix}.
\end{align*}
It is well known that $\LDstatic{\rho}{\eps} \sim N(0,\Sigma_1)$ is the Ornstein-Uhlenbeck transition density. By \cite{janati2020}, it is known that $\SBstatic{\rho}{\eps} \sim N(0,\Sigma_2)$. One then computes
\begin{align}\label{eq:gaussiancompute}
    H(\LDstatic{\rho}{\eps}|\SBstatic{\rho}{\eps}) + H(\SBstatic{\rho}{\eps}|\LDstatic{\rho}{\eps}) = \frac{1}{2}\Tr(\Sigma_{1}^{-1}\Sigma_{2})+\frac{1}{2}\Tr(\Sigma_{2}^{-1}\Sigma_{1}) - 2= \frac{1}{1152}\eps^{4}+\cO(\eps^{5}). 
\end{align}
% can compute $H(\LDstatic{\rho}{\eps}|\SBstatic{\rho}{\eps}) = \frac{1}{2304}\eps^4+O(\eps^5)$. 
On the other hand, the entropic interpolation between Gaussians is computed in \cite{gentil2017analogy}, and $I(\rho_{t}^{\eps})$ is minimized at $t = 1/2$. From a computation in Section \ref{sec:examples}, see \eqref{eq:entropic_interpolation},
\begin{align*}
    I(\rho) - \int_{0}^{1}I(\rho_t^\eps) dt \leq I(\rho) - I(\rho_{1/2}^{\eps}) = \cO(\eps^2). 
\end{align*}
Altogether, the RHS of \eqref{eq:rel-ent-ld-sb} is then $\cO(\eps^3)$. That is, the transition densities for the OU process provide a tighter approximation (an order of magnitude) for same marginal $\Schro$ bridge than given by Theorem \ref{thm:ld-sb-general-g}.

\begin{proof}[Proof of Theorem \ref{thm:ld-sb-general-g}]
We argue in the following manner. First, using the densities for $\SBstatic{\rho}{\eps}$ and $\LDstatic{\rho}{\eps}$ given in (\ref{eq:LD-density}) and (\ref{eq:SB-density}), respectively, we compute their likelihood ratio and obtain exact expressions for the two relative entropy terms appearing in the statement of the Lemma. The resulting expression is the difference in expectation of the expression $c(x,y,\eps)$ defined in (\ref{eq:bridge-exp}) with respect to the Schr\"{o}dinger bridge and Langevin diffusion. We then extract out the leading order terms from both expectations and show that they cancel. To obtain the stated rate of decay, we then analyze the remaining term with an analytic argument from the continuity equation. 

\noindent\textbf{Step 1.} The following holds    
    \begin{align}\label{eq:sym-ent-identity}
        H(\LDstatic{\rho}{\eps}{} | \SBstatic{\rho}{\eps})+H(\SBstatic{\rho}{\eps} | \LDstatic{\rho}{\eps}) &= \Exp{\SBstatic{\rho}{\eps}{}}[c(X,Y,\eps)]-\Exp{\LDstatic{\rho}{\eps}{}}[c(X,Y,\eps)].
    \end{align}
To prove \eqref{eq:sym-ent-identity}, from (\ref{eq:LD-density}) and (\ref{eq:SB-density})
    \begin{align*}
        \frac{d\LDstatic{\rho}{\eps}{}}{d\SBstatic{\rho}{\eps}{}} = \frac{1}{\sqrt{\rho(x)\rho(y)}}\exp\left(-\frac{1}{\eps}f_{\eps}(x)-\frac{1}{\eps}f_{\eps}(y)-c(x,y,\eps)\right).
    \end{align*}
As both $\SBstatic{\rho}{\eps}{}, \LDstatic{\rho}{\eps}{} \in \Pi(\rho,\rho)$, we have that
\begin{align*}
    H(\LDstatic{\rho}{\eps}{} | \SBstatic{\rho}{\eps}) &= \Exp{\LDstatic{\rho}{\eps}{}}\left[\log \frac{d\LDstatic{\rho}{\eps}{}}{d\SBstatic{\rho}{\eps}{}}\right] = - \frac{2}{\eps}\Exp{\rho}[f_{\eps}(X)] -\Ent(\rho) - \Exp{\LDstatic{\rho}{\eps}{}}[c(X,Y,\eps)], \\
    H(\SBstatic{\rho}{\eps}|\LDstatic{\rho}{\eps}{}) &= \Exp{\SBstatic{\rho}{\eps}}\left[-\log \frac{d\LDstatic{\rho}{\eps}{}}{d\SBstatic{\rho}{\eps}{}}\right] = \frac{2}{\eps}\Exp{\rho}[f_{\eps}(X)] + \Ent(\rho) + \Exp{\SBstatic{\rho}{\eps}{}}[c(X,Y,\eps)].
\end{align*}
Equation \eqref{eq:sym-ent-identity} follows by adding these expressions together. 

\noindent\textbf{Step 2.} Recall $R_{\eps}$ as used in \eqref{eq:dynSB-defn} and let
\begin{align}\label{eq:SB-LD-remainder}
    R(x,y,\eps) &:= -\log\left(\Exp{R_{\eps}}\left[\exp\left(-\int_{0}^{1} \eps \left(\cU(\omega_s)-\cU(\omega_0)\right)ds \middle| \omega_0 = x, \omega_1 = y\right)\right]\right).
\end{align}
Then we claim that 
\begin{align}\label{eq:bound-remainder}
    H(\LDstatic{\rho}{\eps}{} | \SBstatic{\rho}{\eps})+H(\SBstatic{\rho}{\eps} | \LDstatic{\rho}{\eps}) &\leq \Exp{\SBstatic{\rho}{\eps}{}}[R(X,Y,\eps)].
\end{align}

% To prove \eqref{eq:bound-remainder}, we claim that
% \begin{align}\label{eq:ld-c-upperbdd}
%     \Exp{\LDstatic{\rho}{\eps}{}}[-c(X,Y,\eps)] \leq \frac{\eps}{8}I(\rho) 
% \end{align}
Recall that $W_{x}$ is Wiener measure starting from $x$, and
\begin{align*}
    c(x,y,\eps) &= - \log\left(\Exp{W_{x}}\left[\exp\left(-\int_{0}^{\eps} \cU(\omega_s)ds\right) \middle|  \; \omega_\eps = y \right]\right) \\
        &= \eps \cU(x) - \log\left(\Exp{W_{x}}\left[\exp\left(-\int_{0}^{\eps} \left(\cU(\omega_s)-\cU(\omega_0)\right)ds\right)\middle|\; \omega_{\eps} = y\right]\right).
\end{align*}
For the integral inside the $\exp\left(\cdot\right)$, we rescale $s$ from $[0,\eps]$ to $[0,1]$. By Brownian scaling the law of $(\omega_{\eps s}, s \geq 0)$ under $W_{x}$ is equal to the law of $(\omega_s, s \geq 0)$ under $R_{\eps}$ with initial condition $x$, which gives
\begin{align*}
    c(x,y,\eps) = \eps \cU(x) + R(x,y,\eps).
\end{align*}
As $\SBstatic{\rho}{\eps}{} \in \Pi(\rho,\rho)$, thanks to an integration by parts, 
\begin{align*}
    \Exp{\SBstatic{\rho}{\eps}{}}[\cU(X)] &= \int_{\R^d} \cU(x)\rho(x)dx= \frac{1}{8}I(\rho)-\int_{\R^d} \frac{1}{4} \Delta g(x) e^{-g(x)}dx \\
    &= \frac{1}{8}I(\rho) - \int_{\R^d} \frac{1}{4} \|\nabla g(x)\|^{2} e^{-g(x)}dx  = -\frac{1}{8}I(\rho). 
\end{align*}
Thus from \eqref{eq:sym-ent-identity} observe
\begin{align*}
    \Exp{\SBstatic{\rho}{\eps}{}}[c(X,Y,\eps)]-\Exp{\LDstatic{\rho}{\eps}{}}[c(X,Y,\eps)] &\leq -\frac{\eps}{8}I(\rho) + \Exp{\SBstatic{\rho}{\eps}}[R(X,Y,\eps)]-\Exp{\LDstatic{\rho}{\eps}{}}[c(X,Y,\eps)].
    % \frac{\eps}{8}I(\rho) =\Exp{\SBstatic{\rho}{\eps}}[R(X,Y,\eps)].
\end{align*}
Thus, to prove \eqref{eq:bound-remainder} it will suffice to show
\begin{align}\label{eq:ld-c-upperbdd}
    \Exp{\LDstatic{\rho}{\eps}{}}[-c(X,Y,\eps)] \leq \frac{\eps}{8}I(\rho).
\end{align}
% Finally, to prove the claim in \eqref{eq:ld-c-upperbdd}, 
Observe that
\begin{align*}
    \frac{d\LDstatic{\rho}{\eps}{}}{d\BMstatic{\rho}{\eps}{}} &= \frac{\sqrt{\rho(y)}}{\sqrt{\rho(x)}}\exp(-c(x,y,\eps)).
\end{align*}
Thus, as $\LDstatic{\rho}{\eps} \in \Pi(\rho,\rho)$ and $-\infty < \Ent(\rho) < +\infty$, 
\begin{align}\label{eq:rel-ent-ld-bm}
    H(\LDstatic{\rho}{\eps}{}|\BMstatic{\rho}{\eps}{}) &= \Exp{\LDstatic{\rho}{\eps}{}}\left[\log \frac{d\LDstatic{\rho}{\eps}{}}{d\BMstatic{\rho}{\eps}{}}\right] =\Exp{\LDstatic{\rho}{\eps}{}}\left[-c(X,Y,\eps)\right]. 
\end{align}
On $C^{d}[0,\eps]$ let $W_{\rho}$ denote the law of Brownian motion with initial distribution $\rho$. 
% Let $Q_x$ denote the law of the solution to \eqref{eq:langevin_sde} with initial condition $x$. 
By a well-known formula (see \cite[Equation (5.1)]{DGW} for instance)
\begin{align*}
    H(Q_x|W_x) &= \frac{1}{2} \Exp{Q_x}\left[\int_{0}^{\eps} \|-\frac{1}{2}\nabla g(\omega_s)\|^2 ds\right] = \frac{1}{8}\int_0^{\eps} \Exp{Q_x}[\|\nabla g(\omega_s)\|^2]ds. 
\end{align*}
It then follows from the factorization of relative entropy that
\begin{align*}
    H(Q|W_{\rho}) &= \int H(Q_x|W_x)\rho(x)dx = \frac{1}{8}\int \left(\int_0^{\eps} \Exp{Q_x}\left[\|\nabla g(\omega_s)\|^2\right]ds \right)\rho(x)dx \\
    &= \frac{1}{8}\int_{0}^{\eps}E_{Q}\left[\|\nabla g(\omega_s)\|^2\right]ds = \frac{\eps}{8}I(\rho). 
\end{align*}
As $\LDstatic{\rho}{\eps} = (\pi_0,\pi_{\eps})_{\#}Q$, \eqref{eq:ld-c-upperbdd} follows by \eqref{eq:rel-ent-ld-bm} and the information processing inequality (see \cite[Lemma 9.4.5]{ambrosio2005gradient} for instance).
% \begin{align*}
%     \Exp{\LDstatic{\rho}{\eps}{}}\left[-c(x,y,\eps)\right] &= H(\LDstatic{\rho}{\eps}{}|\BMstatic{\rho}{\eps}{}) \leq H(Q|W_{\rho}) = \frac{\eps}{8}I(\rho). 
% \end{align*}

\noindent\textbf{Step 3.} Show $\Exp{\SBstatic{\rho}{\eps}{}}[R(X,Y,\eps)]$ is bounded above by the RHS in \eqref{eq:rel-ent-ld-sb}. 
To start, apply Jensen's inequality to interchange the $-\log$ and $\Exp{R_{\eps}}[\; \cdot \; |\;\omega_0 = x, \omega_1= y]$ in the expression for $R(x,y,\eps)$ in \eqref{eq:SB-LD-remainder} to obtain
\begin{align*}
    R(x,y,\eps) &\leq \eps\Exp{R_{\eps}}\left[\int_{0}^{1} \left(\cU(\omega_s)-\cU(\omega_0)\right)ds \middle| \; \omega_0 = x,\omega_1 = y\right].
\end{align*}
Since the dynamic $\Schro$ bridge, $\SBdynam{\rho}{\eps}$ defined in \eqref{eq:dynSB-defn}, disintegrates as the Brownian bridge once its endpoints are fixed \cite[Proposition 2.3]{schroLeonard13}, we have that
\begin{align}\label{eq:rxyeps-bound}
    \Exp{\SBstatic{\rho}{\eps}{}}[R(X,Y,\eps)] &\leq \eps \Exp{\SBstatic{\rho}{\eps}{}}\left[\Exp{R_{\eps}}\left[\int_{0}^{1} \left(\cU(\omega_s)-\cU(\omega_0)\right)ds \middle| \; \omega_0 = X,\omega_1 = Y\right]\right]\\
    &= \eps \int_{0}^{1} \Exp{\SBdynam{\rho}{\eps}}\left[\cU(\omega_s)-\cU(\omega_0)\right]ds,
\end{align}
where the last equality follows from Fubini's Theorem. 

Recall that $(\rho_{t}^{\eps},v_{t}^{\eps})_{t \in [0,1]}$ is a weak solution for the continuity equation \eqref{eq:continuity-equation} if for all $\psi \in C_{c}^{1}(\mathbb{R}^{d})$ and $t \in [0,1]$
\begin{align*}
    \frac{d}{dt} \Exp{\rho_{t}^{\eps}}[\psi] &= \Exp{\rho_{t}^{\eps}}[\langle v_{t}^{\eps}, \nabla \psi \rangle].
\end{align*}
Integrating in time over $[0,s]$ for $s \in (0,1]$ gives
\begin{align*}
    \Exp{\rho_{s}^{\eps}}[\psi] - \Exp{\rho_{0}^{\eps}}[\psi] &= \int_{0}^{s} \Exp{\rho_{u}^{\eps}}[\langle v_{u}^{\eps}, \nabla \psi \rangle] du.
\end{align*}
We now use a density argument to extend the above identity to
\begin{align}\label{eq:cont-eqn-for-U}
    \Exp{\SBdynam{\rho}{\eps}}[\cU(\omega_s)-\cU(\omega_0)] = \Exp{\rho_{s}^{\eps}}[\cU] - \Exp{\rho_{0}^{\eps}}[\cU] &= \int_{0}^{s} \Exp{\rho_{u}^{\eps}}[\langle v_{u}^{\eps}, \nabla \cU \rangle]du.
\end{align}
For $R > 0$ let $\chi_R: \mathbb{R}^{d} \to \mathbb{R}$ be such that $\chi_R \in C_{c}^{\infty}(\mathbb{R}^{d})$, $0 \leq \chi_R \leq 1$, $\left.\chi_{R}\right|_{B(0,R)} \equiv 1$, and $\|\nabla \chi_R\|_{\infty} \leq 2$. As $\chi_{R} \cU \in C_{c}^{1}(\mathbb{R}^{d})$,
\begin{align*}
    \Exp{\SBdynam{\rho}{\eps}}[(\chi_{R}\cU)(\omega_s)-(\chi_{R}\cU)(\omega_0)] &= \int_{0}^{s} \Exp{\SBdynam{\rho}{\eps}}[\langle v_{u}^{\eps}(\omega_u), \nabla (\chi_{R}\cU)(\omega_u) \rangle]du \\
    &= \int_{0}^{s} \Exp{\SBdynam{\rho}{\eps}}\left[\langle v_u^{\eps},\nabla \cU \rangle \chi_{R}\right]ds + \int_{0}^{s} \Exp{\SBdynam{\rho}{\eps}}\left[\langle v_u^{\eps}, \nabla \chi_{R}\rangle \cU \right]ds.
\end{align*}
By Lemma \ref{lem:integ-entropic-inter}, $\int_{0}^{1} \Exp{\SBdynam{\rho}{\eps}}[\|\nabla \cU(\omega_s)\|^2]ds < +\infty$. Moreover, it is shown in \eqref{eq:integrated-velocity-bound} that $\int_{0}^{1} \Exp{\SBdynam{\rho}{\eps}}[\|v_s^{\eps}(\omega_s)\|^2] ds < +\infty$. Thus, by the Dominated Convegence Theorem
\begin{align*}
    \lim\limits_{R \uparrow +\infty} \int_{0}^{s} \Exp{\SBdynam{\rho}{\eps}}\left[\langle v_u^{\eps},\nabla \cU \rangle \chi_{R}\right]ds &= \int_{0}^{s} \Exp{\SBdynam{\rho}{\eps}}\left[\langle v_u^{\eps},\nabla \cU \rangle \right]ds.
\end{align*}
As $\abs{\chi_{R}\cU(\omega_s)+\chi_{R}\cU(\omega_0)} \leq \abs{\cU(\omega_s)}+\abs{\cU(\omega_0)}$ and 
\begin{align*}
    \abs{\langle v_u^{\eps}, \nabla \chi_{R}\rangle \cU} \leq 2\|v_{u}^{\eps}\|\abs{\cU} \leq \|v_{u}^{\eps}(\omega_u)\|^2+\abs{\cU(\omega_u)}^2
\end{align*}
% $\abs{\langle v_u^{\eps}, \nabla \chi_{R}\rangle \cU} \leq \|v_{u}^{\eps}(\omega_u)\|^2+\abs{\cU(\omega_u)}^2$
and both upper bounds are integrable in their suitable spaces (again by Lemma \ref{lem:integ-entropic-inter} and \eqref{eq:integrated-velocity-bound}), \eqref{eq:cont-eqn-for-U} follows from dominated convergence by sending $R \uparrow +\infty$. 

% By [ADD DENSITY ARGUMENT] $F$ is a $L^2(\rho_t^{\eps})$ limit of some sequence $(\psi_n, n \geq 1) \subset C_{c}^{\infty}(\mathbb{R}^{d})$ (possibly depending on $t$) for Lebesgue-a.e.\ $t \in [0,1]$. A similar argument holds for its derivatives. This gives that
% \begin{align*}
%     \Exp{\SBdynam{\rho}{\eps}}[\cU(\omega_s)-\cU(\omega_0)] = \Exp{\rho_{s}^{\eps}}[\cU] - \Exp{\rho_{0}^{\eps}}[\cU] &= \int_{0}^{s} \Exp{\rho_{u}^{\eps}}[\langle v_{u}^{\eps}, \nabla \cU \rangle]du.
% \end{align*}
Integrating \eqref{eq:cont-eqn-for-U} once more in $s$ over $[0,1]$, multiplying by $\eps$, and plugging in \eqref{eq:cont-eq-entropic-inter} gives the RHS of \eqref{eq:rxyeps-bound} as
\begin{align*}
    \eps \int_{0}^{1} \Exp{\SBdynam{\rho}{\eps}} [\cU(\omega_s)-\cU(\omega_0)] ds &= \eps^2 \int_{0}^{1} \int_{0}^{s}  \Exp{\rho_{u}^{\eps}}\left[\left\langle \frac{1}{\eps}v_{u}^{\eps},\nabla \cU\right\rangle\right]  duds.  
\end{align*}
Apply Cauchy-Schwarz on $L^{2}(C^{d}[0,1] \times [0,1], d\SBdynam{\rho}{\eps}\otimes du)$ to obtain for each $s \in [0,1]$
\begin{align*}
    \abs{-\int_{0}^{s}\Exp{\SBdynam{\rho}{\eps}}\left[\left\langle \frac{1}{\eps} v_{u}^{\eps}(\omega_u),\nabla \cU(\omega_u) \right\rangle\right]  du} &\leq \int_{0}^{1} \Exp{\SBdynam{\rho}{\eps}}\left[\frac{1}{\eps}\|v_{u}^{\eps}(\omega_u)\|\|\nabla \cU(\omega_u)\|\right] du \\
    &\leq \left(\int_{0}^{1} \frac{1}{\eps^2}\|v_{u}^{\eps}\|_{L^{2}(\rho_{u}^{\eps})}^{2} du \right)^{1/2}\left(\int_{0}^{1} \Exp{\rho_{u}^{\eps}}[\|\nabla \cU\|^2] du \right)^{1/2}
\end{align*}
% We now wish to provide a bound on the following term
% \begin{align*}
%      -\int_{0}^{1}\int_{0}^{s} \frac{1}{2}\Exp{P_{\eps}}\left[\left(\nabla \psi^{\eps}(u,\omega_u)- \nabla \psi^{\eps}(1-u,\omega_u)\right)\cdot \nabla F(\omega_u)\right]duds.
% \end{align*}
Thus, altogether we have that
\begin{align}\label{eq:almost-bound-rxyeps}
    \Exp{\SBstatic{\rho}{\eps}{}}[R(X,Y,\eps)] &\leq \eps^2\left(\int_{0}^{1} \frac{1}{\eps^2}\|v_{t}^{\eps}\|_{L^{2}(\rho_{t}^{\eps})}^{2} dt\right)^{1/2}\left(\int_{0}^{1} \Exp{\rho_{t}^{\eps}}[\|\nabla \cU\|^2] dt \right)^{1/2}.
\end{align}

Recall the entropic Benamou Brenier expression for entropic cost given in \eqref{eq:symm-ent-cost-BB}. By the optimality of \eqref{eq:cont-eq-entropic-inter} and suboptimality of the constant curve $\rho_t = \rho$ for all $t \in [0,1]$,
\begin{align*}
    H(\SBstatic{\rho}{\eps}|\BMstatic{\rho}{\eps}) &= \frac{1}{2\eps} \int_{0}^{1} \|v_{t}^{\eps}\|_{L^{2}(\rho_t^{\eps})}^{2} dt + \frac{\eps}{8}\int_{0}^{1} I(\rho_{t}^{\eps})dt \leq \frac{\eps}{8}I(\rho).
\end{align*}
Rearranging terms gives
\begin{align}\label{eq:integrated-velocity-bound}
    \frac{1}{\eps^2} \int_{0}^{1} \|v_{t}^{\eps}\|_{L^{2}(\rho_t^\eps)}^2 dt \leq \frac{1}{4}\left(I(\rho)-\int_{0}^{1} I(\rho_t^{\eps}) dt \right).
\end{align}

% By the non-symmetric expansions of the $\Schro$ cost given in \eqref{eq:nonsymm-ent-cost-BB} and the optimality of \eqref{eq:cont-eq-entropic-inter} we observe that
% \begin{align}\label{eq:nonsym-cost-expan}
%     \int_{0}^{1} \|\nabla \psi^{\eps}(t,\cdot)\|_{L^{2}(\rho_{t}^{\eps})}^2 dt = \frac{2}{\eps}(H(\SBstatic{\rho}{\eps}{}|\BMstatic{\rho}{\eps})-\Ent(\rho)).
% \end{align}
% By \cite[Theorem 1.6]{conforti21deriv} we then have
% \begin{align}\label{eq:integrated-psi-bdd}
%     \limsup\limits_{\eps \downarrow 0} \int_{0}^{1} \|\nabla \psi^{\eps}(t,\cdot)\|_{L^{2}(\rho_{t}^{\eps})}^2 dt &\leq \frac{1}{4}I(\rho). 
% \end{align}
With this inequality, \eqref{eq:almost-bound-rxyeps} becomes the desired
\begin{align*}
    \Exp{\SBstatic{\rho}{\eps}{}}[R(X,Y,\eps)] &\leq \frac{1}{2}\eps^2\left(I(\rho) - \int_{0}^{1} I(\rho_{t}^{\eps})dt\right)^{1/2}\left(\int_{0}^{1} \Exp{\rho_{t}^{\eps}}[\|\nabla \cU\|^2] dt \right)^{1/2}.
\end{align*}
Under Assumption \ref{assumption:LD_SB_relative_entropy} (2), Lemma \ref{lem:integ-entropic-inter} applies and gives that the rightmost constant has an upper bound once we consider $\eps \in (0,\eps_{0})$ for some $\eps_0 > 0$. Thus, the nonnegativity of the LHS of \eqref{eq:rel-ent-ld-sb} and Lemma \ref{lem:integrated-fisher-continuity} give the stated convergence as $\eps \downarrow 0$. 

\end{proof}

\subsection{Discussion}
% \GM{move discussion of sharpness of bound to after statement of Theorem}

\textbf{Integrated Fisher Information.} Under additional assumptions the integrated Fisher information along the entropic interpolation in \eqref{eq:rel-ent-ld-sb} can be replaced with a potentially more manageable quantity. In certain settings (for instance, when $\rho \in C_{c}^{\infty}(\mathbb{R}^{d})$ with finite entropy and Fisher information \cite[Lemma 3.2]{glrt-hwi-20}), there is a conserved quantity along the entropic interpolation called the energy, which we denote $\mathcal{E}_{\eps}(\rho)$. With $(v_{t}^{\eps},\rho_{t}^{\eps})_{t \in [0,1]}$ as defined in \eqref{eq:cont-eq-entropic-inter}, for any $t \in [0,1]$ 
\begin{align}\label{eq:energy-along-entropic-inter}
    \mathcal{E}_{\eps}(\rho) := \frac{1}{2\eps^2}\|v_{t}^{\eps}\|_{L^{2}(\rho_{t}^{\eps})}^{2} - \frac{1}{8}I(\rho_{t}^{\eps}).
\end{align}
Now, evaluate \eqref{eq:energy-along-entropic-inter} at $t = 1/2$. Then $v_{1/2}^{\eps} = 0$ and $\mathcal{E}_{\eps}(\rho) = - \frac{1}{8}I(\rho_{1/2}^{\eps})$. On the other hand, integrate \eqref{eq:energy-along-entropic-inter} over $t \in [0,1]$, this also gives $\mathcal{E}_{\eps}(\rho)$ and thus
\begin{align*}
    -\frac{1}{8}I(\rho_{1/2}^{\eps}) &= \frac{1}{2\eps^2}\int_{0}^{1}\|v_t^\eps\|_{L^{2}(\rho_{t}^{\eps})}^{2} dt - \frac{1}{8}\int_{0}^{1} I(\rho_{t}^{\eps}) dt.
\end{align*}
This in turn implies
\begin{align*}
    \frac{1}{\eps^2}\int_{0}^{1}\|v_t^\eps\|_{L^{2}(\rho_{t}^{\eps})}^{2} dt &= \frac{1}{4}\left(\int_{0}^{1} I(\rho_{t}^{\eps}) dt-I(\rho_{1/2}^{\eps})\right) \leq \frac{1}{4}\left(I(\rho) - I(\rho_{1/2}^{\eps})\right).
\end{align*}
Thus, assuming the conservation of energy holds under Assumption \ref{assumption:LD_SB_relative_entropy}, replacing the bound in \eqref{eq:integrated-velocity-bound} with the RHS of the above inequality rewrites Theorem \ref{thm:ld-sb-general-g} as
\begin{align}\label{eq:rel-ent-sb-ld-midpoint}
    H(\LDstatic{\rho}{\eps} | \SBstatic{\rho}{\eps})+H(\SBstatic{\rho}{\eps} | \LDstatic{\rho}{\eps}) \leq \frac{1}{2}\eps^2 \left(I(\rho) - I(\rho_{1/2}^{\eps})\right)^{1/2}\left(\int_{0}^{1} \Exp{\rho_t^\eps}\|\nabla \cU\|^2 dt\right)^{1/2}.
\end{align}
Alternatively, \eqref{eq:rel-ent-sb-ld-midpoint} also holds when $\rho$ is a univariate Gaussian, as explicit computations of the entropic interpolation in this case from \cite{gentil2017analogy} demonstrate that $I(\rho_{t}^{\eps})$ is minimized at $t = 1/2$. 
% \GM{we get $O(\eps^3$ in Gaussian case}

% \textbf{Gaussian Computation.} \GM{discuss sharpness of bound here}
% When $\rho = N(0,1)$, consider the two matrices
% \begin{align*}
%     \Sigma_1 =  \begin{pmatrix}
%         1 & e^{-\eps/2} \\ e^{-\eps/2} & 1
%     \end{pmatrix}, \Sigma_2 = \begin{pmatrix}
%         1 & \frac{1}{2}(\sqrt{\eps^2+4}-\eps) \\ \frac{1}{2}(\sqrt{\eps^2+4}-\eps) & 1 
%     \end{pmatrix}
% \end{align*}
% It is well known that $\LDstatic{\rho}{\eps} \sim N(0,\Sigma_1)$ is the Ornstein-Uhlenbeck transition density. By \cite{janati2020}, it is known that $\SBstatic{\rho}{\eps} \sim N(0,\Sigma_2)$. One can compute $H(\LDstatic{\rho}{\eps}|\SBstatic{\rho}{\eps}) = \frac{1}{2304}\eps^4+O(\eps^5)$. That is, the transition densities for the OU process provide a much sharper approximation (almost two orders of magnitude) for same marginal $\Schro$ bridge than given by Theorem \ref{thm:ld-sb-general-g}. 

\textbf{$\Schro$ Cost Expansion.}
In the course of proving Theorem \ref{thm:ld-sb-general-g}, we recover a result in alignment with the expansion of $\Schro$ cost about $\eps = 0$ developed in \cite[Theorem 1.6]{conforti21deriv}. Under assumptions on bounded $\rho$ with compact support and finite entropy, \cite[Theorem 1.6]{conforti21deriv} states
\begin{align}\label{eq:ct-thm1-6}
    H(\SBstatic{\rho}{\eps}|\BMstatic{\rho}{\eps}) &= \frac{\eps}{8}I(\rho) + o(\eps). 
\end{align}
Now, from the Pythagorean Theorem of relative entropy \cite[Theorem 2.2]{csiszar75idiv}
\begin{align*}
    H(\LDstatic{\rho}{\eps}|\BMstatic{\rho}{\eps}) &\geq H(\LDstatic{\rho}{\eps}|\SBstatic{\rho}{\eps}) + H(\SBstatic{\rho}{\eps}|\BMstatic{\rho}{\eps}).
\end{align*}
Recall from Step 2 of Lemma \ref{thm:ld-sb-general-g} that $H(\LDstatic{\rho}{\eps}|\BMstatic{\rho}{\eps}) \leq \frac{\eps}{8}I(\rho)$, so by Theorem \ref{thm:ld-sb-general-g} we obtain the following expansion of an upper bound of $\Schro$ cost
\begin{align}\label{eq:CTextension}
    H(\SBstatic{\rho}{\eps}|\BMstatic{\rho}{\eps}) &\leq H(\LDstatic{\rho}{\eps}|\BMstatic{\rho}{\eps}) - H(\LDstatic{\rho}{\eps}|\SBstatic{\rho}{\eps}) = \frac{\eps}{8}I(\rho) - o(\eps^2). 
\end{align}
Thus, in the same-marginal case (and under different regularity assumptions) the second order term in the $\Schro$ expansion is zero.

% \textbf{Gaussian Computation,}
% When $\rho = N(0,1)$, consider the two matrices
% \begin{align*}
%     \Sigma_1 =  \begin{pmatrix}
%         1 & e^{-\eps/2} \\ e^{-\eps/2} & 1
%     \end{pmatrix}, \Sigma_2 = \begin{pmatrix}
%         1 & \frac{1}{2}(\sqrt{\eps^2+4}-\eps) \\ \frac{1}{2}(\sqrt{\eps^2+4}-\eps) & 1 
%     \end{pmatrix}
% \end{align*}
% It is well known that $\LDstatic{\rho}{\eps} \sim N(0,\Sigma_1)$ is the Ornstein-Uhlenbeck transition density. By \cite{janati2020}, it is known that $\SBstatic{\rho}{\eps} \sim N(0,\Sigma_2)$. One can compute $H(\LDstatic{\rho}{\eps}|\SBstatic{\rho}{\eps}) = \frac{1}{2304}\eps^4+O(\eps^5)$. That is, the transition densities for the OU process provide a much sharper approximation (almost two orders of magnitude) for same marginal $\Schro$ bridge than given by Theorem \ref{thm:ld-sb-general-g}. 

\textbf{Convergence of Potentials.}
Revisiting the inequality in \eqref{eq:integrated-velocity-bound} and recalling \eqref{eq:cont-eq-entropic-inter}, observe that 
\begin{align*}
    \int_{0}^{1} \|\nabla \psi^{\eps}(t,\cdot)-\nabla\psi^{\eps}(1-t,\cdot)\|_{L^2(\rho_t^{\eps})}^2 dt \leq I(\rho)-\int_{0}^{1} I(\rho_t^{\eps}) dt. 
\end{align*}
By Lemma \ref{lem:integrated-fisher-continuity}, for Lebesgue a.e.\ $t \in [0,1]$ it holds that
\begin{align}\label{eq:pointwise-l2-entrop-limit}
    \lim\limits_{\eps \downarrow 0} \|\nabla \psi^{\eps}(t,\cdot)-\nabla\psi^{\eps}(1-t,\cdot)\|_{L^2(\rho_t^{\eps})}^2 = 0.
\end{align}
This fact hints in the direction of a result on the convergence of rescaled gradients of entropic potentials. Suppose that the convergence in \eqref{eq:pointwise-l2-entrop-limit} holds at $t = 1$ or $t= 0$ (of course, a priori there is no way to guarantee such convergence at a specific $t$). In this case,
\begin{align*}
    \nabla \psi^{\eps}(1,\cdot)-\nabla\psi(0,\cdot) &=  \nabla \log a^{\eps}-\nabla \log P_{\eps}a^{\eps} \\
    &= \nabla \log a^{\eps} - \nabla \log \frac{\rho}{a^{\eps}}  = \frac{2}{\eps} \nabla f_{\eps} + \nabla \log \rho.
\end{align*}
As $\rho_0^{\eps} = \rho_1^\eps = \rho$, \eqref{eq:pointwise-l2-entrop-limit} would imply that $\lim\limits_{\eps \downarrow 0} \frac{1}{\eps} \nabla f_{\eps} = -\frac{1}{2}\nabla \log \rho$ in $L^{2}(\rho)$. This would be in accordance with \cite[Theorem 1]{sander_22}, but now with fully supported densities. %\SP{Change the last sentence.}

% \subsection{Example Computation}
% \input{arxiv_version/LangevinSBApprox/ou-computations}

% \section{Uniform Convergence of Scheme}\label{uniform-convergence-section}
% \input{UniformConvergence/main}

\section{One Step Analysis of $\Schro$ Bridge Scheme}\label{sec:one-step}
%\subsection{Setup}
%\input{arxiv_version/SBScheme/setup}
%\subsection{One Step Convergence}
Fix $\rho \in \cP_{2}^{ac}(\mathbb{R}^{d})$ and recall the tangent space at $\rho$, denoted by $\Tan{\rho}$, and defined in \cite[Definition 8.4.1]{ambrosio2005gradient} by
\begin{align*}
    \Tan{\rho} := \overline{\{\nabla \varphi: \varphi \in C_{c}^{\infty}(\mathbb{R}^{d})\}}^{L^{2}(\rho)}.
\end{align*}
Now, consider a vector field $v=\nabla u \in \Tan{\rho}$, such that the surrogate measure \eqref{eq:surrogate} is defined. Our objective in this section is to lay down conditions under which the approximation stated below \eqref{sb-step} holds. More precisely, we show that 
\begin{equation}\label{eq:short_thm_2}
\lim_{\eps \downarrow 0}\frac{1}{\eps}\Was{2}(\SB{\rho}{\eps}{},(\Id + \eps v)_{\#}\rho)=0.
\end{equation}

%We drop the index $k$ throughout this section which focuses on one step of the iteration. 

% For instance, one can set $v_{\eps} = \frac{1}{\eps}(T_{\rho}^{\opt{\eps}{1}{\rho}}-\Id)$. Moreover, this implies that $v_{\eps} = -\frac{1}{\eps} \nabla \psi_{\eps}$, where $\psi_{\eps}: \mathbb{R}^{d} \to \mathbb{R}$ is a Kantorovich potential (see \cite[Definition 1.12]{santam2015ot}) for the quadratic cost optimal transport from $\rho$ to $T_{\eps}^{1}(\rho)$. That is, each $v_{\eps}$ is of the gradient type.

%The setting we investigate is the following. Let $(\nabla \psi_{\eps}, \eps > 0)\subset L^{2}(\rho)$ be the vector fields given above.

In fact, we will prove the above theorem in a somewhat more general form when the tangent vector field may not even be of a gradient form. The general setting is as follows. Consider a collection of vectors fields of gradient type $(v_{\eps} = \nabla \psi_{\eps}, \eps > 0)$,
% $v_{\eps}=\nabla \psi_\eps: \mathbb{R}^{d} \to \mathbb{R}^{d}$, 
we wish to tightly approximate the pushforwards $(\Id+\eps v_{\eps})_{\#}\rho$ with $\SB{\rho}{\eps}{}$. As a shorthand, we write
\begin{align}\label{eq:S-eps}
    \opt{\eps}{1}{\rho} := (\Id + \eps v_{\eps})_{\#}\rho.
\end{align}
Suppose that for each $\eps > 0$ there exists $\theta_\eps \in \mathbb{R} \setminus \{0\}$ such that the surrogate measure defined in \eqref{eq:surrogate} exists. Set
\begin{align}\label{eq:simga-eps-onestep}
    \sigma_{\eps}(x) &= \exp\left(2\theta_{\eps} \psi_{\eps}(x)- \Lambda_0(\theta_{\eps})\right), \quad \text{where}\quad \Lambda_0(\theta_{\eps}):=\log \int e^{2\theta \psi_{\eps}(y)}dy.
\end{align}
Redefine one step in the SB scheme as
\begin{align}\label{sb-step-general}
    \tag{SB} \SB{\rho}{\eps}{} &= \left(\left(1-\theta_{\eps}^{-1}\right)\Id + \theta_{\eps}^{-1} \BPbase{\sigma_\eps}{\eps}\right)_{\#}\rho,
\end{align}
where $\BPbase{\sigma_\eps}{\eps}$ is the barycentric projection defined in \eqref{eq:bary-proj}.

In Theorem \ref{thm:one_step_convergence} we show that, under appropriate assumptions, 
\begin{equation}\label{eq:short_thm_2-general}
\lim_{\eps \downarrow 0}\frac{1}{\eps}\Was{2}(\SB{\rho}{\eps}{},\opt{\eps}{1}{\rho})=0.
\end{equation} 

There are multiple reasons why one might be interested in this general setup. Firstly, we know that, for any $v \in \Tan{\rho}$, there exists a sequence of tangent elements of the type $v_\eps=\nabla \psi_\eps$ such that $\lim_{\eps\rightarrow 0+} v_\eps=v$ in $\mathbf{L}^2(\rho)$. In particular, by an obvious coupling, 
\begin{equation}\label{eq:ltwoapprox}
\limsup_{\eps \rightarrow 0+}\frac{1}{\eps} \Was{2}( (\Id + \eps v_{\eps})_{\#}\rho ,(\Id + \eps v)_{\#}\rho)\le \lim_{\eps \rightarrow 0+}\norm{v_\eps - v}_{\mathbf{L}^2(\rho)}=0.
\end{equation}
But, of course, if $v$ is not of the gradient form, one cannot construct a surrogate measure as in \eqref{eq:surrogate} to run the Schr\"odinger bridge scheme. A natural remedy is to show that \eqref{eq:short_thm_2-general} holds under appropriate assumptions. Then, combined with \eqref{eq:ltwoapprox} we recover \eqref{eq:short_thm_2} even when $v$ is not a gradient. 

There is one more reason why one might use the above generalized scheme. As explained in \eqref{eq:explicit_euler_Intro}, we would like to iterate the Schr\"odinger bridge scheme to track the explicit Euler iterations of a Wasserstein gradient flow. However, the more natural discretization scheme of Wasserstein gradient flows is the implicit Euler or the JKO scheme. See \eqref{eq:IE} in Section \ref{sec:schro-bridge-scheme} for the details. In this case, however, the vector field $v_\eps$, which is the Kantorovich potential between two successive steps of the JKO scheme, changes with $\eps$. This necessitates our generalization. Hence, we continue with \eqref{eq:simga-eps-onestep} and the generalized \eqref{sb-step-general} scheme. 

The proof of \eqref{eq:short_thm_2-general} relies on the close approximation of the Langevin diffusion to the $\Schro$ bridge provided in Theorem \ref{thm:ld-sb-general-g}. To harness this approximation, we introduce an intermediatery scheme based on the Langevin diffusion, which we denote $\LD{\cdot}{\eps}{}$ and define in the following manner. Let $\eps > 0$ and consider the symmetric Langevin diffusion $\left( X_t,\; 0\le t \le \eps\right)$ with stationary density $\sigma_{\eps}$ as defined in \eqref{eq:simga-eps-onestep}. The drift function for this diffusion is $x \mapsto \frac{1}{2}\nabla \log \sigma_{\eps}(x)= \theta_{\eps}\nabla \psi_\eps(x)=\theta_{\eps}v_{\eps}(x)$. In analogy with the definition of $\BPbase{\sigma_{\eps}}{\eps}(\cdot)$ in \eqref{eq:bary-proj}, define $\LBPbase{\sigma_{\eps}}{\eps}(\cdot)$ by
\begin{align}
    \LBPbase{\sigma_{\eps}}{\eps}(x) &= \Exp{\LDstatic{\sigma_{\eps}}{\eps}}[Y|X=x].
\end{align}
Next, in analogy with the definition of $\SB{\cdot}{\eps}{}$ in \eqref{sb-step-general}, define the Langevin diffusion update
\begin{align}\label{eq:LD-update}
    \tag{LD} \LD{\rho}{\eps}{} &= \left(\left(1-\theta_\eps^{-1}\right)\Id + \theta_\eps^{-1}\LBPbase{\sigma_{\eps}}{\eps}\right)_{\#}\rho. 
\end{align}
In analogy with Section \ref{sec:sb-ld}, for $\eps > 0$ set
    \begin{align*}
        \cU_{\eps} &:= \frac{1}{8}\|\nabla \log \sigma_{\eps}\|^2 + \frac{1}{4}\Delta \log \sigma_{\eps}.
    \end{align*}
We now make the following assumptions on the collection of densities $(\sigma_\eps, \eps > 0)$ defined in \eqref{eq:simga-eps-onestep}.
\begin{assumption}\label{assumptions:sigma-eps-onestep}
For each $\eps > 0$, $\sigma_\eps$ as defined in \eqref{eq:simga-eps-onestep} satisfies Assumption \ref{assumption:LD_SB_relative_entropy}. Additionally
\begin{itemize}
    \item[(1)] There exists $\sigma_0 \in \cP_{2}^{ac}(\mathbb{R}^{d})$ such that $(\sigma_\eps, \eps > 0)$ converges weakly to $\sigma_0$ and $I(\sigma_\eps) \to I(\sigma_0)$ as $\eps \downarrow 0$. Moreover, Assumption \ref{assumption:LD_SB_relative_entropy} holds uniformly in $\eps$ in the following manner: there exists $\eps_0 > 0$ and $C,D > 0$, $N \geq 1$ such that for all $\eps \in (0,\eps_0)$, $\|\nabla \cU_{\eps}(x)\|^2 \leq C(1+\|x\|^{N})$ and for $X_{\eps} \sim \sigma_{\eps}$, $\|X_{\eps}\|_{\psi_{1}} \leq D$.
    % whenever $\|x\| \leq R$ and $\|\nabla \cU_{\eps}(x)\|^2 \leq C\|x\|^{N}$ whenever $\|x\| \geq R$. 
%    \GM{to avoid introducing $R$, should we just change this assumption to $\|\nabla \cU(x)\|^2 \leq C(1+\|x\|^{N})$-- this is all we end up using anyway}\SP{Agreed}\GM{fixed-- will make changes everywhere else}
%    \SP{Change or define this notation for weak convergence: $\sigma_\eps \rightharpoonup \sigma_0$.}\GM{dropped this notation}
    \item[(2)] The $(\sigma_{\eps}, \eps \in (0,\eps_0))$ are uniformly semi log-concave, that is, there exists $\lambda \in \mathbb{R}$ such that for all $\eps \in (0,\eps_0)$, $- \nabla^{2}\log \sigma_{\eps} \geq \lambda \Id$. Additionally, $L := \limsup\limits_{\eps \downarrow 0} \Exp{\sigma_{\eps}}\|\nabla^2 \log \sigma_{\eps}\|_{HS}^{2} < +\infty$, where $\|.\|_{HS}$ is the Hilbert-Schmidt norm.
%     $\lambda_{\eps
% } > 0$ such that $-\nabla^2 \log \sigma_{\eps} \geq \lambda_{\eps} \Id$ and $\lambda := \inf\limits_{\eps \in (0,\eps_0)} \lambda_{\eps} > 0$ \GM{replace this with there exists $K \in \mathbb{R}$ such that for all $\eps \in (0,\eps_0)$, $- \nabla^{2}\log \sigma_{\eps} \geq K \Id$} \cite[Theorem 1.8(3)]{cattiaux-guillin-slc-14} \GM{also need something like uniform $\|.\|_{\psi_{1}}$ upper bound for Lemma \ref{lem:bound-preserved-eps} to hold} and $\limsup\limits_{\eps \downarrow 0} \Exp{\sigma_{\eps}}\|\nabla^2 \log \sigma_{\eps}\|_{HS}^{2} < +\infty$.
    \item[(3)] $\chi:=\limsup\limits_{\eps
     \downarrow 0} \|\rho/\sigma_{\eps
     }\|_{\infty} < +\infty$ and $\theta:=\liminf\limits_{\eps \downarrow 0} \abs{\theta_{\eps}} > 0$.
    % \item[(4)] $\limsup\limits_{\eps \downarrow 0} \Exp{\sigma_{\eps}}\|\nabla^2 \log \sigma_{\eps}\|_{HS}^{2} < +\infty$
    % \item[(4)] $\sigma_\eps \rightharpoonup \sigma_0$ as $\eps \downarrow 0$, and $\psi_{\eps} \to \psi_{0}$, $\nabla \psi_{\eps} \to \nabla \psi_0$ Lebesgue a.e.\ as $\eps \downarrow 0$. Moreover, there exists $\eps_0 > 0$ and $C,R > 0$, $N \geq 1$ such that for all $\eps \in [0,\eps_0)$ for all multi-indices $\alpha$ with $\abs{\alpha} \leq 3$
    % \begin{align*}
    % \|\partial_{\alpha}\psi_{\eps}\|_{\infty} \leq C \text{ in $B(0,R)$, and } \|\partial_{\alpha}\psi_{\eps}(x)\| \leq C\|x\|^{N} \text{for $\|x\| \geq R$}
    % \end{align*}
\end{itemize}
\end{assumption}

% \GM{I think a better (more succinct and in line with Assumption \ref{assumption:LD_SB_relative_entropy}) would be something like this. The issue is that we have to add many additional conditions to get convergence of Fisher information that I think we would be better off assuming it (and these are technical assumptions that distract from the essence of the proof)\\
% (4) $\sigma_\eps \rightharpoonup \sigma_0$ and $I(\sigma_\eps) \to I(\sigma_0)$ as $\eps \downarrow 0$. Moreover, Assumption \ref{assumption:LD_SB_relative_entropy} holds uniformly in $\eps$ in the following manner: there exists $\eps_0 > 0$ and $C,R > 0$, $N \geq 1$ such that for all $\eps \in (0,\eps_0)$, $\|\nabla \cU_{\eps}(x)\|^2 \leq C$ whenever $\|x\| \leq R$ and $\|\nabla \cU_{\eps}(x)\|^2 \leq C\|x\|^{N}$ whenever $\|x\| \geq R$}
 % When $\opt{\eps}{1}{\cdot}$ is the explicit Euler iteration (see \eqref{eq:GA}), observe that for each $\eps > 0$ the update is of the form $\opt{\eps}{1}{\rho} = (\Id + \eps \nabla \psi)_{\#}\rho$ for a function $\psi: \mathbb{R}^{d} \to \mathbb{R}$ that is independent of $\eps$. Thus, the surrogate measure does not change in $\eps$, i.e.\ there exists $\sigma_0 \in \cP_{2}^{ac}(\mathbb{R}^{d})$ such that $\sigma_{\eps} = \sigma_0$ for all $\eps > 0$. However, the surrogate measure will change in $\eps$ when $\opt{\eps}{1}{\cdot}$ is the implicit Euler iteration.
  
\begin{remark}\label{rmk:theta}
   In all our examples, $\theta_\eps$ is either always $+1$ or always $-1$. Hence, this parameter simply needs to have a positive lower bound (in absolute value) and it does not scale with $\eps$. 
\end{remark}

While these assumptions may appear stringent, they are indeed satisfied by a wide class of examples, see Section \ref{sec:examples}. In many of these examples, $\sigma_{\eps}$ takes the form $\sigma_{\eps} = \rho e^{V_{\eps}}$. Thus, condition (3) in Assumption \ref{assumptions:sigma-eps-onestep} is satisfied if $\liminf\limits_{\eps \downarrow 0} \inf\limits_{x\in \mathbb{R}^{d}} V_{\eps} > -\infty$. For instance, when using \eqref{sb-step} to approximate \eqref{eq:GA} for the gradient flow of entropy, $\sigma_{\eps} = \rho$ for all $\eps > 0$ and thus $V_{\eps} \equiv 0$. Similarly, for approximating the gradient flow of $H(\cdot | N(0,1))$ from certain starting measures, $V_{\eps}(x) = \frac{1}{2}x^2 + C$, where $C$ is a normalizing constant, for all $\eps > 0$. See Section \ref{sec:examples} for more detailed discussion.

%  In this setting, one step of the three iteration schemes writes as
% \begin{align*}
%     \opt{\eps}{}{\rho} &= \left(\Id + \eps \nabla \psi_{\eps}\right)_{\#}\rho \\
%     \SB{\rho}{\eps}{} &= \left(\left(1-\theta_{\eps}^{-1}\right)\Id + \theta_{\eps}^{-1}\BPbase{\sigma_{\eps}}{\eps}\right)_{\#}\rho\\
%     \LD{\rho}{\eps}{} &= \left(\left(1-\theta_{\eps}^{-1}\right)\Id + \theta_{\eps}^{-1}\LBPbase{\sigma_{\eps}}{\eps}\right)_{\#}\rho.
% \end{align*}

% The major result for this subsection is that for small $\eps > 0$, the iteration given by \eqref{sb-step} is within $o(\eps)$ of the corresponding Euler iteration (either implicit or explicit). This is a generalization of the one-step approximation developed in the Introduction in \eqref{eq:sbldapprox} in which \eqref{sb-step} approximates an explicit Euler iteration for the heat flow. 
We now provide a quantitative statement about the convergence stated in \eqref{eq:short_thm_2-general},
% The quantitative statement of this statement is given in Theorem \ref{thm:one_step_convergence}, 
where we take care to note the dependence on constants appearing in Assumption \ref{assumptions:sigma-eps-onestep} as this upper bound must hold uniformly along iterations of this scheme in Section \ref{sec:schro-bridge-scheme}. Recall $\opt{\eps}{1}{\cdot}$ as defined in \eqref{eq:S-eps} and $\SB{\cdot}{\eps}{}$ as defined in \eqref{sb-step-general}.

% \GM{describe in plain english what is going on here}
\begin{theorem}\label{thm:one_step_convergence}
Under Assumption \ref{assumptions:sigma-eps-onestep}, there exists a constant $K > 0$ depending on $C,D,N,\lambda,L, \chi, \theta$,
% depending linearly on $C$, polynomially on $\lambda^{-1}$, and super-exponentially on $N$,
such that for all $\eps \in (0,\eps_0)$
\begin{align*}
    \Was{2}(\SB{\rho}{\eps}{},\opt{\eps}{1}{\rho})&\leq
    K \eps\left[ \left(I(\sigma_{\eps})-\int_{0}^{1}I((\sigma_{\eps})^{\eps}_{s})ds\right)^{1/4} + \sqrt{\eps}\right].
\end{align*}  

%\[
%\eps\abs{\theta_{\eps}}^{-1}\left\|\frac{\rho}{\sigma_{\eps}}\right\|_{\infty}^{1/2}\left(K \left(I(\sigma_{\eps})-\int_{0}^{1}I((\sigma_{\eps})^{\eps}_{s})ds\right)^{1/4} + \left(2\eps \Exp{\sigma_{\eps}}\|\nabla^2 \log \sigma_{\eps}\|^2_{HS}\right)^{1/2}\right).
%\]

In particular,
% \begin{align*}
%     \lim\limits_{\eps \downarrow 0} \frac{1}{\eps}\Was{2}(\SB{\rho}{\eps}{},\LD{\rho}{\eps}{})=0, \; \lim\limits_{\eps \downarrow 0} \frac{1}{\eps}\Was{2}(\LD{\rho}{\eps}{},\opt{\eps}{1}{\rho})=0, \; \lim\limits_{\eps \downarrow 0} \frac{1}{\eps}\Was{2}(\SB{\rho}{\eps}{},\opt{\eps}{1}{\rho})=0. 
% \end{align*}
\begin{align*}
    \lim\limits_{\eps \downarrow 0} \frac{1}{\eps}\Was{2}(\SB{\rho}{\eps}{},\opt{\eps}{1}{\rho})=0. 
\end{align*}
% Lastly, $K$ has a uniform upper bound so long as $C$ and $N$ are bounded above and $\lambda$ has positive lower bound.  
\end{theorem}

Hence, from the discussion following \eqref{eq:ltwoapprox} we immediately obtain:

\begin{corollary}\label{lem:general-v-approx}
Let $v \in \Tan{\rho}$, and let $(\nabla \psi_{\eps}, \eps > 0) \subset C^{\infty}(\mathbb{R}^{d}) \cap \Tan{\rho}$ be such that $\lim\limits_{\eps \to 0+} \|\nabla \psi_{\eps}-v\|_{\mathbf{L}^2(\rho)} =0$ and the corresponding surrogate measures $(\sigma_{\eps} > 0)$ satisfy Assumption \ref{assumptions:sigma-eps-onestep}. Then
\begin{align*}
    \lim\limits_{\eps \downarrow 0+} \frac{1}{\eps}\Was{2}((\Id+\eps v)_{\#}\rho,\SB{\rho}{\eps}{}) = 0.
\end{align*}
% where $\SB{\rho}{\eps}{}$ approximates $\opt{\eps}{1}{\rho} = (\Id+\eps \nabla \psi_{\eps})_{\#}\rho$. 
\end{corollary}

\begin{proof}[Proof of Theorem \ref{thm:one_step_convergence}]
This result follows from Lemmas \ref{lem:BP_upperbounds}, \ref{lem:bound-preserved-eps}, and \ref{lem:LD_BP_covergence} below. Lemmas \ref{lem:BP_upperbounds} and \ref{lem:bound-preserved-eps} together prove an upper bound on $\Was{2}(\SB{\rho}{\eps}{},\LD{\rho}{\eps}{})$ that show, in particular,
\[
\lim\limits_{\eps \downarrow 0} \frac{1}{\eps}\Was{2}(\SB{\rho}{\eps}{},\LD{\rho}{\eps}{})=0,
\]
and Lemma \ref{lem:LD_BP_covergence} proves an upper bound on $\Was{2}(\LD{\rho}{\eps}{},\opt{\eps}{1}{\rho})$ which implies 
\[
\lim\limits_{\eps \downarrow 0} \frac{1}{\eps}\Was{2}(\LD{\rho}{\eps}{},\opt{\eps}{1}{\rho})=0.
\]
That is, the three updates $\SB{\rho}{\eps}{}$, $\LD{\rho}{\eps}{}$, and $\opt{\eps}{1}{\rho}$ are within $o(\eps)$ of each other in the Wasserstein-2 metric. Theorem \ref{thm:one_step_convergence} now follows from the triangle inequality and simplifying the choice of the constant $K$.
\end{proof}

% Moreover, Theorem \ref{thm:one_step_convergence} allows us to approximate the pushforward by arbitrary velocity fields in $\Tan{\rho}$ that can be approximated by certain velocity fields of gradient type in $\Tan{\rho}$. 
% giving the following Corollary that is an immediate conseqeuence of the Theorem and \eqref{eq:ltwoapprox}.
%\begin{proof}[Proof of Corollary \ref{lem:general-v-approx}]
%From the triangle inequality,
%\begin{align*}
%    \Was{2}((\Id+\eps v)_{\#}\rho,\SB{\rho}{\eps}{}) &\leq \Was{2}((\Id+\eps v)_{\#}\rho,\opt{\eps}{1}{\rho}) + \Was{2}(\opt{\eps}{1}{\rho},\SB{\rho}{\eps}{}).
%\end{align*}
%The first term on the RHS vanishes by \eqref{eq:ltwoapprox}, and the second term vanishes by Theorem \ref{thm:one_step_convergence}.
%\end{proof}

We begin by converting the relative entropy decay established in Theorem \ref{thm:ld-sb-general-g} between the $\Schro$ bridge and Langevin diffusion into a decay rate between their respective barycentric projections in the Wasserstein-2 metric. The essential tool for this conversion is a Talagrand inequality from \cite[Theorem 1.8(3)]{cattiaux-guillin-slc-14}.

% \GM{expalin what's going on with $\theta_{\eps}$ here}
\begin{lemma}\label{lem:BP_upperbounds}
Let $\eps > 0$. Under Assumption \ref{assumptions:sigma-eps-onestep}, there exists a constant $\alpha>0$ depending on $\eps$ and $\lambda$ such that
\begin{align*}
    \Was{2}^{2}((\BPbase{\sigma_{\eps}}{\eps})_{\#}\rho,(\LBPbase{\sigma_{\eps}}{\eps})_{\#}\rho) &\leq \alpha \|\rho/\sigma_{\eps}\|_{\infty}H(\SBstatic{\sigma_{\eps}}{\eps}| \LDstatic{\sigma_{\eps}}{\eps}) \\
    \Was{2}^{2}(\SB{\rho}{\eps}{},\LD{\rho}{\eps}{}) &\leq \alpha\theta_{\eps}^{-2} \|\rho/\sigma_{\eps}\|_{\infty}H(\SBstatic{\sigma_{\eps}}{\eps}| \LDstatic{\sigma_{\eps}}{\eps}).
\end{align*}
Moreover, with $\alpha$ can be chosen to be increasing in $\eps$. 
\end{lemma}

%Before proving this lemma, we pause to remark on the appearance of the $\theta_{\eps}$ in the second inequality. 
Recall that $\theta_{\eps} \in \mathbb{R} \setminus \{0\}$ is a constant chosen for integrability reasons outlined in \eqref{eq:simga-eps-onestep}-- it is simply an $O(1)$ term as $\eps \downarrow 0$. See Remark \ref{rmk:theta}.

\begin{proof}[Proof of Lemma \ref{lem:BP_upperbounds}]
Fix $T > 0$ and $\eps \in (0,T)$. To avoid unnecessary subscripts, set $\sigma := \sigma_{\eps}$. To start, we produce a coupling of $\LDstatic{\sigma}{\eps}$ and $\SBstatic{\sigma}{\eps}$ by applying a gluing argument (as in \cite[Lemma 7.6]{villani2021topics}, for example) to the optimal couplings between corresponding transition densities for for the Langevin diffusion and $\eps$-static $\Schro$ bridge. From this coupling, we obtain a coupling of the desired barycentric projections of $\LDstatic{\sigma}{\eps}$ and $\SBstatic{\sigma}{\eps}$. We then use the chain rule of relative entropy alongside the Talagrand inequality stated in \cite[Theorem 1.8(3)]{cattiaux-guillin-slc-14} to obtain the inequality stated in the Lemma. 

Let $(q_{t}(\cdot,\cdot),t \geq 0)$ denote the transition densities of the stationary Langevin diffusion associated to $\sigma$. By \cite[Theorem 1.8(3)]{cattiaux-guillin-slc-14}, there a constant $\alpha>0$ depending on $\lambda$ and $T$ such that following Talagrand inequality holds for each $x \in \mathbb{R}^{d}$
\begin{align*}
    \Was{2}^{2}(\eta,q_{\eps}(x,\cdot)) &\leq \alpha H(\eta |q_{\eps}(x,\cdot)) \text{ for all $\eta \in \cP(\mathbb{R}^{d})$.}
\end{align*}
Let $(p(x,\cdot), x \in \mathbb{R}^{d})$ denote the conditional densities of the $\eps$-static $\Schro$ bridge with marginals $\sigma$. Construct a triplet of $\mathbb{R}^{d}$-valued random variables $(X,Y,Z)$ in the following way. Let $X \sim \rho$, and let the conditional measures $(Y|X=x,Z|X=x)$ have the law of the quadratic cost optimal coupling between $p(x,\cdot)$ and $q_{\eps}(x,\cdot)$. Denote the law of $(X,Y,Z)$ on $\mathbb{R}^{d} \times \mathbb{R}^{d} \times \mathbb{R}^{d}$ by $\nu$. 
% For each $x \in \mathbb{R}^{d}$ let $u_{x}(dy,dz) \in \cP(\mathbb{R}^{d} \times \mathbb{R}^{d})$ denote the quadratic cost optimal coupling between $q_{\eps}(x,\cdot)$ and $p(x,\cdot)$. We now define a measure $v \in \cP(\mathbb{R}^{d}\times \mathbb{R}^{d} \times \mathbb{R}^{d})$ by defining for $\varphi \geq 0$ measurable and bounded
% \begin{align*}
%     \int \varphi dv &:= \int \left[\int \int \varphi(x,y,z) du_{x}(y,z)\right] \rho(x)dx. 
% \end{align*}
% It is now easier to proceed with the following probabilistic argument. Let $(X,Y,Z)$ be random variables with joint law $v$, then we o
Observe that $(\Exp{v}[Y|X],\Exp{v}[Z|X])$ is a coupling of $(\LBPbase{\sigma_{\eps}}{\eps})_{\#}\rho$ and $(\BPbase{\sigma_{\eps}}{\eps})_{\#}\rho$. Moreover, observe that $((1-\theta_{\eps}^{-1})X+\theta_{\eps}^{-1}\Exp{v}[Y|X],(1-\theta_{\eps}^{-1})X+\theta_{\eps}^{-1}\Exp{v}[Z|X])$ is a coupling of $\SB{\rho}{\eps}{}$ and $\LD{\rho}{\eps}{}$. Thus, to compute an upper bound of $\Was{2}^{2}(\SB{\rho}{\eps}{},\LD{\rho}{\eps}{})$, it is sufficient to compute $\Exp{}\|\Exp{v}[Y|X]-\Exp{v}[Z|X]\|^2$.

We first observe for each $x \in \mathbb{R}^{d}$ that by two applications of the (conditional) Jensen's inequality, the optimality of $(Y|X=x,Z|X=x)$, and the aforementioned Talagrand inequality 
\begin{align*}
    \|\Exp{v}[Y|X=x]-\Exp{v}[Z|X=x]\| 
    % &\leq \Exp{v}[\|Y-Z\||X=x]
    % &\leq \sqrt{\Exp{v}[\|Y-Z\|^2|X=x]} \\
    &\leq \Was{2}(q_{\eps}(x,\cdot),p(x,\cdot)) \leq \sqrt{\alpha H(p(x,\cdot) |q_{\eps}(x,\cdot))}.
\end{align*}
Hence, using the chain rule of relative entropy and the fact that the Langevin diffusion and $\Schro$ bridge both start from $\sigma$
\begin{align*}
    \Was{2}^{2}&((\BPbase{\sigma}{\eps})_{\#}\rho,(\LBPbase{\sigma}{\eps})_{\#}\rho) \leq \Exp{v}\|\Exp{v}[Y|X=x]-\Exp{v}[Z|X=x]\|^2 \\
    &\leq \int \alpha H(p(x,\cdot)|q_{\eps}(x,\cdot)) \rho(x)dx = \int \alpha H(p(x,\cdot)|q_{\eps}(x,\cdot)) \sigma(x)\frac{\rho(x)}{\sigma(x)}dx\\
    &\leq \alpha \|\rho/\sigma_{\eps}\|_{\infty}H(\SBstatic{\sigma_{\eps}}{\eps}{}|\LDstatic{\sigma_{\eps}}{\eps}{}).
\end{align*}
\end{proof}

Now, we revisit the bound from Theorem \ref{thm:ld-sb-general-g} and apply it to Lemma \ref{lem:LD_BP_covergence}.
% \GM{replace with a single constant}
\begin{lemma}\label{lem:bound-preserved-eps}
    Under Assumption \ref{assumptions:sigma-eps-onestep}, there exists a constant $K>0$ depending on $N$, $C$, and $D$
    % $K_1(\lambda)$ depending polynomially on $1/\lambda$ and $K_{2}(N)$ depending super-exponetially on $N$ 
    such that for all $\eps \in (0,\eps_0)$ 
    \begin{align}\label{eq:bound-preserved-eps}
        H(\SBstatic{\sigma_{\eps}}{\eps}|\LDstatic{\sigma_{\eps}}{\eps})+H(\LDstatic{\sigma_{\eps}}{\eps}|\SBstatic{\sigma_{\eps}}{\eps}) \leq \frac{\eps^2}{2}K\left(I(\sigma_{\eps})-\int_{0}^{1} I((\sigma_{\eps
    })_{s}^{\eps}) ds \right)^{1/2}  
    %\left(C+K_{1}(\lambda)K_{2}(N)\right)^{1/2}
    \end{align}
    In particular,  
    \[\lim\limits_{\eps \downarrow 0} \frac{1}{\eps^2}\left(H(\SBstatic{\sigma_{\eps}}{\eps}|\LDstatic{\sigma_{\eps}}{\eps})+H(\LDstatic{\sigma_{\eps}}{\eps}|\SBstatic{\sigma_{\eps}}{\eps})\right) = 0.\]
\end{lemma}
\begin{proof}
    Let $\eps \in (0,\eps_0)$ and in analogy with the Section \ref{sec:sb-ld} let $((\sigma_{\eps})_{s}^{\eps},s\in[0,1])$ denote the $\eps$-entropic interpolation from $\sigma_{\eps}$ to itself and set
    \begin{align*}
        \cU_{\eps} &= \frac{1}{8}\|\nabla \log \sigma_{\eps}\|^2 + \frac{1}{4}\Delta \log \sigma_{\eps}.
    \end{align*}
    By the assumed uniform subexponential bound on the $\sigma_{\eps}$ and uniform polynomial growth of $\cU_{\eps}$, there exists a constant $K > 0$ as specified in the Lemma statement such that     
    % Let $(X_{\eps},Y_{\eps}) \sim \SBstatic{\sigma_{\eps}}{\eps}$ and let $Z \sim N(0,\Id)$ be independent of the pairing. Again, $X_{s}^{\eps}$ as constructed in \eqref{eq:rv-entropic-inter} has law $(\sigma_{\eps})^{\eps}_{s}$. 
    % Thus, in terms of the constants given in Assumption \ref{assumptions:sigma-eps-onestep}
    % \begin{align*}
    %     \Exp{(\sigma_{\eps})_{s}^{\eps}}[\|\nabla \cU_{\eps}\|^2] &\leq C + C\Exp{}[\|sX_{\eps}+(1-s)Y_{\eps}+\sqrt{\eps s(1-s)}Z\|^{N}] \\
    %     &\leq C + C3^{N-1}\left(2 \Exp{}\|X_{\eps}\|^{N} + \eps^{N/2}\Exp{}\|Z\|^{N}\right).
    % \end{align*}
    % By Assumption \ref{assumptions:sigma-eps-onestep}(1), there exists a constant $K_{1} > 0$ depending on $D$ and $N$ such that $\sup\limits_{\eps \in (0,\eps_0)}\Exp{\sigma_{\eps}}\|X\|^{N} \leq K_1$ and $\Exp{}\|Z\|^{N} \leq K_{1}$.
    % % \begin{align*}
    % %     \left(\sup\limits_{\eps \in (0,\eps_0)}\Exp{\sigma_{\eps}}\|X\|^{N}\right), \Exp{}\|Z\|^{N} \leq K_{1}. 
    % % \end{align*}
    % Altogether, this establishes the existence of a constant $K > 0$ as specified in the Lemma statement such that 
    \begin{align}\label{eq:upper-bdd-varyingeps}
        \sup\limits_{\eps \in (0,\eps_0),s \in [0,1]}\Exp{(\sigma_{\eps})_{s}^{\eps}}[\|\nabla \cU_{\eps}\|^2] \leq K.
        %C+K(\lambda)3^{N-1}(2+\eps^{N/2})N^{N/2}
    \end{align}
    % Set $K_{2}(N) := 3^{N-1}(2+\eps_0^{N/2})N^{N/2}$. 
    Theorem \ref{thm:ld-sb-general-g} and \eqref{eq:upper-bdd-varyingeps} then establish \eqref{eq:bound-preserved-eps} when $\eps \in (0,\eps_0)$.
    % \begin{align*}
    %     H(\SBstatic{\sigma_{\eps}}{\eps}|\LDstatic{\sigma_{\eps}}{\eps})+H(\LDstatic{\sigma_{\eps}}{\eps}|\SBstatic{\sigma_{\eps}}{\eps}) \leq \frac{\eps^2}{2}K\left(I(\sigma_{\eps})-\int_{0}^{1} I((\sigma_{\eps
    % })_{s}^{\eps}) ds \right)^{1/2}. 
    % \end{align*}
    It now remains to show that
    \begin{align}\label{eq:varying-eps-cont-fisher}
        \lim\limits_{\eps \downarrow 0} \left[I(\sigma_{\eps})-\int_{0}^{1} I((\sigma_{\eps
    })_{s}^{\eps}) ds \right] = 0.
    \end{align}
    First, we claim that $((\sigma_{\eps})_{s}^{\eps}, \eps > 0)$ converges weakly to $ \sigma_{0}$ as $\eps \downarrow 0$ for all $s \in [0,1]$. As $(\sigma_{\eps}, \eps > 0)$ converges weakly to $\sigma_0$, the collection $\{\sigma_0\} \cup (\sigma_{\eps}, \eps > 0)$ is tight. This implies that $(\SBstatic{\sigma_{\eps}}{\eps}, \eps > 0)$ is tight. Let $\pi^{*}$ denote a weak subsequential limit-- we will not denote the subsequence. It is clear that $\pi^{*} \in \Pi(\sigma_0,\sigma_0)$. In the terminology of \cite{bgn-eot-gld}, the support of each $\SBstatic{\sigma_{\eps}}{\eps}$ is $(c,\eps)$-cyclically invariant, where $c(x,y)= \frac{1}{2}\|x-y\|^2$. As $\pi^*$ is a weak subsequential limit of the $(\SBstatic{\sigma_{\eps}}{\eps}, \eps > 0)$, the exact same argument presented in \cite[Lemma 3.1, Lemma 3.2]{bgn-eot-gld}, modified only superficially to allow for the different marginal in each $\eps$, establishes that $\pi^*$ is $c$-cyclically monotone. That is, $\pi^* = \SBstatic{\sigma_0}{0}$ is the quadratic cost optimal transport plan from $\sigma_0$ to itself. Moreover, this argument shows that along any subsequence of $\eps$ decreasing to $0$, there is a further subsequence converging to $\SBstatic{\sigma_0}{0}$. Hence, the limit holds without passing to a subsequence. 
    Construct a sequence of random variables $((X_{\eps},Y_{\eps}), \eps \geq 0)$ with $\text{Law}(X_{\eps},Y_{\eps}) = \SBstatic{\sigma_{\eps}}{\eps}$, and let $Z$ be a standard normal random variable independent to the collection. Fix $s \in [0,1]$, and recall that $\text{Law}(sX_{\eps}+(1-s)Y_{\eps}+\sqrt{\eps s(1-s)}Z) = (\sigma_{\eps})_{s}^{\eps}$. The Continuous Mapping Theorem gives that $sX_{\eps}+(1-s)Y_{\eps}+\sqrt{\eps s(1-s)}Z$ converges in distribution to $sX_{0}+(1-s)Y_{0}$ as $\eps \downarrow 0$. 
    % Next, by the Skorokhod Representation Theorem, construct a probability space on which there is a sequence of random variables $((X_{\eps},Y_{\eps}), \eps \geq 0)$ with $\text{Law}(X_{\eps},Y_{\eps}) = \SBstatic{\sigma_{\eps}}{\eps}$ and $(X_{\eps},Y_{\eps}) \to (X_0,Y_0)$ a.s.\ Let $Z$ be a standard normal random variable independent to the collection. Fix $s \in [0,1]$, then
    % \begin{align*}
    %     \text{Law}(sX_{\eps}+(1-s)Y_{\eps}+\sqrt{\eps s(1-s)}Z) = (\sigma_{\eps})_{s}^{\eps}. 
    % \end{align*}
    % Let $f: \mathbb{R}^{d} \to \mathbb{R}$ be continuous and bounded, by the Dominated Convergence Theorem
    % \begin{align*}
    %     \lim\limits_{\eps \downarrow 0}\Exp{}[f(sX_{\eps}+(1-s)Y_{\eps}+\sqrt{\eps s(1-s)}Z)] = \Exp{}[f(sX_0+(1-s)Y_0)].
    % \end{align*}
    As $\text{Law}(sX_0+(1-s)Y_0) = \sigma_0$, this establishes the weak convergence of $(\sigma_{\eps})_{s}^{\eps}$ to $\sigma_{0}$ as $\eps \downarrow 0$. By the lower semicontinuity of Fisher information with respect to weak convergence \cite[Proposition 14.2]{bobkov-fisher-22} and Fatou's Lemma
    \begin{align*}
        I(\sigma_0) &\leq \int_{0}^{1} \liminf\limits_{\eps \downarrow 0} I((\sigma_{\eps})_{s}^{\eps}) ds \leq \liminf\limits_{\eps \downarrow 0} \int_{0}^{1} I((\sigma_{\eps})_{s}^{\eps}) ds.
    \end{align*}
    Thus, \eqref{eq:varying-eps-cont-fisher} holds as Assumption \ref{assumptions:sigma-eps-onestep}(1) gives $\lim\limits_{\eps \downarrow 0} I(\sigma_{\eps}) = I(\sigma_0)$. 
\end{proof}

As the last piece of the proof, we show in the following lemma below that $\LD{\rho}{\eps}{}$ is an $o(\eps)$ approximation of the first step of the Euler iteration $\opt{\eps}{1}{\rho}$. 

% \begin{assumption}\label{assumption:LD_BP_convergence}
%     Let $\rho = e^{-g}$ be such that $I(\rho) < \infty$ and the associated Langevin semigroup $(G_{t},t\geq 0)$ is $L^{2}(\rho)$-continuous \cite[Definition 1.2.2(vi)]{bgl-markov}. 
% \end{assumption}

\begin{lemma}[]\label{lem:LD_BP_covergence}
For $\rho = e^{-g}$ satisfying Assumption \ref{assumptions:sigma-eps-onestep}, 
\[
\Was{2}^2(\opt{\eps}{1}{\rho}, \LD{\rho}{\eps}{}) \leq 2\theta_{\eps}^{-2}\|\rho/\sigma_{\eps}\|_{\infty}\eps^3 \Exp{\sigma_{\eps}}\|\nabla^2 \log \sigma_{\eps}\|^2_{HS}\,.
\]
% In particular, $\lim\limits_{\eps \downarrow 0} \frac{1}{\eps}\Was{2}(\opt{\eps}{1}{\rho}, \LD{\rho}{\eps}{}) = 0$.
\end{lemma}
\begin{proof}
The proof follows by producing an intuitive coupling of $\GA{\rho}{\eps}{}$ and $\LD{\rho}{\eps}{}$ and then applying a semigroup property stated in \cite[Equation (4.2.3)]{bgl-markov}. Let $(G_{s}^{\sigma_{\eps}}, s \geq 0)$ be the semigroup associated to the Langevin diffusion corresponding to $\sigma_{\eps}$, then for all $\eps > 0$, $s \geq 0$, and $f: \mathbb{R}^{d} \to \mathbb{R}$ in the domain of the Dirichlet form 
\begin{align}\label{eq:semigrp-bdd}
    \|G_{s}^{\sigma_{\eps}}f - f\|^{2}_{L^{2}(\sigma_{\eps})} \leq 2s\|\nabla f\|_{L^{2}(\sigma_{\eps})}^{2}. 
\end{align}
Let $X \sim \rho$, then by Dynkin's formula
\begin{align*}
    \LBPbase{\rho}{\eps}(X) &= X+\int_{0}^{\eps}G^{\sigma_{\eps}}_{s}(\theta_{\eps}\nabla \psi_{\eps})(X)ds
\end{align*}
and thus
    \begin{align*}
    \LD{\rho}{\eps}{} &= \text{Law}\left((1-\theta_{\eps}^{-1})X+\theta_{\eps}^{-1}\left(X+\int_{0}^{\eps}G^{\sigma_{\eps}}_{s}(\theta_{\eps}\nabla \psi_{\eps})(X)ds\right)\right)\\
    &= \text{Law}\left(X+\int_{0}^{\eps}G^{\sigma_{\eps}}_{s}(\nabla \psi_{\eps})(X)ds\right). 
    \end{align*}
Also by definition
\begin{align*}
    \opt{\eps}{1}{\rho} &= \text{Law}\left(X + \eps \nabla \psi_{\eps}(X)\right). 
\end{align*}
    Thus, we have the upper bound
    \begin{align*}
        \Was{2}^{2}(\LD{\rho}{\eps}{},\opt{\eps}{1}{\rho}) &\leq \Exp{} \left\|\left(X+\int_{0}^{\eps}G^{\sigma_{\eps}}_{s}(\nabla \psi_{\eps})(X)ds\right)-\left(X + \eps \nabla \psi_{\eps}(X))\right)\right\|^{2} \\
        &= \Exp{}\left\|\int_{0}^{\eps} \left(G_{s}^{\sigma_{\eps}}(\nabla \psi_{\eps})(X)-\nabla \psi_{\eps}(X) \right)ds \right\|^2 \\
        &\leq \eps \int_{0}^{\eps} \Exp{}\|G_{s}^{\sigma_{\eps}}(\nabla \psi_{\eps})(X)-\nabla \psi_{\eps}(X)\|^2 ds \\
        &= \eps^2 \int_{0}^{1} \|G_{\eps s}^{\sigma_{\eps}}(\nabla \psi_{\eps})-\nabla \psi_{\eps}\|_{L^{2}(\rho)}^2 ds,
    \end{align*}
    where the penultimate inequality follows from Jensen's inequality. Then, perform a change of measure and apply \eqref{eq:semigrp-bdd} to obtain
    \begin{align*}
        \Was{2}^{2}(\LD{\rho}{\eps}{},\opt{\eps}{1}{\rho}) &\leq 2\|\rho/\sigma_{\eps}\|_{\infty}\eps^3 \Exp{\sigma_{\eps}}\|\nabla^2 \psi_{\eps}\|_{HS}^{2}.
    \end{align*}
    % By Assumption \ref{assumptions:sigma-eps-onestep}, this establishes the desired convergence rate. 
    % By the $L^{2}(\rho)$ continuity of the semigroup, we have that $\lim\limits_{\eps \downarrow 0} \|G_{\eps s}(\nabla g)-\nabla g\|_{L^{2}(\rho)}^{2} = 0$ for all $s \in [0,1]$. Moreover, we observe that
    % \begin{align*}
    %     \|G_{\eps s}(\nabla g)-\nabla g\|_{L^{2}(\rho)}^2 &\leq 2\|G_{\eps s}(\nabla g)\|_{L^{2}(\rho)}^{2} + 2\|\nabla g\|_{L^{2}(\rho)}^{2} \\
    %     &\leq 4 I(\rho),
    % \end{align*}
    % using that the semigroup is an $L^{2}$ contraction. Hence, by the Dominated Convergence Theorem we have the desired
    % \begin{align*}
    %     \limsup\limits_{\eps \downarrow 0} \frac{1}{\eps^2}\Was{2}^{2}(\GA{\rho}{\eps}{}, \LD{\rho}{\eps}{}) &\leq \limsup\limits_{\eps \downarrow 0} \frac{1}{4} \int_{0}^{1} \|G_{\eps s}(\nabla g)-\nabla g\|_{L^{2}(\rho)}^2 ds = 0.
    % \end{align*}
\end{proof}

% \begin{corollary}\label{cor:sb-gf}
% Under \ref{assumption:LD_BP_convergence}, we have
% \begin{align*}
%     \lim\limits_{\eps \downarrow 0} \frac{1}{\eps}\Was{2}(\GF{\rho}{\eps}{},\SB{\rho}{\eps}{}) = 0. 
% \end{align*}
% \end{corollary}

% \section{Euler iterations of Wasserstein gradient flows}\label{sec:convex-analysis}
% \input{arxiv_version/ConvexAnalysis/main}

\section{Iterated $\Schro$ Bridge Scheme}\label{sec:schro-bridge-scheme}
In this section we present conditions under which suitably interpolated iterations of \eqref{sb-step} converge to the gradient flow of entropy and relative entropy (i.e.\ solutions to the parabolic heat equation and Fokker-Planck equation). Our results rely on the close approximation of $\eqref{sb-step}$ to explicit Euler iterations that we define later on. We also present a collection of sufficient conditions for which iterations of \eqref{sb-step} converge to the gradient flow of geodesically convex functionals. Roughly put, our assumptions are that (1) the one step approximation given in Theorem \ref{thm:one_step_convergence} holds along iterations of \eqref{sb-step} (which we call consistency) and (2) the error of this approximation does not accumulate too rapidly over a finite time horizon (which is implied by the condition we call contractivity in the limit). Although these conditions seem reasonable (in Section \ref{sec:examples} we show examples that satisfy them), it is not immediate how generally they hold. 
% The principle difficulty is in verifying that these conditions hold-- indeed, we require strong assumptions to obtain convergence to the heat flow and Fokker-Planck. 

% In this section we consider the explicit and implicit/JKO discretization schemes in the Wasserstein space that minimize a $\lambda$-geodesically convex functionals and prove their uniform convergence to the Wasserstein gradient flow curves as the discretization step size goes to zero. In the next section we will show how, under suitable assumptions, our iterated scheme using Schr\"odinger bridges can be used to approximate both these discrete solutions. 

\subsection{Preliminaries}
We set up some notation for subdifferential calculus in Wasserstein space.  Let $\cF: \cP_2(\R^d) \to (-\infty, \infty]$ be a function that  
%\GM{do we use this $\mathfrak{F}$ again somewhere?} 
is proper, lower-semicontinuous, and $\lambda$-geodesically convex functional for some $\lambda \in \mathbb{R}$.
% with $\lambda \geq 0$. \SP{do we need $\lambda \ge 0$ at this stage?} 
Let $D(\cF) = \{\rho \in \cP_{2}(\mathbb{R}^{d}): \cF(\rho) < +\infty\}$. The subdifferential set of $\cF$ at $\rho \in D(\cF)$ is denoted by ${\partial} \cF(\rho)$ (see \cite[Definition 10.1.1]{ambrosio2005gradient} for definition of Fréchet subdifferentials). With $D(\partial \cF) = \{\rho \in D(\cF): \partial \cF(\rho) \neq \emptyset\}$, we make the assumptions that $D(\partial \cF) \subset \ac(\R^d)$, the subset of absolutely continuous probability measures in $\cP_2(\R^d)$. Using Brenier's theorem \cite[Theorem 1.22]{santam2015ot}, this implies that there exists a unique optimal transport map from any $\rho \in D(\partial \cF)$ to any $\nu \in \cP_2(\R^d)$. We denote this map by $T_\rho^\nu$.
 %That is, all measures with non-empty subdifferential can be pushed to any measure in $\cP_2(\R^d)$. 
 Since $\cF$ is $\lambda$-geodesically convex, $\cF$ is a regular functional \cite[Definition 10.1.4]{ambrosio2005gradient} which implies that for all $\rho \in D(\partial \cF)$, the subdifferential set $\partial \cF(\rho)$ contains a unique element with minimum $L^2(\rho)$ norm. This unique minimum selection subdifferential is called the Wasserstein gradient hereon, denoted by $\WasDiff \cF(\rho)$, and belongs to the tangent space $\Tan{\rho}$ \cite[Definition 8.4.1]{ambrosio2005gradient}. We further assume that there exists an $\eps_0 >0$ such that the functional
 \[
 \nu \in \cP_2(\R^d) \mapsto \cF(\nu) + \frac{1}{2\eps} \Was{2}^2(\mu, \rho)
 \]
 admits at least one minimum point for any $\eps \in (0, \eps_0)$ and $\rho \in \cP_2(\R^d)$.
 
% For any $\rho \in D(\partial \cF)$ and $\xi \in \partial \cF(\rho)$, the $\lambda$-geodesic convexity of $\cF$ implies the following lower bound
 %\begin{equation}\label{eq:lambda_convexity}
 %    \cF(\nu) - \cF(\rho) \geq \int \iprod{\xi, T_{\rho}^\nu - \Id} d\rho + \frac{\lambda}{2} \Was{2}^2(\rho, \nu) \quad \forall \nu \in \cP_2(\R^d)\,.
 %\end{equation}
The two primary classes of examples that we consider later are (i) the (one half) entropy function $\cF(\rho)=\frac{1}{2} \Ent(\rho)$ for which the gradient flow is the probabilistic heat flow and (ii) (one half) relative entropy (AKA Kullback-Leibler divergence) with respect to a log-concave probability measure, i.e., $\cF(\rho)=\frac{1}{2}\kl{\cdot}{\nu}$, for some log-concave density $\nu$. These are well-known examples for which the above assumptions are satisfied. In Section \ref{sec:examples} we will cover examples that are not covered by our assumptions and yet, as we will show by direct computations, our iterated Schr\"odinger bridge scheme continues to perform well.

\subsection{Gradient Flows and their Euler Approximations}

Let $\left(\rho(t),\; t \in [0,T]\right)$ denote the Wasserstein gradient flow of $\cF$ during time interval $[0,T]$. We assume that it exists as an absolutely continuous curve that is the unique solution of the continuity equation 
\[
\partial_t \rho(t) + \nabla \cdot (v(t) \rho(t))=0, \quad \text{where}\; v(t)= - \WasDiff\cF(\rho(t))
\]
is assumed to be in $\Tan{\rho(t)}$.

Euler iterations of the Wasserstein gradient flow of $\cF$, whether it is explicit or implicit, is a sequence of elements in the Wasserstein space $\left( \rho_\eps(k),\; k=0,1,2,\ldots \right)$, starting at some $\rho_\eps(0)=\rho(0)$, given by an iterative map $\opt{\eps}{1}{\cdot}:\cP_2(\R^d) \rightarrow \cP_2(\R^d)$ of the form \eqref{eq:S-eps}, in the sense that $\rho_\eps(k)= \opt{\eps}{1}{\rho_\eps(k-1)}=\opt{\eps}{k}{\rho_\eps(0)}$, for $k\in \N$. The continuous-time approximation curve is given by the piecewise constant interpolation 
\begin{equation}\label{eq:first_order_approximation_flow}
    \rho_\eps(t) = \rho_\eps(k) \quad t\in[k\eps, (k+1)\eps).
\end{equation}

%We use the unique minimizer of $\surr[\rho]$ for each $\rho \in D(\partial \cF)$ to define 
%\begin{equation}\label{eq:opt_update}
%    \opt{\eps}{k+1}{\rho_0} = \argmin_{\nu \in \cP_2(\R^d)} \surr[\opt{\eps}{k}{\rho_0}](\nu)\,,
%\end{equation}
%starting from $\opt{\eps}{0}{\rho_0} = \rho_0 \in D(\partial \cF)$. 

% Now consider a generalized gradient flow for $\cF$ denoted by $(\rho(t), t\geq 0)$ with associated velocity field $(v(t), t\in[0,T])$ \GM{do these need to be tangent velocities?} and starting point $\rho_0\in D(\partial \cF)$. 

%Our main theorem in this section is true for any $\opt{\eps}{1}{\cdot}$ satisfying the following definition.  

%Below we give sufficient conditions that guarantees that $(\opt{\eps}{}{t}, t\in[0,T])$ uniformly approximates $(\rho(t), t\in[0,T])$ under the assumptions of contraction and consistency of the map $\rho \mapsto \opt{\eps}{1}{\rho}$. Assume that the approximating curve $S_\eps$ satisfies the definition below. 

%Our approach for iteratively minimizing $\cF$, starting from $\rho_0 \in D(\partial \cF)$, involves - 1) at each iterate $\opt{\eps}{k}{\rho_0}$, we construct an upper envelope functional $\nu\in\cP_2(\R^d) \mapsto \surr[\opt{\eps}{k}{\rho_0}](\nu)$ that uniformly majorizes $\cF$, agrees with $\cF$ at $\opt{\eps}{k}{\rho_0}$, and can be analytically minimized, and then 2) we update $\opt{\eps}{k}{\rho_0}$ by the minima of $\surr[\opt{\eps}{k}{\rho_0}]$. 
%The construction is described in detail below with the following necessary assumptions.

\begin{definition}[First order approximation]\label{def:first_order_approximation}
    The sequence $\left( \rho_\eps(k),\; k \in \N \right)$, or, equivalently, its continuous time interpolation $(\rho_\eps(t), t \ge 0)$ is called a \textit{first-order approximation} of a Wasserstein gradient flow $(\rho(t), t\ge 0)$ if, for any $T>0$, 
    $$\lim_{\eps \to 0} \sup_{t \in [0,T]} \Was{2}\of{\rho_{\eps}(t), \rho(t)} = 0.$$
\end{definition}
%\MA{We need to decide where the discrete scheme $(\rho_\eps(k), k\in[N_\eps])$ is a first-order approximation or the interpolation $(\rho_\eps(t), t\in[0,T]])$. In theorems 4 and 5, we say that the discrete scheme is the first-order approximation. }
We now wish to develop sufficient conditions under which $\opt{\eps}{1}{\cdot}$ gives a first order approximation. We say that $\opt{\eps}{1}{\cdot}$ is \textbf{consistent} if for any $T>0$,
\begin{align}\label{eq:consistency}
    \limsup_{\eps \to 0} \sup_{t \in [0,T-\eps]} \frac{1}{\eps}\Was{2}\of{\opt{\eps}{1}{\rho(t)}, \rho({t+\eps})}= 0\,,
\end{align}
and $\opt{\eps}{1}{\cdot}$ is a \textbf{contraction in the limit} if
% , for any $T>0$, there exists some $C, \eps^*>0$, such that for any $k \in [N_\eps]$, with $N_\eps= \lfloor T/ \eps \rfloor$, and any $\eps \in (0, \eps^*)$,
\begin{equation}\label{eq:contraction}
    \limsup\limits_{\eps \to 0}\sup\limits_{k \in [N_{\eps}]} \frac{1}{\eps}\left(\Was{2}\of{\opt{\eps}{1}{\opt{\eps}{k}{\rho(0)}}, \opt{\eps}{1}{\rho(k\eps)}} - \Was{2}\of{\opt{\eps}{k}{\rho(0)}, \rho(k\eps)}\right) \leq 0\,.
\end{equation}

\begin{remark}
Observe that if $\opt{\eps}{1}{\cdot}$ is a contraction, i.e.\ $\Was{2}(\opt{\eps}{1}{\rho},\opt{\eps}{1}{\sigma}) \leq \Was{2}(\rho,\sigma)$, then it is a contraction in the limit as well. Thus, this condition is a more general notion.  
\end{remark}

\begin{theorem}\label{thm:S_uniform_convergence}
 Suppose the map $\opt{\eps}{1}{\cdot}$ satisfies the two conditions - consistency and contraction in the limit. 
Then $\rho_\eps(\cdot)$ is a first-order approximation of the gradient flow $(\rho(t),  t \in [0,T])$.
\end{theorem}

%\GM{do we need this over all $\rho_1,\rho_2$, or just the iterates? The proof below only requires the contraction property along the pairs $((\rho(k\eps),\opt{\eps}{k}{\rho_0}), k \geq 1)$}

%\begin{assumption}\label{assumption:opt_contraction}
%     For any $\rho_1, \rho_2 \in D(\partial \cF)$, we have that
%    \[
%    \Was{2}\of{\opt{\eps}{1}{\rho_1}, \opt{\eps}{1}{\rho_2}} \leq C_\eps \Was{2}\of{\rho_1, \rho_2}\,,
%    \]
%    for some constant $0 < C_\eps < 1$ such that $1 - C_\eps \geq C\eps$ for some fixed constant $C>0$.
%    There exists a constant $\alpha >0$ such that for any $t \in [0,T)$, 
%    \[
%    \frac{\Was{2}\of{\opt{\eps}{1}{\rho(t)}, \rho(t+\eps)}}{\eps} \leq c\eps^{\alpha}\,,
%    \]
%    for some constant $c>0$ independent of $\eps$.
%\end{assumption}

%Suppose Definition \ref{assumption:opt_contraction}

\begin{proof}
    For any $t \in [k\eps, (k+1) \eps)$, by the triangle inequality and absolute continuity of $\rho(\cdot)$ 
    \[
    \Was{2}\of{\rho_\eps(t), \rho(t)} \leq \Was{2}\of{\opt{\eps}{k}{\rho(0)}, \rho({k\eps})} + \int_{k\eps}^t \norm{v(s)}_{L^2(\rho(s))} ds.
    \]
    The second term above converges to $0$ with $\eps$ because 
$\|v(t)\|_{L^{2}(\rho(t))} \in L^1_{\text{loc}}([0,+\infty))$. For the first term, consider the decomposition
\begin{align}\label{eq:triangle-split}
    \Was{2}\of{\opt{\eps}{k}{\rho(0)}, \rho\of{k\eps}} \leq \Was{2}\of{\opt{\eps}{k}{\rho(0)}, \opt{\eps}{1}{\rho\of{(k-1)\eps}}} + \Was{2}\of{\opt{\eps}{1}{\rho\of{(k-1)\eps}}, \rho\of{k\eps}}\,.
\end{align}
    Let $\delta > 0$, then by the assumptions of consistency and contraction in the limit, for all $\eps > 0$ small enough 
    % there exists an $\eps_0 > 0$ such that for all $\eps \in (0,\eps_0)$ and $k \in [N_{\eps}]$, where $N_{\eps} = \lfloor T/\eps \rfloor$, by consistency
    \begin{align*}
        \Was{2}\of{\opt{\eps}{1}{\rho\of{(k-1)\eps}}, \rho\of{k\eps}} < \delta \eps,
    \end{align*}
    and
    \begin{align*}
        \Was{2}\of{{\opt{\eps}{k}{\rho(0)}}, \opt{\eps}{1}{\rho((k-1)\eps)}} - \Was{2}\of{\opt{\eps}{k-1}{\rho(0)}, \rho((k-1)\eps)} < \delta \eps,
    \end{align*}
    respectively. 
    % \begin{align*}
    %     \Was{2}\of{{\opt{\eps}{k}{\rho(0)}}, \opt{\eps}{1}{\rho((k-1)\eps)}} - \Was{2}\of{\opt{\eps}{k-1}{\rho(0)}, \rho((k-1)\eps)} \text{ and } \Was{2}\of{\opt{\eps}{1}{\rho\of{(k-1)\eps}}, \rho\of{k\eps}}
    % \end{align*}
    % are both less than $\delta \eps$. 
    % Since both $\opt{\eps}{k-1}{\rho_0}$ and $\rho\of{(k-1)\eps}$ are in $D(\partial \cF)$, using the contraction property in Definition~\ref{def:first_order_approximation}, for all $\eps$ small enough, 
    % \[
    % \begin{split}
    % \Was{2}\of{\opt{\eps}{k}{\rho_0}, \opt{\eps}{1}{\rho\of{(k-1)\eps}}}&= \Was{2}\of{\opt{\eps}{1}{\opt{\eps}{k-1}{\rho_0}}, \opt{\eps}{1}{\rho\of{(k-1)\eps}}}\\
    % &\leq (1 + C\eps) \Was{2}\of{\opt{\eps}{k-1}{\rho_0}, \rho\of{(k-1)\eps}}.
    % \end{split}
    % \]
    Thus, \eqref{eq:triangle-split} gives the following recursion for $\eps \in (0,\eps_0)$ and $k \in [N_{\eps}]$ with $k \geq 1$
    \begin{align*}
        \Was{2}\of{\opt{\eps}{k}{\rho(0)}, \rho\of{k\eps}} \leq \Was{2}\of{\opt{\eps}{k-1}{\rho(0)}, \rho\of{(k-1)\eps}} + 2\delta \eps.
    \end{align*}
    Hence, recursively, 
    \[\limsup\limits_{\eps\to 0}\sup\limits_{k \in [N_{\eps}]} \Was{2}\of{\opt{\eps}{k}{\rho}, \rho\of{k\eps}} \leq  \Was{2}\of{\rho(0), \rho(0)}+\frac{T}{\eps}\cdot (2\delta \eps) = 2T\delta.\] As $\delta > 0$ was arbitrary, 
    % Also, given any $\delta >0$, for all small enough $\eps>0$, $\Was{2}\of{\opt{\eps}{1}{\rho\of{(k-1)\eps}}, \rho\of{k\eps}}< \delta \eps$. 
    % Thus, recursively, $\limsup_{\eps \rightarrow 0+}\sup_{k \in [N_\eps]}\Was{2}\of{\opt{\eps}{k}{\rho_0}, \rho\of{k\eps}} \leq \limsup_{\eps \rightarrow 0+} \left(\delta \eps \cdot \frac{e^{CT}}{\eps}\right)=e^{CT}\delta$.
    % Since $\delta$ is arbitrarily small, 
    this completes the proof. 
\end{proof}

    %Consequently, 
    %\[
    %\sup_{t \in [0,T]}\Was{2}\of{T_{\eps}(t), \rho\of{t}} \leq c \eps^\alpha \frac{1 - C_\eps^N}{1 - C_\eps} + \sup_{k \in [N_\eps]}\int_{(k-1) \eps}^{k \eps} \norm{v_s}^2 ds\,.
    %\]
    %Therefore, the desired convergence is obtained from the above upperbound.

%+ c \eps^\alpha\,

%The triangular argument, illustrated in Figure~\ref{fig:triangular_argument}, involves calculating the distance between $\rho\of{n\eps}$ and $\opt{\eps}{n}{\rho_0}$ for any $n \in [N_\eps]$. This computation is divided using a triangle inequality, with the pivot being the one-step approximation of $\rho\of{k\eps}$ using $\opt{\eps}{}{\rho\of{(k-1)\eps}}$.

%\begin{figure}
%    \centering
%    \includegraphics[width=0.75\textwidth]{figures/triangular_argument.png}
%    \caption{Triangle argument for uniform convergence of first-order approximations of generalized gradient flow.}
%    \label{fig:triangular_argument}
%\end{figure}

%For completion, we show in the Appendix that for $\eps>0$, the sequence of iterations $(\opt{\eps}{n}{\rho_0}, n\in\N)$ generated by \eqref{eq:opt_update} provide a valid optimization scheme for minimizing $\cF$. 

%We now show that both the explicit Euler and the implicit Euler schemes in the Wasserstein space satisfy Definition \ref{def:first_order_approximation}. 

%We argue that both of these schemes are first-order approximations of the Wasserstein gradient flow of $\cF$.

\subsubsection*{Explicit Euler Scheme}
The explicit Euler scheme in the Wasserstein space can be obtained as an iterative geodesic approximation of the Wasserstein gradient flow minimizing $\cF$ for the choice of 
\begin{equation}\label{eq:GA}\tag{EE}
    \opt{\eps}{1}{\rho} = \of{\Id - \eps \WasDiff \cF \of{\rho}}_\# \rho.
\end{equation}

%For adequate convergence of explicit Euler steps to the functional minimum, we further assume that $\cF \in \mathfrak{F}_L$, i.e. $\cF$ is $L$-Lipschitz smooth. 
%\SP{Can't assume that globally. Reformulate over the iterations.}

%This assumption gives the upperbound \eqref{eq:L_smoothness} on $\cF$. We consider the following upper envelope mapping on $\cF$ with $\eps \leq L^{-1}$ 
%\begin{equation}\label{eq:GA_upper_envelope}
%    \surr[\rho]: \nu \in \cP_2(\R^d) \mapsto  \cF(\rho) + \int \iprod{\WasDiff \cF(\rho), T_{\rho}^\nu - \Id} d\rho + \frac{1}{2\eps} \Was{2}^2\of{\rho, \nu} \quad \text{ for any } \rho\in D(\partial \cF)\,.
%\end{equation}
%\SP{This envelope will not hold globally.}

%The pointwise minimum of $\surr[\rho]$ occurs at $T_\rho^\nu = (\Id - \eps \WasDiff \cF(\rho))_\# \rho$. Since we have assumed that $\cF \in \mathfrak{F}_L$, this minimum is achieved and is absolutely continuous for small enough $\eps$. Therefore, with the upper envelope mapping \eqref{eq:GA_upper_envelope}, Assumption~\ref{assumption:upper_envelope} is satisfied. Consequently, the explicit Euler scheme, starting from $\rho_0\in D(\partial \cF)$ and stepsize $\eps < L^{-1}$, is given by the following recursive formulation

The explicit Euler scheme need not always give a first-order approximation sequence to a Wasserstein gradient flow. However, the assumptions in the statements of Theorems \ref{corollary:entropy-with-sb} and \ref{corollary:kl-with-sb} provides sufficient conditions for which \eqref{eq:GA} satisfies consistency and contraction in the limit.
% Later in Lemmas \ref{lem:heat-flow-sb-eps2}, \ref{lem:kl-sb-eps2}, and in Step 1 of the proof of Theorem~\ref{corollary:entropy-with-sb} we give sufficient conditions for which the scheme indeed satisfies consistency and contraction in limit (Definition~\ref{def:first_order_approximation}). 
Consequently, Theorem~\ref{thm:S_uniform_convergence} gives the uniform convergence of the scheme.

\subsubsection*{Implicit Euler Scheme}
The implicit Euler scheme or, equivalently the JKO scheme \cite{jko98}, in the Wasserstein space can be defined using the following map
\begin{equation}\label{eq:IE}\tag{IE}
     \opt{\eps}{1}{\rho} = \argmin_{\nu\in\cP_2(\R^d)} \left[\cF(\nu) + \frac{1}{2\eps}\Was{2}^2\of{\nu, \rho}\right].
\end{equation}
We quickly verify that when $\rho \in \cP_{2}^{ac}(\mathbb{R}^{d})$ and $\eps > 0$, there exists $u_{\eps}: \mathbb{R}^{d} \to \mathbb{R}$ such that $\opt{\eps}{1}{\rho} = (\Id + \eps \nabla u_{\eps})_{\#}\rho$. As $\rho \in \cP_{2}^{ac}(\mathbb{R}^{d})$, by Brenier's Theorem there exists a unique gradient of a convex function $\nabla \psi_{\eps}$ that pushforwards $\rho$ to $\opt{\eps}{1}{\rho}$. Setting $u_{\eps}(x) := \frac{1}{\eps}\left(\psi_{\eps}(x)-\frac{1}{2}\|x\|^{2}\right)$, it holds that $\opt{\eps}{1}{\rho} = (\Id + \eps u_{\eps})_{\#}\rho$ as desired. 
% \GM{spell out how we can see this as pushforward in desired form}
%The optimality of $\opt{\eps}{1}{\rho}$ implies that $\eps^{-1}(T_{\opt{\eps}{1}{\rho}}^\rho - \Id) \in \partial \cF(\opt{\eps}{1}{\rho})$ \cite[Lemma 10.1.2]{ambrosio2005gradient} and consequently, $\opt{\eps}{1}{\rho} \in D(\partial \cF)$. 
The implicit Euler scheme starting from $\rho(0) \in D(\partial \cF)$ is given by $\rho_\eps(k)=\opt{\eps}{k}{\rho(0)}$.
%\begin{equation}\label{eq:PP}\tag{Implicit Euler}
%    \opt{\eps}{n+1}{\rho_0} = \argmin_{\rho\in{\cP_2(\R^d)}} \off{\cF(\rho) + \frac{1}{2\eps}\Was{2}^2\of{\rho, \opt{\eps}{n}{\rho_0}}}\,.
%\end{equation}
%If $(\rho(t), t\geq 0)$ is the gradient flow for minimizing KL divergence, then the above is the celebrated JKO algorithm \cite{jko98} in optimal transport. With $\lambda >0$, Theorem~\ref{thm:MM_linear_convergence} gives us the following linear rate of convergence for sequence generated by \eqref{eq:PP}
%\[
%\Was{2}^2\of{\opt{\eps}{n}{\rho_0}, \rho^*} \leq \of{\frac{1}{1 + 2\lambda \eps}}^{n} \Was{2}^2({\rho_0, \rho^*})\,.
%\]
%We prove in Lemma~\ref{lem:PP_contraction} that \eqref{eq:PP} satisfy Definition~\ref{assumption:opt_contraction} with contraction coefficient $C=\lambda$.
When $\cF$ is geodesically convex, it is known that Theorem~\ref{thm:S_uniform_convergence} holds through  \cite[Theorem 4.0.7, Theorem 4.0.9, Theorem 4.0.10]{ambrosio2005gradient}. We do not need to verify Definition \ref{def:first_order_approximation}.

In the next section we introduce our iterated Schrödinger bridge scheme that approximates the explicit scheme, under suitable assumptions.
% It is known that the JKO scheme produces a first-order approximation to the Wasserstein gradient flow $(\rho(t), t\in[0,T])$ \cite[Theorem 4.0.7, Theorem 4.0.10]{ambrosio2005gradient}. Hence, Theorem \ref{thm:S_uniform_convergence} holds. 

%\subsection{Examples of Generalized Gradient Flow and their First-order Approximation}
%\input{arxiv_version/ConvexAnalysis/examples}
% \subsection{Setup}
% \input{arxiv_version/SBScheme/setup}

% \subsection{One Step Convergence}
% \input{arxiv_version/SBScheme/one-step}

\subsection{Uniform Convergence}

We now prove two theorems giving conditions under which iterations of \eqref{sb-step} provide a first order approximation to the gradient flow of $\frac{1}{2}\Ent$ and $\frac{1}{2}H(\cdot|\nu)$. 
% The proofs of these Theorems follow an identical structure: first, we establish that the explicit Euler iterations possess the consistency and contractivity in the limit properties given in Theorem \ref{thm:S_uniform_convergence} (see Lemma \ref{lem:heat-flow-sb-eps2} and Lemma \ref{lem:kl-sb-eps2}). 
The proof of both theorems rely on an appeal to a general result in Lemma \ref{thm:SB_uniform_convergence} that gives sufficient conditions for \eqref{sb-step} iterates to approximate Euler iterates. In principle, Lemma \ref{thm:SB_uniform_convergence} can be applied with $\opt{\eps}{1}{\cdot}$ given by implicit or explicit Euler iterations, but our Theorems \ref{corollary:entropy-with-sb} and \ref{corollary:kl-with-sb} both only use explicit Euler iterations. 

To simplify the bookkeeping of the constants introduced in Assumptions \ref{assumptions:sigma-eps-onestep}, we define two classes of probability densities. For $C,D,M > 0$, $N \geq 1$, and $\lambda \in \mathbb{R}$ set
\begin{align}\label{eq:pi-CDN}
    \Pi_{1}(C,D,N) &:= \{\rho \in \cP_{2}^{ac}(\mathbb{R}^{d}): \|\rho\|_{\psi_{1}} \leq D, \nabla \cU_{\rho} \text{ exists, } \|\cU_{\rho}(x)\|^{2} \leq C(1+\|x\|^{N})  \},
\end{align}
and
\begin{align}\label{eq:pi-lamM}
    \Pi_{2}(\lambda,M) &:= \{\rho \in \cP_{2}^{ac}(\mathbb{R}^{d}) \; \text{satisfying Assumption \ref{assumption:LD_SB_relative_entropy}, } -\nabla^{2} \log \rho \geq \lambda \Id, \Exp{\rho}\|\nabla^{2} \log \rho\|_{HS}^{2} \leq M \}.
\end{align}
Lastly, recall the RHS of Theorems \ref{thm:ld-sb-general-g} and \ref{thm:one_step_convergence} and that $(\rho_{s}^{\eps}, s \in [0,1])$ is the entropic interpolation from $\rho$ to itself. We define
\begin{align}\label{eq:Delta-int-fisher-info}
    \Delta_{\eps}(\rho) := I(\rho) - \int_{0}^{1}I(\rho_{s}^{\eps}) ds.
\end{align}
We now state the following theorem for the (probabilistic) heat flow:
\begin{align}\label{eq:heat-eqn}
    \partial_{t}\rho(t) = \frac{1}{2}\Delta_{x} \rho(t)=\frac{1}{2}\nabla \cdot\left( (\nabla \log \rho(t)) \rho(t)\right),
\end{align}
starting at $\rho(0)$. Let $\opt{\eps}{1}{\cdot}$ denote the explicit Euler iteration map for the heat flow.
That is, for $\rho \in \cP_{2}^{ac}(\mathbb{R}^{d})$
  \begin{align}\label{eq:explicit-euler-heat}
      \opt{\eps}{1}{\rho} := \left(\Id - \frac{\eps}{2}\nabla \log \rho\right)_{\#}\rho.
  \end{align}

\begin{theorem}\label{corollary:entropy-with-sb}
Assume that $\rho(0)$ is such that
  \begin{enumerate}
      % \item[(1)] $\limsup\limits_{\eps \downarrow 0}\sup\limits_{k \in \left[N_{\eps}\right]} I(\opt{\eps}{k}{\rho_0}), \limsup\limits_{\eps \downarrow 0 }\sup\limits_{k \in \left[N_{\eps}\right]} I(\SB{\rho_0}{\eps}{k}) < +\infty$
      \item[(1)] $\limsup\limits_{\eps \downarrow 0} \sup\limits_{k \in [N_{\eps}]} \Delta_{\eps}\left(\SB{\rho(0)}{\eps}{k}\right) = 0$
      \item[(2)] There exists $\eps_0 > 0$ and constants $C, D > 0$, and $N \geq 1$ such that for all $\eps \in (0,\eps_0)$, $(\opt{\eps}{k}{\rho(0)}, k \in \{0\} \cup [N_{\eps}]) \cup (\SB{\rho(0)}{\eps}{k}, k \in [N_{\eps}]) \subset \Pi_{1}(C,D,N)$.
      % \begin{enumerate}
      % % \item $-\nabla^{2} \log \rho \geq \lambda \Id$ and $\Exp{\rho}\|\nabla^2 \log \rho\|^{2}_{HS} \leq M$, and 
      % \item $\|X\|_{\psi_{1}} \leq D$ for $X \sim \rho$, and 
      % \item $\| \nabla \cU_{\rho}(x)\|^{2} \leq C(1+\|x\|^{N})$, where $\cU_{\rho}$ is the function in \eqref{eq:harmoniccharacteriestic} corresponding to $\rho$.
      % \end{enumerate}
      \item[(3)] There exists $\lambda \in \mathbb{R}$ and $M > 0$ such that for all $\eps \in (0,\eps_0)$ and $k \in [N_{\eps}]$, $(\SB{\rho(0)}{\eps}{k}, k \in [N_{\eps}]) \subset \Pi_{2}(\lambda,M)$. 
      % all  $\rho \in $ $-\nabla^{2} \log \SB{\rho(0)}{\eps}{k} \geq \lambda \Id$ and $\Exp{\SB{\rho(0)}{\eps}{k}}\|\nabla^{2} \log \SB{\rho(0)}{\eps}{k}\|_{HS}^{2} \leq M$.
  \end{enumerate}
Then $(\SB{\rho_0}{\eps}{k}, k \in [N_{\eps}])$ is a first order approximation of the (probabilistic) heat flow. 
% That is, 
% \begin{align*}
%     \lim\limits_{\eps \downarrow 0} \sup\limits_{k \in \left[N_{\eps}\right]} \Was{2}(\rho(k\eps),\SB{\rho}{\eps}{k}) = 0. 
% \end{align*}
\end{theorem}
% \begin{remark}
%     Known to be true for the diffusion when $\rho(0)$ is subexponential, but not known/difficult to verify along iterations of both explicit Euler and SB
% \end{remark}

%This is the observation of \cite[Theorem 1]{sander_22} that we now put on mathematical footing. More precisely, in the language of \cite{sander_22} we have set $W_{K}= W_{Q} = \Id$ and $W_{V} = -\Id$. 

\begin{remark}
Theorem \ref{corollary:entropy-with-sb} lays down sufficient conditions under which \cite[Theorem 1]{sander_22} holds. 
Part of the assumptions are on the Schr\"odinger bridge iterates which are assumed to be \text{nice enough} such that our one step approximations hold uniformly. In Section \ref{sec:examples}, we show that these assumptions are indeed satisfied when $\rho(0)$ is Gaussian. We are currently unable to verify that these natural conditions must be true once sufficiently stringent conditions (say, strong log concavity and enough smoothness) are assumed on the initial measure $\rho(0)$.  
\end{remark}

%While the additional assumptions in Theorem \ref{corollary:entropy-with-sb} appear stringent, in Section \ref{sec:examples} we show that they are indeed satisfied in the Gaussian case. \GM{should we give some words on how the assumptions along iterations are unsatisfactory and ideally we should be able to verify that they hold, but this is difficult to do in complete generality? }

% The proof of this Corollary relies on a verification of Assumption \ref{assumption:uniform_convergence} in this setting. To simplify this lengthy argument we begin with the following lemma.  

We now state an analogous theorem for the Wasserstein gradient flow of relative entropy. Fix a reference measure $\nu = e^{-h} \in \cP_{2}(\mathbb{R}^{d})$ satisfying assumptions delineated in Theorem \ref{corollary:kl-with-sb} below.

Let $\opt{\eps}{1}{\cdot}$ denote the explicit Euler iteration for the Fokker-Planck equation
\begin{align}\label{eq:fokker-planck}
    \partial_{t} \rho(t) = \frac{1}{2}\nabla\cdot\left((\nabla \log \rho(t) + \nabla h)\rho(t)\right).
\end{align}
That is, for $\rho \in \cP_{2}^{ac}(\mathbb{R}^{d})$
  \begin{align}\label{eq:explicit-euler-kl}
      \opt{\eps}{1}{\rho} := \left(\Id - \frac{\eps}{2}\nabla \log \rho-\frac{\eps}{2}\nabla h\right)_{\#}\rho.
  \end{align}
For $\eps > 0$ and $k \in [N_{\eps}]$, let $\sigma_{\eps}^{k} \in \cP_{2}^{ac}(\mathbb{R}^{d})$ denote the surrogate measure obtained in the application of $\SB{\cdot}{\eps}{1}$ to $\SB{\rho(0)}{\eps}{k-1}$, and let $\theta_{\eps}^{k} \in \mathbb{R} \setminus \{0\}$ denote the corresponding integrability constant from \eqref{eq:simga-eps-onestep}. That is, $\sigma_{\eps}^{k}$ is the probability density proportional to $(\rho^{-1}\nu)^{\theta_{\eps}^{k}}$, for $\rho=\SB{\rho(0)}{\eps}{k-1}$. 

% We make the following regularity assumptions on $h$ and the solution to \eqref{eq:fokker-planck}
% \begin{assumption}[Regularity of Fokker-Planck]\label{assumption:regularity-of-FP}
% Assume that $h$ is convex and such that for a class $\mathscr{F} \subset \cP_{2}^{ac}(\mathbb{R}^{d})$ \GM{specify later}
% \begin{align*}
%     \limsup\limits_{\eps \downarrow 0} \sup\limits_{\rho \in \mathscr{F}} \frac{1}{\eps} \Was{2}(\opt{\eps}{1}{\rho},\rho(\eps)) = 0.
% \end{align*}
% \end{assumption}

\begin{theorem}\label{corollary:kl-with-sb}
Let $h$ be convex, and assume that $\rho(0)$ is such that
  \begin{enumerate} 
      % \item[(1)] $\limsup\limits_{\eps \downarrow 0}\sup\limits_{k \in \left[N_{\eps}\right]} I(\opt{\eps}{k}{\rho_0}), \limsup\limits_{\eps \downarrow 0 }\sup\limits_{k \in \left[N_{\eps}\right]} I(\SB{\rho_0}{\eps}{k}) < +\infty$
      \item[(1)] $\limsup\limits_{\eps \downarrow 0} \sup\limits_{k \in [N_{\eps}]} \Delta_{\eps}(\sigma_{\eps}^{k}) = 0.$
      \item[(2)] There exists $C,D, M > 0$, $N \geq 1$, and $\lambda \in \mathbb{R}$ such that for all $\eps \in (0,\eps_0)$, $(\sigma_{\eps}^{k}, k \in [N_{\eps}]) \subset \Pi_{1}(C,D,N) \cap \Pi_{2}(\lambda,M)$, and  $\|\SB{\rho(0)}{\eps}{k}/\sigma_{k}^{\eps}\|_{\infty}, (\theta_{\eps}^{k})^{-2} \leq M$ for all $k \in [N_{\eps}]$.
      % \item[(3)] There exists $\lambda \in \mathbb{R}$ and $C, D, M > 0$ and $N \geq 1$ such that for all $\eps \in (0,\eps^*)$ and $k \in [N_{\eps}]$, $\sigma_{k}^{\eps}$ satisfies Assumption \ref{assumption:LD_SB_relative_entropy} as well as
      % \begin{itemize}
      %     \item[(a)] $-\nabla^{2} \log \sigma_{k}^{\eps} \geq \lambda \Id$, $\Exp{\sigma_{k}^{\eps}}\|\nabla^{2} \log \sigma_{k}^{\eps}\|_{HS}^{2} \leq M$,
      %     \item[(b)] $\|\SB{\rho(0)}{\eps}{k}/\sigma_{k}^{\eps}\|_{\infty} \leq M$, and $(\theta_{\eps}^{k})^{-2} \leq M$, and
      %     \item[(c)] $\|\nabla \cU_{\sigma_{k}^{\eps}}(x)\|^{2} \leq C(1+\|x\|^{N})$, where $\cU_{\sigma_{k}^{\eps}}$ is the function in \eqref{eq:harmoniccharacteriestic} corresponding to $\sigma_{k}^{\eps}$.
      % \end{itemize}
      \item[(3)] (Regularity of Fokker-Planck) The collection $\mathscr{F} = (\opt{\eps}{k}{\rho(0)}, k \in \{0\} \cup [N_{\eps}]) \cup (\SB{\rho(0)}{\eps}{k}, k \in [N_{\eps}])$ is such that
      \begin{align*}
          \lim\limits_{\eps \downarrow 0} \sup\limits_{\rho \in \mathscr{F}} \frac{1}{\eps} \Was{2}(\opt{\eps}{1}{\rho},\rho(\eps)) = 0.
      \end{align*}
  \end{enumerate} 
  Then $(\SB{\rho_0}{\eps}{k},k \in [N_{\eps}])$ is a first order approximation of the gradient flow of (one-half) KL divergence, i.e.\ the solution to \eqref{eq:fokker-planck}. 
% That is, 
% \begin{align*}
%     \lim\limits_{\eps \downarrow 0} \sup\limits_{k \in \left[N_{\eps}\right]} \Was{2}(\rho(k\eps),\SB{\rho}{\eps}{k}) = 0. 
% \end{align*}
\end{theorem}
% \begin{remark}
%     We show that Assumption 2(b) holds for a single, nice enough measure in Lemma \ref{lem:kl-sb-eps2}, but we must assume this holds uniformly to prove the Theorem.
% \end{remark}
\begin{remark}\label{rmk:OUisgood}
    These assumed regularity of the Fokker-Planck is just requiring that the convergence given in \cite[Proposition 8.4.6]{ambrosio2005gradient} holds uniformly along the measures in $\mathscr{F}$. That is, if we start Fokker-Planck from any of the measures $\rho$ in $\mathscr{F}$, the explicit Euler iteration is a good enough approximation for small times, uniformly in $\rho$. In Lemma \ref{lem:heat-flow-sb-eps2} below we show that this statement is true for the heat flow since its corresponding transition kernel is a Gaussian convolution. It is also satisfied by the flow corresponding to the standard Ornstein-Uhlenbeck (OU) process for which the transition kernel is again a Gaussian convolution. In fact, Lemma \ref{lem:heat-flow-sb-eps2} continues to hold in this case with minor modifications to the proof. 
\end{remark}

Observe the difference in the assumptions required for Theorems   \ref{corollary:entropy-with-sb} and \ref{corollary:kl-with-sb}. The new assumptions in Theorem \ref{corollary:kl-with-sb} are regarding the surrogate measures $\sigma_\eps^k$ corresponding to the Schr\"odinger iterate $(\SB{\rho(0)}{\eps}{k}$.
%compared to the Assumptions in Theorem \ref{corollary:entropy-with-sb}, we are also now making assumptions on the surrogate measures.
% as opposed to the $(\SB{\rho_0}{\eps}{k}, k \in [N_{\eps}])$. 
The reason for this difference is that for the heat flow $\sigma_\eps^k=(\SB{\rho(0)}{\eps}{k}$ and $\theta_\eps^k=-1$, for every $k\in [N_\eps]$ and any $\eps>0$. Hence the assumptions simplify.

% \GM{now give discussion about how the $\|\rho/\sigma\|_{\infty}$ bound is trickier in this case. However, even when this assumption is violated, the result still holds in examples section.}

Before proving Theorems \ref{corollary:entropy-with-sb} and \ref{corollary:kl-with-sb}, we present general conditions under which the iterates of (\ref{sb-step}), called SB scheme, converge to the gradient flow of a functional $\mathcal{F}$. For a general case, it is difficult to show that a continuous time interpolation of the SB scheme is a first order approximation of the gradient flow (Definition~\ref{def:first_order_approximation}). The following general conditions ensure uniform convergence to the gradient flow via Euler iterations. We then restrict ourselves to the specific cases when $\cF$ is entropy and relative entropy. 

\begin{assumption} \label{assumption:uniform_convergence}
Fix $\mathcal{F}: \cP_{2}(\mathbb{R}^{d}) \to (-\infty,+\infty]$ be a geodesically convex functional and fix $T > 0$. With initial measure $\rho(0) \in \cP_{2}^{ac}(\mathbb{R}^{d})$, let $(\rho(t), t \geq 0)$ denote the gradient flow of $\mathcal{F}$. Let $\left(\rho_{\eps}(k\eps) = \opt{\eps}{k}{\rho_0}, k \in \left[N_{\eps}\right]\right)$ be a sequence of Euler iterations that gives rise to a first order approximation of the gradient flow of $\cF$ satisfying Definition \ref{def:first_order_approximation}. Moreover, assume that:
\begin{itemize}
    \item[(1)] there is \textbf{consistency} of $\opt{\eps}{1}{\cdot}$ along the $\of{\SB{\rho_0}{\eps}{k}, k \in [N_\eps]}$, meaning
    \begin{align}\label{eq:sb-consistency}
        \limsup\limits_{\eps \downarrow 0}\sup\limits_{k \in [N_\eps]}\frac{1}{\eps}\Was{2}\left(\SB{\rho_0}{\eps}{k},\opt{\eps}{1}{\SB{\rho_0}{\eps}{k-1}}\right) = 0.
    \end{align}
    \item[(2)] There is \textbf{contraction in the limit}, meaning 
    % $C,\eps^* > 0$ such that for any $\eps \in (0,\eps^*)$ and $k \in [N_\eps]$, with $N_\eps= \lfloor T/ \eps \rfloor$,
    \begin{align}\label{eq:sb-contraction}
        \limsup\limits_{\eps \to 0}\sup\limits_{k \in [N_{\eps}]}\frac{1}{\eps}\left(\Was{2}\left(\opt{\eps}{k}{\rho_0},\opt{\eps}{1}{\SB{\rho_0}{\eps}{k-1}}\right)-\Was{2}\left(\opt{\eps}{k-1}{\rho_0},\SB{\rho_0}{\eps}{k-1}\right)\right) &\leq 0. 
    \end{align}
\end{itemize}
\end{assumption}
These assumptions are reminiscent of \eqref{eq:consistency} and \eqref{eq:contraction}. Hence, the proof of the following lemma is very similar to that of Theorem \ref{thm:S_uniform_convergence}.

% \MA{We are using $N_\eps = \lfloor T \eps^{-1} \rfloor$.}\GM{thank youuu}
\begin{lemma}\label{thm:SB_uniform_convergence}
Under Assumption \ref{assumption:uniform_convergence}, 
\begin{align*}
    \lim\limits_{\eps \downarrow 0} \sup\limits_{k \in \left[N_{\eps}\right]} \Was{2}(\rho(k\eps),\SB{\rho_0}{\eps}{k}) = 0.
\end{align*}
\end{lemma}

\begin{proof}
    % This argument is identical to that of the proof of Theorem \ref{thm:S_uniform_convergence}. 
    To start, observe that 
    \begin{align}\label{eq:tri-arg}
    \sup_{k \in [N_{\eps}]}\Was{2}(\SB{\rho_0}{\eps}{k},\rho(k\eps)) &\leq \sup_{k \in [N_{\eps}]}\Was{2}(\SB{\rho_0}{\eps}{k},\opt{\eps}{k}{\rho_0}{}) + \sup_{k \in [N_{\eps}]}\Was{2}(\opt{\eps}{k}{\rho_0}{},\rho(k\eps)).
\end{align}
By Theorem \ref{thm:S_uniform_convergence}, the rightmost term vanishes as $\eps \downarrow 0$. For the remaining term, for $k \in [N_{\eps}]$ split it as
\begin{align*}
    \Was{2}(\SB{\rho_0}{\eps}{k},\opt{\eps}{k}{\rho_0}) &\leq \Was{2}\left(\SB{\rho_0}{\eps}{k},\opt{\eps}{1}{\SB{\rho_0}{\eps}{k-1}}\right) + \Was{2}\left(\opt{\eps}{1}{\SB{\rho_0}{\eps}{k-1}},\opt{\eps}{k}{\rho_0}{}\right).
    % &\leq \Was{2}\left(\SB{\rho_0}{\eps}{k},\opt{\eps}{1}{\SB{\rho_0}{\eps}{k-1}}\right) + (1+C\eps)\Was{2}\left(\SB{\rho_0}{\eps}{k-1},\opt{\eps}{k-1}{\rho_0}{}\right), 
\end{align*}
By Assumption \ref{assumption:uniform_convergence}(2), one obtains the exact same recursion as in Theorem \ref{thm:S_uniform_convergence}. Hence, the same argument gives that $\lim\limits_{\eps \downarrow 0} \sup\limits_{k \in [N_{\eps}]}\Was{2}(\SB{\rho_0}{\eps}{k},\opt{\eps}{k}{\rho_0}{}) = 0$, completing the Lemma. 
\end{proof}

We now begin the proof of Theorem \ref{corollary:entropy-with-sb}, which follows as an application of Lemma \ref{thm:SB_uniform_convergence} once we verify Assumption \ref{assumption:uniform_convergence} holds in this setting. The verification is a lengthy argument that relies on a tight approximation at small time of the explicit Euler step given in \eqref{eq:explicit-euler-heat} and the heat flow, which we present in Lemma below.

\begin{lemma}\label{lem:heat-flow-sb-eps2}
    Let $\rho = e^{-g} \in \Pi_{1}(C,D,N)$.
    % be such that $g$ satisfies Assumption \ref{assumption:LD_SB_relative_entropy} for some constants $C > 0$ and $n \geq 1$, and $\|X\|_{\psi_{1}} \leq D$ for $X \sim \rho$.
    % \begin{itemize}
    %     \item[(1)] $g$ satisfies Assumption \ref{assumption:LD_SB_relative_entropy} for some constants $C > 0$ and $n \geq 1$.
    %     \item[(2)] There exists $\lambda > 0$ such that $\nabla^{2} g \geq \lambda \Id$ \GM{instead, just that $\rho$ is subexponential}.
    % \end{itemize}
    Let $\opt{\eps}{1}{\cdot}$ be as defined in \eqref{eq:explicit-euler-heat}, and $(\rho(t),t \geq 0)$ the solution to \eqref{eq:heat-eqn} with $\rho(0) =\rho$. Fix $\eps_0 > 0$, then there is a constant $K > 0$ depending on $C$, $D$, and $n$ such that for any $\eps \in (0,\eps_0)$
    \begin{align*}
        \Was{2}(\opt{\eps}{1}{\rho},\rho(\eps)) \leq K \eps^{2}.
    \end{align*}
\end{lemma}
\begin{proof}
    We consider two evolving particle systems. Let $(\rho(t,\cdot), t\geq 0)$ denote the (probabilistic) heat flow starting from $\rho$, and set $g(t,x) = -\log \rho(t,x)$. 
    % As we have assumed log concavity of $\rho$, each $g(t,\cdot)$ is a convex function.
    Let $X, Y: [0,\infty) \times \mathbb{R}^{d} \to \mathbb{R}^{d}$ be two particle systems with the same initial configuration $X_0 \sim \rho$ solving the ODE (letting $\dot{X}_{t}$ denote time derivative)
\begin{align*}
    \dot{X}_{t} = \frac{1}{2}\nabla g(t,X_{t}) \text{ and } \dot{Y}_{t} = \frac{1}{2}\nabla g(0,X_0). 
\end{align*}
Observe that $X_{\eps} \sim \rho(\eps)$ and $Y_{\eps} \sim \opt{\eps}{1}{\rho}$. Now, define $G: [0,+\infty) \times \mathbb{R}^{d} \to \mathbb{R}^{d}$ by $G(t,x) = X(t,x) - Y(t,x)$. Observe that $G(0,\cdot) = \dot{G}(0,\cdot) = 0$, giving that
\begin{align}\label{eq:iter-int-for-G}
    G(t) &= \int_{0}^{t}\int_{0}^{s} \ddot{G}(u) du ds. 
\end{align}
By the chain rule and the fact that $(\rho(t,\cdot),t \geq 0)$ solves the heat equation, observe that 
%\SP{Have we used dot for time derivative elsewhere?}
\begin{align*}
    \ddot{G}(t,X_{t}) &= \frac{1}{2}\partial_{t}\nabla_{x}g(t,X_{t}) \\
    &= \frac{1}{2}\left(\nabla_{x}^{2}g(t,X_{t})\dot{X}_{t}-\frac{1}{2}\nabla_{x} \left[\frac{\Delta_{x} \rho(t,X_{t})}{\rho(t,X_{t})}\right]\right).
\end{align*}
In analogy with Section \ref{sec:sb-ld}, define
\begin{align*}
    \cU(t,u) &= \frac{1}{8}\|\nabla_{x} g(t,u)\|^2-\frac{1}{4}\Delta_{x} g(t,u), 
\end{align*}
and compute that
\begin{align*}
    \frac{1}{4}\nabla_{x}\left[\frac{\Delta \rho(t,u)}{\rho(t,u)}\right] = \nabla_{x}\left(\frac{1}{8}\|\nabla_{x} g(t,u)\|^2+\cU(t,u)\right) = \frac{1}{4}\nabla_{x}^{2} g(t,u) \nabla g(t,u)+\nabla_{x}\cU(t,u).  
\end{align*}
Altogether, it now holds that
\begin{align}\label{eq:intacc}
    \ddot{G}(t,X_t) = -\nabla_{x}\cU(t,X_t).
\end{align}
For a fixed $\eps_0 > 0$, there uniform constant $C' > 0$ and $n' \geq 1$ (potentially slightly modified from $C$ and $n$) such that for all $t \in [0,\eps_0]$, $\|\nabla_{x} \cU(t,x)\|^{2} \leq C'(1+\|x\|^{n'})$. Let $X_0 \sim \rho$ and $Z \sim N(0,\Id)$ be independent, then for $t \geq 0$ the random variable $X_0 + \eps Z$ has law $\rho(t)$. Moreover, as $\|X_0 + t Z\|_{\psi_1} \leq D + t \|Z\|_{\psi_{1}}$, the $\rho(t)$ have bounded subexponential norm over $t \in [0,\eps_0]$.
% As each $\rho(t)$ has a LSI constant bounded over $t \in [0,\eps]$, $\|X\|$ is sub-Gaussian for each $X \sim \rho(t)$ \cite[Proposition 5.4.1]{bgl-markov}. As the second moments remain bounded along the heat flow \GM{will add details later},
Hence, there exists a constant $K'$ as described in the Lemma such that for $t \in [0,\eps_0]$, $\Exp{\rho(t)}\|\nabla_{x} \cU(t,X)\|^{2} \leq K'$. Hence, applying Jensen's inequality twice to \eqref{eq:iter-int-for-G} gives
\begin{align*}
    \Was{2}^{2}(\opt{\eps}{1}{\rho},\rho(\eps)) &\leq \Exp{}\|G(\eps,X_{0})\|^{2} \leq \int_{0}^{\eps}\int_{0}^{s} \eps s \Exp{}[\|\ddot{G}(u,X_{u})\|^{2}] du \, ds \leq \frac{\eps^{4}}{3}K',
\end{align*}
establishing the Lemma. 
% \GM{say the $\rho(t)$ all uniformly satisfy Assupmtion 1, don't chase constants}
% there exists a potentially modified constant $C' > 0$ such that $\sup\limits_{t \in [0,\eps_0]}\|\nabla_{x} \cU(t,x)\|^{2} \leq C' +C'\|x\|^{n}$. 
% retains its boundedness on $B(0,R)$ and polynomial growth of power $n$ on $\mathbb{R}^{d} \setminus B(0,R)$, with perhaps a modified constant $C'$ that holds for $t \in [0,\eps_0]$ and can be obtained as a simple function of $C$ \GM{is there a better say to say this?}. That is, $\sup\limits_{t \in [0,\eps_0]} \|\cU(t,x)\|^{2} \leq C' + C'\|x\|^{n}$.
% As $\rho$ is $\lambda$-strictly convex, each $\rho(t,\cdot)$ is $\lambda$-strictly convex and thus $\Exp{\rho(t)}\|X\|^{n} \leq K_{1}n^{n/2}$ for a constant $K_{1}$ depending polynomially on $\lambda^{-1}$. Altogether, there exists a constant $K'$ as described in the Lemma such that $\Exp{}[\|\ddot{G}(t,X_t)\|^2] \leq K'$. Applying Jensen's inequality twice to \eqref{eq:iter-int-for-G} gives
% \begin{align*}
%     \Was{2}^{2}(\opt{\eps}{1}{\rho},\rho(\eps)) &\leq \int_{0}^{\eps}\int_{0}^{s} \eps s \Exp{}[\|\ddot{G}(u,X_{u})\|^{2}] du \, ds \leq \frac{\eps^{4}}{3}K',
% \end{align*}
\end{proof}

We now prove Theorem \ref{corollary:entropy-with-sb}. 
\begin{proof}[Proof of Theorem \ref{corollary:entropy-with-sb}]
It suffices to verify that Assumption \ref{assumption:uniform_convergence} holds in this setting. 

\textbf{Step 1: Explicit Euler Iterations are Consistent and Contractive in the Limit.} Since $\rho(0)$ is assumed to be in $\Pi_1(C,D,N)$, by Lemma \ref{lem:heat-flow-sb-eps2}, along the heat flow starting at $\rho(0)$, there is a uniform constant $K > 0$ such that, for any $\eps > 0$ small enough and for all $t \in [0,T-\eps]$, $\Was{2}(\opt{\eps}{1}{\rho(t)}, \rho(t+\eps)) \leq K\eps^2$. 

Next, we establish contractivity in the limit. For a measure $\mu \in \cP(\mathbb{R}^{d})$ let $[\mu](t)$ denote the time marginal of the heat flow \eqref{eq:heat-eqn} at time $t$ started from $\mu$. By contractivity of the heat flow with respect to initial data in the Wasserstein-2 metric \cite[Theorem 11.1.4]{ambrosio2005gradient},
\begin{align*}
    \Was{2}([\opt{\eps}{k}{\rho(0)}](\eps),\rho((k+1)\eps)) \leq \Was{2}(\opt{\eps}{k}{\rho(0)},\rho(k\eps)).
\end{align*}
It then follows from triangle inequality
\begin{align*}
    \Was{2}\left(\opt{\eps}{1}{\opt{\eps}{k}{\rho(0)}},\opt{\eps}{1}{\rho(k\eps)}\right) &\leq \Was{2}(\opt{\eps}{k}{\rho(0)},\rho(k\eps)) + \Was{2}(\opt{\eps}{1}{\opt{\eps}{k}{\rho(0)}},[\opt{\eps}{k}{\rho(0)}](\eps))\\
    % \Was{2}\left(\opt{\eps}{1}{\opt{\eps}{k}{\rho(0)}},[\opt{\eps}{k}{\rho(0)}](\eps)\right)\\
    % &+\Was{2}([\opt{\eps}{k}{\rho(0)}](\eps),\rho((k+1)\eps))+\Was{2}(\opt{\eps}{1}{\rho(k\eps)},\rho((k+1)\eps)) \\
    &+\Was{2}(\opt{\eps}{1}{\rho(k\eps)},\rho((k+1)\eps)).
\end{align*}
Hence, by assumptions in the theorem statement, for $\eps$ small enough there is a universal constant $K > 0$ such that for all $k \in [N_{\eps}]$
\begin{align*}
     \left(\Was{2}\left(\opt{\eps}{1}{\opt{\eps}{k}{\rho(0)}},\opt{\eps}{1}{\rho(k\eps)}\right) - \Was{2}\left(\opt{\eps}{k}{\rho(0)},\rho(k\eps)\right)\right) \leq 2K\eps^2, 
\end{align*}
establishing contraction in the limit. This demonstrates that the explicit Euler iterations converge to the heat flow.

\textbf{Step 2: SB Iterations are Consistent and Contractive in the Limit.} To start, fix $\eps > 0$ and $k \geq 1$. To compute $\SB{\rho_0}{\eps}{k}$ from $\SB{\rho_0}{\eps}{k-1}$, we must obtain the surrogate measure for \eqref{sb-step}. Pick $\theta_{\eps} = -1$, then the surrogate measure is the measure itself, that is, $\sigma_{\eps} = \SB{\rho_0}{\eps}{k-1}$. This in turn gives $\|\SB{\rho_0}{\eps}{k-1}/\sigma_{\eps}\|_{\infty} = 1$. For consistency, by Theorem \ref{thm:one_step_convergence} and the assumptions in the statement of Theorem \ref{corollary:entropy-with-sb}, for small enough $\eps > 0$ there is a universal constant $K > 0$ such that for all $k \in [N_{\eps}]$ with $k \geq 1$
\begin{align*}
    \Was{2}\left(\SB{\rho(0)}{\eps}{k},\opt{\eps}{1}{\SB{\rho(0)}{\eps}{k-1}}\right) &\leq \eps K \left(I(\SB{\rho(0)}{\eps}{k-1})-\int_{0}^{1}I((\SB{\rho(0)}{\eps}{k-1})_{s}^{\eps})ds\right)^{1/4}.
\end{align*}
By the first assumption in the Theorem statement, consistency holds. Contractivity in the limit holds by the triagle inequality (as in the previous argument) and applying Lemma \ref{lem:heat-flow-sb-eps2}.
\end{proof}

The proof of Theorem \ref{corollary:kl-with-sb} is completely analogous to the proof of Theorem \ref{corollary:entropy-with-sb} and thus omitted. The only modification is that we replace the role of Lemma \ref{lem:heat-flow-sb-eps2} with the assumed regularity of Fokker-Planck equation in Theorem \ref{corollary:kl-with-sb}(3). See also Remark \ref{rmk:OUisgood}.
% The only difference is that instead of applying the triangle inequality with densities along the heat flow, we apply the triangle inequality
% with densities along the Fokker-Planck equation \eqref{eq:fokker-planck}. As $\sigma$ is log-concave, solutions to the Fokker-Planck equation possess the same contraction with respect to the initial data in the Wasserstein-2 metric as the heat flow (in fact, the contraction is improved) \cite[Theorem 11.1.4]{ambrosio2005gradient}. Hence, under the stated assumptions in Theorem \ref{corollary:kl-with-sb}, the exact same arguments for consistency and contraction in the limit go through.   

\section{Examples}\label{sec:examples}
%In this section, we demonstrate the step-by-step application of \eqref{sb-step} iterations to approximate a special family of absolutely continuous curves in the Wasserstein space. Wasserstein gradient flows of functionals on Wasserstein space form a family of absolutely continuous curves in which, at least formally, the velocity field is given by the gradient of a function (the first variation of the functional). Thus, \eqref{sb-step} applies rather naturally to such curves. 
% It is known that Wasserstein gradient flows emerge from the optimization problem of minimizing a functional defined on the Wasserstein space. 
% For geodesically convex functionals, the Wasserstein gradient flow converges to the minimum at linear or sub-linear rates, depending on the convexity parameter. 
%To illustrate the method, we approximate the gradient flows of convex functionals with favorable properties, whose uniform convergence properties were proven in Theorem~\ref{thm:SB_uniform_convergence}.

In this section we compute explicit examples of iterated Schr\"{o}dinger bridge approximations to various gradient flows and other absolutely continuous curves on the Wasserstein space. Since Schr\"odinger bridges are known in closed form for Gaussian marginals (\cite{janati2020}), our computations are limited to Gaussian models. First, we consider gradient flows of two convex functionals - 1) entropy and 2) KL divergence. Second, while not being gradient flows and aligning with our theoretical analysis in Section~\ref{sec:schro-bridge-scheme}, we also consider time reversal of these gradient flows to showcase the curiously nice approximation properties of the \eqref{sb-step} scheme for these absolutely continuous curves in the Wasserstein space. These latter examples have recently become important in applications. For example, see \cite{song2020score} for an application to diffusion based generative models in machine learning.

\subsection{Explicit Euler discretization of the gradient flow of Entropy} \label{sec:example_entropy}

We show via explicit calculations that \eqref{sb-step} iterations are close to the explicit Euler approximation of the heat flow $(\rho(t), t\in [0,T])$ starting from a Gaussian density. Here $\cF = \frac{1}{2} \Ent$ and the continuity equation for the gradient flow of $\cF$ is $\partial_t \rho(t) + \nabla \cdot \of{v(t) \rho(t)} = 0$ where $v(t) = -\frac{1}{2}\nabla \log \rho(t)$. The explicit Euler iterations, $(\rho_\eps(k), k\in[N_\eps])$, where $N_\eps := \lfloor T\eps^{-1} \rfloor$, are given by $\rho_\eps(k+1) = \of{\Id + \eps v_\eps(k)}_\#\rho_\eps(k)$ where $v_\eps(k) = -\frac{1}{2} \nabla \log \rho_\eps(k)$.
% Finally, assuming the SB iterates satisfy Assumption~\ref{assumptions:sigma-eps-onestep} at each step, choosing $\theta_\eps = -1$, the surrogate measure is $\sigma_\eps(k) = \SB{\rho_0}{\eps}{k}$.

For $\rho_\eps(0) = \rho = \cN(0,\eta^2 I_d)$, $\eta^2>0$, we will calculate the \eqref{eq:GA} and \eqref{sb-step} iterates and show that the SB step is an $\cO(\eps^2)$ approximation of the explicit Euler step.
Recall that the form of $\Schro$ bridge with equal marginals $\rho$ is explicitly given by \cite[Theorem 1]{janati2020} as
\[
\SBstatic{\rho}{\eps} = \cN\of{\begin{pmatrix}
    0\\0
\end{pmatrix}, \eta^2\begin{pmatrix}
    I_d & C^\eps_{\eta^2}I_d\\
      C^\eps_{\eta^2}I_d & I_d 
\end{pmatrix}} \quad \text{ where } \quad  C^\eps_{\eta^2} = \frac{1}{\eta^2}\of{\sqrt{\eta^4 +\frac{\eps^2}{4}} - \frac{\eps}{2}}\,.
\]
From this the barycentric projection can be computed as
\begin{align}\label{eq:gaussian_bp}
    \BPbase{\rho}{\eps}(x) = \E{\SBstatic{\rho}{\eps}}{Y|X=x} = C^\eps_{\eta^2}x\,.
\end{align}

For the choice of $\rho$, $\SB{\rho}{\eps}{1}=(2\Id - \BPbase{\rho}{\eps})_\#\rho = \cN\of{0, \left(2 - C^\eps_{\eta^2}\right)^2 \eta^2 I_d}$, and the corresponding $\rho_{\eps}(1) = \of{\Id - \frac{\eps}{2} \nabla \log \rho}_\#\rho = \cN\of{0, \of{1 + \frac{\eps}{2\eta^2}}^2 \eta^2 I_d}$.

% Suppose $\rho_0 = \cN(0,\eta^2 I_d)$, then the gradient flow is given by $\rho(t) = N(0,(\eta^2+t) I_d)$. 
% By the computation in \eqref{eq:bp-to-ent-potents}, we then observe that
% \begin{align*}
%     \SB{N(0,\eta^2)}{\eps}{} &= \cN(0,(2-C_{\eta^2}^{\eps})^{2}\eta^2 I_d)\,.
% \end{align*}
% The explicit Euler update is
% \[
% \rho_\eps(1) = \of{\Id - \eps \nabla \log \rho_0}_\# \rho_0 = \cN\of{0, \of{1 + \frac{\eps}{2\eta^2}}^2 \eta^2}\,.
% \]
Now let's check the sharpness of Theorem \ref{thm:one_step_convergence}. By the explicit form of the $\Was{2}$ distance between two Gaussians 
\[
\Was{2}\of{\rho_\eps(1), \SB{\rho}{\eps}{1}} = d \abs{(2 - C^\eps_{\eta^2}) \eta - (1 + \eps \eta^{-2})\eta} = d\eta \abs{\sqrt{1 + \frac{\eps^2}{4\eta^4}} - 1}\,.
\]
Therefore, $\Was{2}\of{\rho_\eps(1), \SB{\rho}{\eps}{1}} = \cO(\eps^2)$, which an order of magnitude sharper than Theorem~\ref{thm:one_step_convergence}.

% \begin{align*}
%     \of{1 + \frac{\eps}{\eta^2}}^2 \eta^2 - (2-C_{\eta^2}^{\eps})^2 \eta^2 &= (4\eta^2 + \eps) \of{1 - \sqrt{1 + \frac{\eps^2}{4\eta^2}}} - \frac{\eps^2}{4\eta^2}\\
%     \implies (4\eta^2 + \eps)\of{\frac{\eps^2}{8\eta^2} - \frac{\eps^4}{64 \eta^4}} - \frac{\eps^2}{4\eta^2} &\leq 
%     (4\eta^2 + \eps) \of{1 - \sqrt{1 + \frac{\eps^2}{4\eta^2}}} - \frac{\eps^2}{4\eta^2}\\
%     &\leq (4\eta^2 + \eps)\of{\frac{\eps^2}{8\eta^2} + \frac{\eps^6}{1024 \eta^6}} - \frac{\eps^2}{4\eta^2}\,.
% \end{align*}
% leading to one-step bound
% \begin{align*}
%     \Was{2}(\SB{\rho_0}{\eps}{},\rho_{\eps})
%     &= \frac{1}{2\eta}\max\offf{\frac{\eps^2}{4\eta^2}, \frac{\eps^3}{8\eta^4}}\,.
% \end{align*}
 % This can be ascribed to the simplicity of the problem. 
 Since both \eqref{sb-step} and \eqref{eq:GA} schemes operate via linear pushforwards, then all steps of these schemes are mean-zero Gaussian distributed. For the SB scheme, the surrogate measure at each step is the same as the current measure, i.e.\ $\sigma_\eps(k ) = \SB{\rho}{\eps}{k}$ for all $k\in[N_\eps]$.
 Let $\SB{\rho_0}{\eps}{k} = \cN\of{0, \alpha_k^2 I_d}$ and $\rho_{\eps}(k) = \cN\of{0, \beta_k^2 I_d}$, then the iterates evolve through the following recursive relationship 
\[
\alpha_{k+1}^2 = \of{2 - C^\eps_{\alpha_k^2}}^2 \alpha_k^2 \quad \text{ and } \quad \beta_{k+1}^2 = \of{1 + \frac{\eps}{2\beta_k^2}}^2 \beta_k^2\,. 
\]

To show that $(\rho_\eps(k), k\in[N_\eps])$ and $(\SB{\rho}{\eps}{k}, k\in[N_\eps])$ are first-order approximations (Definition~\ref{def:first_order_approximation}) of the Wasserstein gradient flow $(\rho(t), t\in[0,T])$, we will show that the assumptions of Theorem~\ref{corollary:entropy-with-sb} are satisfied.

% One can easily verify that $\of{\rho_\eps(k), k\in[N_\eps]}$ satisfies consistency \eqref{eq:consistency} and contraction in limit \eqref{eq:contraction} via Step 1 of the proof of Theorem~\ref{corollary:entropy-with-sb}. Consequently, using Theorem~\ref{thm:S_uniform_convergence}, $\of{\rho_\eps(k), k\in[N_\eps]}$ is a first-order approximation of $\of{\rho(t), t\in[0,T]}$.

First, let us consider assumption (1) of Theorem~\ref{corollary:entropy-with-sb}.
The formula for the time marginals of the dynamic Schr\"odinger bridge between $\sigma = \cN(0, \eta^2)$ and itself at time $t=1/2$ is available from \cite{gentil2017analogy} and given by
\begin{equation}\label{eq:entropic_interpolation}
    (\sigma)^\eps_{1/2} = \cN\of{0, \of{\frac{\delta_\eps^2}{4(1+\delta_\eps)}  + 1} \eta^2} \quad \text{ where } \delta_\eps = \frac{1}{2}\of{\eps -2 + \sqrt{4 + \eps^2}}\,.
\end{equation}
Also, it known that for $\sigma = \cN(0, \eta^2 I_d)$, the Fisher information $I(\sigma) = d \eta^{-2}$.
Therefore, for $\sigma_\eps(k) = \SB{\rho}{\eps}{k} = \cN(0, \alpha_k^2)$, we have
\begin{equation}
    I(\SB{\rho}{\eps}{k}) - I((\SB{\rho}{\eps}{k})_{1/2}^\eps) = \frac{d}{\alpha_k^2}\of{1 - \of{\frac{\delta_\eps^2}{4(1 + \delta_\eps)} + 1}^{-1}}  = \frac{d \delta_\eps^2}{\alpha_k^2\of{\delta_\eps^2 + 4(1 + \delta_\eps)} }\,.
\end{equation}
Since $C^\eps_{\nu^2} \leq 1$ for any $\nu^2 > 0$, $(\alpha_k^2, k\in[N_\eps])$ is monotonically increasing and bounded from below by $\eta^2$. Hence, 
\[
\sup_{k\in [N_\eps]} \left( I(\SB{\rho}{\eps}{k}) - I((\SB{\rho}{\eps}{k})_{1/2}^\eps) \right) \le I(\rho) - I(\rho_{1/2}^\eps) = \frac{d \delta_\eps^2}{\eta^2\of{\delta_\eps^2 + 4(1 + \delta_\eps)} }\,.
\]
Since the RHS above converges to zero as $\eps \rightarrow 0+$, we satisfy Assumption (1) of Theorem~\ref{corollary:entropy-with-sb} with rate $\cO(\eps^2)$. 

Now, we consider assumption (2) of Theorem~\ref{corollary:entropy-with-sb}. One can be convinced that for the Gaussian case, if both sequences $(\alpha_k^2, k\in[N_\eps])$ and $(\beta_k^2, k\in[N_\eps])$ are bounded from below and above, then the assumption (2) is satisfied. Trivially, $(\alpha_k^2, k\in[N_\eps])$ and $(\beta_k^2, k\in[N_\eps])$ are bounded from below by $\eta^2$ because they are monotonically increasing. Take $\eps \leq \eta^2$. Using the identity $1 - x \leq \sqrt{1 + x^2}$ for all $x\geq 0$, we have that $2 - C^\eps_{\alpha_k^2} \leq 1 + \eps \alpha_{k}^{-2}$ and therefore $\alpha_{k+1}^2 \leq \alpha_k^2 + 2\eps + \eps^2/\alpha_k^2 \leq \alpha_k^2 + 3\eps$. Consequently, $(\alpha_k^2, k\in[N_\eps])$ is bounded from above by $\eta^2 + 3T$. Now, again using that $\eps \leq \eta^2$, $\beta_{k+1}^2 \leq \beta_k^2 + 5\eps/4$, and therefore $(\beta_k^2, k\in[N_\eps])$ is bounded from above by $\eta^2 + 5T/4$. Step 1 of the proof of Theorem~\ref{corollary:entropy-with-sb} gives uniform convergence of $(\rho_\eps(k), k\in[N_\eps])$ and step 2 gives the uniform convergence of $(\SB{\rho}{\eps}{k}, k\in[N_\eps])$.

\subsection{Explicit Euler discretization of the gradient flow of Kullback-Leibler}\label{sec:example_kl}
Now let us consider the Kullback-Leibler (KL) divergence functional $\kl{\cdot}{\nu}$ where $\nu = e^{-g}$ for $\nabla g(x) = x$. This is the Fokker-Planck equation corresponding to the Ornstein-Uhlenbeck process. 
% is a log-concave measure \GM{update when assumptions finalized} with density of form $e^{-g}$. Then here $\cF = \frac{1}{2}\kl{\cdot}{e^{-g}}$ and the continuity equation for its gradient flow $(\rho(t); t\in[0,T])$ is $\partial_t \rho(t) + \nabla \cdot \of{v(t)\rho(t)} = 0$ where $v(t) = -\frac{1}{2}\nabla \log \rho - \frac{1}{2}\nabla g$. To repeat, the explicit Euler iterations, $(\rho_\eps(k); k\in[N_\eps])$, where $N_\eps := \lfloor T\eps^{-1} \rfloor$, are given by $\rho_\eps(k+1) = \of{\Id + \eps v_\eps(k)}_\#\rho_\eps(k)$ where $v_\eps(k) =  -\frac{1}{2} \nabla \log \rho_\eps(k) - \frac{1}{2}{\nabla g}$.

%The explicit Euler step, starting from $\rho_\eps(0) = \rho $, is given by $\rho_\eps(1) = \of{\Id + \eps v_\eps(0)}_\#\rho$ where $v_\eps(0) = -\frac{1}{2} \nabla \log \rho - \frac{1}{2}{\nabla g}$. 
% which, in the particle picture, follows the Ornstein-Uhlenbeck SDE 
% \begin{equation}\label{eq:OU}
% dX_t = -\frac{1}{2}X_t dt + dB_t
% \end{equation}
% where $(B_t, t\geq 0)$ denotes standard $d$-dimensional Brownian motion.  
Now let us write the \eqref{sb-step} step at $\rho \in \cP_{2}^{ac}(\mathbb{R}^{d})$. For this functional, the selection of $\theta \in \mathbb{R} \setminus \{0\}$ in (\ref{sb-step}) becomes more delicate. The choice of $\theta$ for the surrogate measure $\sigma = \sigma_\eps(0)$ is determined by the integrability of  
\begin{align*}
    \exp\left(-2\theta \left(\frac{1}{2}\log \rho + \frac{1}{2}g\right)\right) = (\rho e^{g})^{-\theta}.
\end{align*}
If $\rho$ is log-concave, then since $g$ is convex the sign of $\theta$ depends on a comparison of the log-concavity of $\rho$ and $\nu$. Suppose $\rho$ is more log-concave than $\nu$, then $\theta$ may be taken to be $-1$. 
% % and $\sigma_0$ is the stationary distribution for the Langevin diffusion 
% % \[
% % dX_t = \frac{1}{2}\of{\nabla \log \rho_0(X_t)) + \nabla g(X_t)}dt + dB_t %= -v(\rho_0)(X_t)dt + dB_t\,.
% % \]
Consequently, the (\ref{sb-step}) step is $\SB{\rho}{\eps}{1} = (2\Id - \BPbase{\sigma}{\eps})_{\#}\rho$.
% % This choice is motivated by the string of approximations
% % \begin{align*}
% %     \of{\Id + \nabla f_{\eps, \sigma}}_\#\rho_0 &\approx \of{\Id - \frac{1}{2}\eps \nabla \log \sigma}_\#\rho_0\\
% %     &= \of{\Id + \eps \of{-\frac{1}{2} \log \rho_0 -\frac{\Id}{2}}}_\#\rho_0\\
% %     &= \of{\Id + \eps \nabla v(\rho_0)}_\#\rho_0\,.
% % \end{align*}
On the other hand, if $\nu$ is more log-concave than the density $\rho$, then $\theta$ may be taken to be $1$. 
% % In this case, $\sigma_0$ is stationary distribution for the Langevin diffusion process
% %  \[
% %  dX_t = -\frac{1}{2}\of{\nabla \log \rho_0(X_t)) + \nabla g(X_t)}dt + dB_t %= \nabla v(\rho_0)(X_t) dt + dB_t
% %  \]
In this case, the \eqref{sb-step} step is $\SB{\rho}{\eps}{1} = (\BPbase{\sigma}{\eps})_{\#}\rho$.
% % This sign convention can also be inspired by the result in Section~\ref{sec:introduction}, where we prove that the law of $(X_0, X_\eps)$, denoted by $u_{\eps, \sigma}$ is close to the Scr\"{o}dinger bridge $\pi_{\eps, \sigma}$ with same marginals $\sigma$ in relative entropy. Notice this motivates the string of equality
% % \[
% % x + \eps \nabla v(\rho_0)(x) \approx \E{u_{\eps, \sigma}}{X_\eps|X_0 = x} \approx  \E{\pi_{\eps, \sigma}}{X_\eps|X_0 = x} = x - \nabla f_{\eps, \sigma}(x)\,.
% % \]

For an initial $\rho = \cN(0, \eta^2 I_d)$, $\eta^2>0$, we calculate the \eqref{sb-step} and the \eqref{eq:GA} step and prove that the \eqref{sb-step}  step is an $\cO(\eps^2)$ approximation of the \eqref{eq:GA} step. 
% The gradient flow starting from $\cN(0, \eta^2)$ is available in closed form and given by $\rho (t) = \cN(0, 1 - (1-\eta^2)e^{-t})$. 
% We recover \eqref{sb-step} step by adequately comparing the log-concavity of $\cN(0, \eta^2 I_d)$ and $\cN(0,I_d)$, and consequently choosing the surrogate measure. This choice is based on the value of $\eta^2$. 
If $\eta^2 > 1$, then $\rho$ is less log-concave than $e^{-g}$, and therefore $\sigma = \rho^{-1} e^{-g} = \cN(0, \eta^2/(\eta^2-1) I_d)$. This surrogate measure satisfies Assumption~\ref{assumptions:sigma-eps-onestep}. By \eqref{eq:gaussian_bp}, we have that $\SB{\rho}{\eps}{1} = \cN\of{0, (C_{\eta^2/(\eta^2-1)}^\eps)^2\eta^2 I_d}$.
Now if $\eta^2 < 1$, then $\rho$ is more log-concave than $e^{-g}$, and therefore $\sigma = \rho e^g = \cN(0, \eta^2/(1-\eta^2) I_d)$. Note that this choice satisfies Assumption~\ref{assumptions:sigma-eps-onestep} (1-2, 4) but the ratio $\frac{\rho}{\sigma} = \rho^2 e^{g}$ is unbounded. Nevertheless, we calculate the \eqref{sb-step} update using this surrogate measure and demonstrate one-step convergence, indicating that our assumptions in this paper are stronger than required.
By \eqref{eq:gaussian_bp} we have $\SB{\rho}{\eps}{1} = \cN(0, (2 - C_{\eta^2/(1-\eta^2)}^\eps)^2\eta^2 I_d)$.
% \[
% \Id + \nabla f_{\eps, \sigma_0} = \of{2 - C_{\eta^2/(1-\eta^2)}^\eps} \Id \quad \text{ where }\quad  C_{\eta^2/(1-\eta^2)}^\eps = \frac{(1-\eta^2)}{\eta^2}\of{\sqrt{\frac{\eta^4}{(1-\eta^2)^2} + \frac{\eps^2}{4}} - \frac{\eps}{2}}\,.
% \]
% Therefore,
%\SP{Please add one more step in the calculation to help the reader.}\MA{Done.}
The corresponding \eqref{eq:GA} update is $\rho_\eps(1) = \of{\Id - \frac{\eps}{2}\of{\nabla \log \rho + \Id}}_\#\rho = \cN\of{0, \of{1 + \frac{\eps (1 - \eta^2)}{2 \eta^2}}^2 \eta^2 I_d}$.
% \[
% \Id - \nabla f_{\eps, \sigma} = C_{\eta^2/(\eta^2-1)}^\eps \Id \quad \text{ where }\quad  C_{\eta^2/(\eta^2-1)}^\eps = \frac{(\eta^2-1)}{\eta^2}\of{\sqrt{\frac{\eta^4}{(\eta^2-1)^2} + \frac{\eps^2}{4}} - \frac{\eps}{2}}\,.
% \]
% Therefore,

Again, we check the sharpness of Theorem~\ref{thm:one_step_convergence} by explicitly calculating the $\Was{2}$ distance between $\rho_\eps(1)$ and $\SB{\rho}{\eps}{1}$.
For any $\eta^2 > 0$, 
\begin{align*}
    \Was{2}\of{\rho_\eps(1), \SB{\rho_0}{\eps}{1}} &= \eta d\abs{2 - C^\eps_{\eta^2/(1-\eta^2)} - 1 - \frac{\eps (1-\eta^2)}{2\eta^2}} = \eta d\abs{1 - \sqrt{1 + \frac{\eps^2 (1-\eta^2)^2}{4 \eta^4}}}\,.
\end{align*}
% Using the Taylor series expansion about $\eps=0$, we obtain the following one-step rate
% \[
% \Was{2}\of{\SB{\rho_0}{\eps}{1}, \opt{\eps}{1}{\rho_0}} = \frac{\eps^2 d (1-\eta^2)^2}{8\eta^3} + o(\eps^2)\,.
% \]
% \[
% \Was{2}\of{\rho_\eps(1), \SB{\rho_0}{\eps}{1}} = \eta d\abs{\sqrt{1 + \frac{\eps^2 (1-\eta^2)^2}{4 \eta^4}} - 1}\,.
% \]
Therefore, again $\Was{2}\of{\rho_\eps(1), \SB{\rho_0}{\eps}{1}} = \cO(\eps^2)$ which is an order of magnitude better than Theorem~\ref{thm:one_step_convergence}.
As mentioned above, here is an example of a fast convergence rate even when the surrogate measure $\sigma$ does not satisfy Assumption~\ref{assumptions:sigma-eps-onestep}(3).

Iterates of the \eqref{sb-step} scheme can be defined beyond the one step in the similar manner if the sign convention remains same throughout the iterative process. The following calculation is done for $\eta^2 < 1$ which is not covered by our theorem. A similar calculation is valid for $\eta^2>1$ which is skipped. We claim that if $\rho$ is more log-concave than $\nu=\cN(0, I_d)$, then $\SB{\rho}{\eps}{1}$ is also more log-concave than $\nu$ for small enough $\eps>0$. This is because, using Taylor series expansion around $\eps=0$, we have 
\begin{align*}
    \of{2 - C^\eps_{\eta^2/(1-\eta^2)}}^2 \eta^2 &= \of{2 - \off{\frac{1-\eta^2}{\eta^2} \of{\sqrt{\frac{\eta^4}{(1-\eta^2)^2} + \frac{\eps^2}{4}} - \frac{\eps}{2}}}}^2 \eta^2\\
    &= \eta^2 + (1-\eta^2)\eps - \frac{\eps^2 (1-\eta^2)^2}{8\eta^2} + \cO(\eps^4).
\end{align*}
Therefore, as long as $\eps < \frac{\eta^2}{(1-\eta^2)}$, the SB steps approach monotonically toward the minimizer $\cN(0,I_d)$ and the sign convention remains the same. Again, since both \eqref{sb-step} and \eqref{eq:GA} schemes evolve via linear pushforwards, all steps are mean-zero Gaussian distributed. Denote $\SB{\rho_0}{\eps}{k} = \cN(0, \alpha_k^2 I_d)$ and $\rho_\eps(k) = \cN(0, \beta_k^2 I_d)$ where $\alpha_k^2$ and $\beta_k^2$ evolve via the following recursive relationship
$$\alpha_{k+1}^2 = \of{2 - C^\eps_{\alpha_k^2/(1-\alpha_k^2)}}^2\alpha_k^2 \quad \text{and} \quad \beta_{k+1}^2 = \of{1 + \frac{\eps (1 - \beta_k^2)}{\beta_k^2}}^2\beta_k^2.$$ 
The corresponding surrogate measures $\of{\sigma_{\eps}(k), k\in[N_\eps]}$ are given by $\sigma_{\eps}(k) = \cN\of{0, {\alpha_k^2}/{(1 - \alpha_k^2)} I_d}$. 

We again show that the assumptions of Theorem~\ref{corollary:kl-with-sb} are satisfied to prove that $(\rho_\eps(k), k\in[N_\eps])$ and $(\SB{\rho}{\eps}k{, k\in[N_\eps]})$ are first-order approximation schemes of the Wasserstein gradient flow $(\rho(t), t\in [0,T])$.
Using \eqref{eq:entropic_interpolation}, we have that
\[
I(\sigma_{\eps}(k)) - I((\sigma_{\eps}(k))^\eps_{1/2}) = \of{\frac{1 - \alpha_k^2}{\alpha_k^2}}\frac{d \delta_\eps^2}{\delta_\eps^2 + 4(1 + \delta_\eps)}\,,
\]
where $\delta_\eps = \frac{1}{2}\of{\eps -2 + \sqrt{4 + \eps^2}}$. The sequence $(\alpha_k^2, k\in[N_\eps])$ is monotonically increasing because
\[
\alpha_{k+1}^2 = \of{2- \sqrt{1 + \frac{\eps^2 (\alpha_k^2-1)^2}{4\alpha_k^4}} + \frac{\eps (1 - \alpha_k^2)}{2 \alpha_k^2}}^2 \alpha_k^2\,,
\]
and the prefactor of $\alpha_k^2$ can be easily checked to be greater than $1$. 
Moreover, each $\alpha_k^2$ is bounded from below by its initial value $\eta^2$. Therefore,
\[
\sup_{k\in[N_\eps]} \of{I(\sigma_\eps(k)) - I((\sigma_\eps(k))^\eps_{1/2})} \leq I(\sigma(0)) - I((\sigma(0))^\eps_{1/2}) = \of{\frac{1-\eta^2}{\eta^2}} \frac{d \delta_\eps^2}{\delta_\eps^2 + 4(1+\delta_\eps)}\,.
\]
Since $\delta_\eps = \cO(\eps)$, the RHS converges to zero as $\eps \to 0+$ at the order $\cO(\eps^2)$. 
Therefore, Assumption (1) of Theorem~\ref{corollary:kl-with-sb} is satisfied with rate $\cO(\eps^2)$. 

Now, we show that assumption (2) of Theorem~\ref{corollary:kl-with-sb} is satisfied. Again, in the Gaussian case, it suffices to show that $\of{{\alpha_k^2}/(1-\alpha_k^2), k\in[N_\eps]}$ and $(\beta_k^2, k\in[N_\eps])$ is bounded from below and above. Since $C^\eps_{\nu^2} < 1$ for any $\nu^2>0$, the sequence $(\alpha_k^2, k\in[N_\eps])$, and equivalently $(\alpha_k^2/(1-\alpha_k^2), k\in[N_\eps])$, is monotonically increasing. Then trivially, $(\alpha_k^2/(1-\alpha_k^2), k\in[N_\eps])$ is lower bounded by $\eta^2/(1-\eta^2)$ and $(\beta_k^2, k\in[N_\eps])$ are lower bounded by $\eta^2$. Take $\eps < \min\of{\frac{\eta^2}{1-\eta^2}, \frac{1}{3}}$. Using the identity $1-x\leq \sqrt{1 + x}$ for all $x\geq 0$, we have that $2 - C^\eps_{\alpha_k^2/(1-\alpha_k^2)} \leq 1 + \eps(1-\alpha_k^2)/\alpha_k^2$. Therefore, $\alpha_{k+1}^2 \leq \of{1 + \eps(1-\alpha_k^2)/\alpha_k^2}^2\alpha_k^2$ and $\beta_{k+1}^2 = \of{1 + \eps(1-\beta_k^2)/\beta_k^2}^2\beta_k^2$ give recursive inequalities of similar form. Since $\eps \leq \alpha_k^2/(1-\alpha_k^2)$, we have $\alpha_{k+1}^2 \leq \alpha_k^2 + 3\eps (1-\alpha_k^2)$. Recursively, $\alpha_{N_\eps}^2 \leq (1-3\eps)^{N_\eps}\eta^2 + {1 - (1-3\eps)^{N_\eps}}$, which is bounded from above. Therefore, $\of{\frac{\alpha_k^2}{(1-\alpha_k^2)}, k\in[N_\eps]}$ is bounded from above by $\frac{\eta^2}{1-\eta^2} + \frac{1 - (1-3\eps)^{N_\eps}}{(1-3\eps)^{N_\eps}(1-\eta^2)}$ and similarly, $(\beta_k^2, k\in[N_\eps])$ is bounded from above by $(1-3\eps)^{N_\eps}\eta^2 + \of{1 - (1-3\eps)^{N_\eps}}$, which is bounded as $\eps 
\downarrow 0$. Consequently, the assumptions of Theorem~\ref{corollary:kl-with-sb} are satisfied. Step 1 of the proof gives uniform convergence of \eqref{eq:GA} scheme and step 2 gives the uniform convergence of the \eqref{sb-step} scheme.

\subsection{Time reversal of gradient flows}\label{sec:example_reverse_flows}
Let us consider the time reversal of gradient flows of the functionals in the previous two examples. Specifically, if $\of{\rho(t), t\in[0,T]}$ is the gradient flow of a functional $\cF: \cP_2(\R^d) \to (-\infty, \infty]$ with velocity field $(v(t), t\in[0,T])$, then let $(\bar \rho(t), t\in[0,T])$ denote the time reversed flow defined as $\bar\rho(t) := \rho(T-t)$.
Naturally, the velocity field in the reverse process is negative of the velocity in the forward process, i.e. if $\of{\bar v(t), t\in[0,T]}$ denotes the velocity field of $\of{\bar\rho(t), t\in[0,T]}$, then $\bar{v}(t) = -v(T-t)$.

These absolutely continuous curves are not gradient flows and therefore do not benefit from the theoretical guarantees established in our theorems. Note that the corresponding PDEs, such as the backward heat equation, are well known to be ill-posed (\cite[page 129, eqn. 4.29]{olver2013introduction}). However, interestingly, we observe that, for the time reversed gradient flows of entropy and KL divergence functionals with Gaussian marginals, the \eqref{sb-step} scheme approximates these curves with the same convergence rates. For both the examples below, these convergence rates are derived by directly showing that $\SB{\cdot}{\eps}{1}$ is consistent \eqref{eq:consistency} and a contraction (a stronger version of contraction in limit \eqref{eq:contraction}). Following that, we use Theorem~\ref{thm:S_uniform_convergence} to prove that the \eqref{sb-step} iterates give a first-order approximation of the time reversed flow $\of{\bar\rho(t), t\in[0,T]}$. Through the examples, we also highlight a key issue with approximating reversed time gradient flows: to have a valid first-order approximation (Definition~\ref{def:first_order_approximation}) of $\of{\bar\rho(t), t\in[0,T]}$, the step size $\eps$ should be small enough where the bound on $\eps$ depends on the properties of the iterates encountered in the future. We include these calculations below for readers' curiosity and their practical importance in score-based generative modeling using diffusions  \cite{song2020score, lim2024score}. 

%Approximating the inverse problem via discrete steps relies on efficient estimation of the score function (gradient of logarithmic density), at each step. While recent advancements in generative artificial intelligence have successfully addressed this requirement via deep networks, the examination of this process as the time reversal of a Wasserstein gradient flow remains largely unexplored. Since the velocity field of the reverse gradient flow is the negative of the velocity field in the forward flow, we have all the information we need to run an SB scheme in the reverse direction.

 %Now we approximate the time reserved gradient flows of entropy and KL divergence functionals using our proposed SB scheme.

First, we consider the time reversed gradient flow of $\cF = \frac{1}{2} \Ent$. Then, $\bar{v}(t) = \frac{1}{2}\nabla \log \rho(T-t) = \frac{1}{2}\nabla \log \bar{\rho}(t)$. Now we calculate the first \eqref{sb-step} step. For $\rho = \rho(0)$, the surrogate measure is $\sigma = \sigma(0) = e^{\theta \log \bar\rho}$ where $\theta$ is chosen to ensure integrability of $e^{\theta \log \bar \rho}$. Therefore, we may choose $\theta = 1$, which implies $\sigma_\eps(0) = \sigma = \rho$ and the \eqref{sb-step} step is $\SB{\rho}{\eps}{1} = \of{\BPbase{{\rho}}{\eps}}_{\#}\rho$. Notice that the sign of $\theta$ is flipped compared to Section~\ref{sec:example_entropy}. This change ensures that the \eqref{sb-step} for the reverse flow is the opposite direction to the \eqref{sb-step} step of the forward flow. This is interesting as we can approximate both forward and reverse time gradient flow of entropy functional using appropriate direction of \eqref{sb-step} steps, calculated only using the current measure.  
%\SP{The above step needs to be explained more clearly.}\MA{Done.}

Now for an initial $\rho(0) = \cN(0, \eta^2 I_d)$, we know that the gradient flow $(\rho(t), t\in[0,T])$ is given by $\rho(t) = \cN(0, (\eta^2 + t)I_d)$. Then, we have $\rho = \bar \rho(0) = \cN(0, (\eta^2 + T) I_d)$ and $\bar\rho(\eps) = \cN(0, (\eta^2 + T - \eps)I_d)$. Now we show that $\SB{\rho}{\eps}{1}$ is an $\cO(\eps^2)$ close approximation of $\bar\rho(\eps)$. For any $\eta^2>0$ and $\eps < \eta^2 + T$,
\begin{align*}
    \Was{2}\of{\bar\rho(\eps), \SB{\rho}{\eps}{1}} &= d\abs{C^\eps_{\eta^2 + T} \sqrt{\eta^2 + T} - \sqrt{\eta^2 + T - \eps}} \\
    &= d\sqrt{\eta^2 +T}\abs{\sqrt{1 + \frac{\eps^2}{4(\eta^2+T)^2}} - \frac{\eps}{2(\eta^2 + T)} - \sqrt{1 - \frac{\eps}{\eta^2 +T}}}\,.
\end{align*}
Denote $x = \eps/(\eta^2+T)$, then the term containing $\eps$ above is $c_\eps = \sqrt{1 + x^2/4} - (x/2) - \sqrt{1-x}$. Using the identity $\sqrt{a} - \sqrt{b} = (a-b)/(\sqrt{a}+\sqrt{b})$ with $a = 1 + x^2/4$ and $b = \of{(x/2) + \sqrt{1-x}}^2$, we can be convinced that constant $c_\eps > 0$. Further, using the identity $\sqrt{1+a} < 1 + a/2$ for $a>0$ and $\sqrt{1-b} > 1-b/2$ for $0<b<1$, we have that $c_\eps < x^2/4 = \frac{\eps^2}{4(\eta^2 + T)^2}$. Therefore, we have the consistency \eqref{eq:consistency} result 
\begin{equation}\label{eq:reverse_entropy_consistency}
    \Was{2}\of{\bar\rho(\eps), \SB{\rho}{\eps}{1}} < \frac{d\eps^2}{4(\eta^2 + T)^{3/2}}\,.
\end{equation}

Again, since \eqref{sb-step} scheme evolves via linear pushforwards, the sequence of \eqref{sb-step} steps are $\SB{\rho}{\eps}{k} = \cN(0, \alpha_k^2 I_d)$ where $\alpha_k^2$ follows the recursive relationship $\alpha_{k+1}^2 = \of{C^\eps_{\alpha_k^2}}^2 \alpha_k^2$. Define the piecewise constant interpolation 
\[
\SB{t}{\eps}{} = \SB{\rho}{\eps}{k} \quad t\in[k\eps, (k+1)\eps)\,.
\]
Unlike the previous two examples, since $\of{\bar\rho(t), t\in[0,T]}$ is not a gradient flow, we do not invoke Theorem~\ref{thm:SB_uniform_convergence} to prove the uniform convergence of SB scheme. 
We instead directly show that $\of{\SB{t}{\eps}{}, t\in[0,T]}$ is a first-order approximation (Definition~\ref{def:first_order_approximation}) of $\of{\bar{\rho}(t), t\in[0,T]}$. We have already shown consistency in \eqref{eq:reverse_entropy_consistency}.
Now we show that $\SB{\cdot}{\eps}{1}$, restricted to Gaussian measures, is Lipschitz with the Lipschitz constant of the order $(1 + \cO(\eps))$.
%\SP{We should write somewhere in Section 4 that one of our inabilities is to directly verify that the SB scheme in general gives a contraction and hence Theorem 2 applies.}
Take $\rho_1 = \cN(0, \sigma_1^2 I_d)$ and $\rho_2 = \cN(0, \sigma_2^2 I_d)$. Then, $\Was{2}\of{\SB{\rho_1}{\eps}{1}, \SB{\rho_2}{\eps}{1}}$ can be bounded above by the following string of inequalities
\begin{align*}
     &= d\abs{\sigma_1\sqrt{1 + \frac{\eps^2}{4\sigma_1^4}} - \frac{\eps}{2\sigma_1} - \sigma_2\sqrt{1 + \frac{\eps^2}{4\sigma_2^4}} + \frac{\eps}{2\sigma_2}} &\\
    &\leq \frac{d\eps}{2\sigma_1 \sigma_2}\abs{\sigma_1 - \sigma_2}+ d\abs{\sqrt{\sigma_1^2 + \frac{\eps^2 }{4\sigma_1^2}} - \sqrt{\sigma_2^2 + \frac{\eps^2 }{4\sigma_2^2}}} & \of{\text{ triangle inequality}}\\
    &= \frac{d\eps}{2\sigma_1 \sigma_2}\abs{\sigma_1 - \sigma_2}+ d \frac{\abs{\sigma_1^2 - \sigma_2^2} \abs{1 - \frac{\eps^2}{4\sigma_1^2 \sigma_2^2}}}{{\sqrt{\sigma_1^2 + \frac{\eps^2 }{4\sigma_1^2}} + \sqrt{\sigma_2^2 + \frac{\eps^2 }{4\sigma_2^2}}}} & \of{\text{$\sqrt{a} -\sqrt{b} = \frac{a-b}{\sqrt{a} + \sqrt{b}}$}}\\
    &\leq \frac{d\eps}{2\sigma_1 \sigma_2}\abs{\sigma_1 - \sigma_2} + d \abs{1 - \frac{\eps^2}{4\sigma_1^2 \sigma_2^2}}\abs{\sigma_1 - \sigma_2} & \of{\text{$\sqrt{\sigma_i^2 + \eps^2/4\sigma_i^2} > \sigma_i$ for $i\in\{1,2\}$}}\,.
\end{align*}
For $\eps < 2\sigma_1 \sigma_2$, $\of{1 - \frac{\eps^2}{4\sigma_1^2 \sigma_2^2}} >0$ and hence we have the contraction
\begin{equation}\label{eq:reverse_entropy_contraction}
    \Was{2}\of{\SB{\rho_1}{\eps}{1}, \SB{\rho_2}{\eps}{1}} < \of{1 + \frac{\eps}{2\sigma_1 \sigma_2}}\Was{2}(\rho_1, \rho_2)
\end{equation}

The sequence $\of{\alpha_k^2, k\in[N_\eps]}$ decreases monotonically because
\[
\alpha_{k}^2 =  \of{C^\eps_{\alpha_{k-1}^2}}^2 \alpha_{k-1}^2 =  \of{\sqrt{1 + \frac{\eps^2}{4\alpha_{k-1}^4}} - \frac{\eps}{2\alpha_{k-1}^2}}^2\alpha_{k-1}^2\,.
\]
and $\sqrt{1+x^2}-x<1$ for all $x>0$. Now we show that $\of{\alpha_k^2, k\in[N_\eps]}$ is bounded from below because for small enough $\eps$. Using $\sqrt{1+x} \leq 1 + x/2$,
\begin{align*}
    \alpha_{k}^2 &= \alpha_{k-1}^2 - \eps \sqrt{1 + \frac{\eps^2}{4\alpha_{k-1}^4}} + \frac{\eps^2}{2\alpha_{k-1}^2} \geq \alpha_{k-1}^2 - \eps + \frac{\eps^2}{2\alpha_{k-1}^2}\of{1 - \frac{\eps}{4\alpha_{k-1}^2}}\,.
\end{align*}
If $\eps$ is uniformly less then $4\alpha_{k-1}^2$ for all $k\in[N_\eps]$, then $\alpha_k^2 \geq \alpha_{k-1}^2 - \eps$. Consequently, the lowest value is $\alpha_{N_\eps}^2 \geq (\eta^2 + T) - \eps N_\eps = \eta^2$. Therefore, choosing $\eps < \eta^2$, we get that $\of{\alpha_k^2, k\in[N_\eps]}$ is bounded from below by $\eta^2$. This highlights an inherent issue with approximating time-reversed gradient flows: a valid approximation depends on the choice of \(\epsilon\), which in turn depends on the variance of the flow at a future time. As a result, if we run an SB scheme with \(\epsilon\) greater than the variance of \(\bar{\rho}(T)\), we might encounter the problem of an SB step with a negative variance.
%\SP{How do you get monotonicity from the above?}\MA{Monotonicity follows from the fact that $\alpha_k^2 = \of{C^\eps_{\alpha_{k-1}^2}}^2 \alpha_{k-1}^2$ and $C^\eps_{\alpha_{k-1}^2} < 1$.}

By the triangle inequality
% Using the triangular argument from Theorem~\ref{thm:S_uniform_convergence}, for any $k\in[N_\eps]$, 
\begin{align*}
    \Was{2}\of{\bar\rho((k+1)\eps), \SB{\rho}{\eps}{k+1}} &\leq \Was{2}\of{\bar\rho((k+1)\eps), \SB{\bar\rho(k\eps)}{\eps}{1}} + \Was{2}\of{\SB{\bar\rho(k\eps)}{\eps}{1}, \SB{\rho}{\eps}{k+1}}\,.
\end{align*}
Using consistency \eqref{eq:reverse_entropy_consistency}, 
$$\Was{2}\of{\bar\rho((k+1)\eps), \SB{\bar\rho(k\eps)}{\eps}{1}} < \frac{d \eps^2}{4(\eta^2 + T - k\eps)^{3/2} }\,,$$ 
and using Lipschitzness \eqref{eq:reverse_entropy_contraction},
\[
\Was{2}\of{\SB{\bar\rho(k\eps)}{\eps}{1}, \SB{\rho}{\eps}{k+1}} < \of{1 + \frac{\eps}{2(\eta^2 + T - k\eps)\alpha_{k}}} \Was{2}\of{\bar\rho(k\eps), \SB{\rho}{\eps}{k}}\,.
\]
Since for all $k\in[N_\eps]$, $\alpha_k^2 \geq \eta^2$, we have
\[
\Was{2}\of{\bar\rho((k+1)\eps), \SB{\rho}{\eps}{k+1}} < \frac{d\eps^2}{4\eta^3} + \of{1 + \frac{\eps}{2\eta^2}}\Was{2}\of{\bar\rho(k\eps), \SB{\rho}{\eps}{k}}\,.
\]
Recursively, 
$$\Was{2}\of{\bar\rho(k\eps), \SB{\rho}{\eps}{k}} \leq \frac{d\eps}{2\eta} \of{1 + \frac{\eps}{2\eta^2}}^{N_\eps} \,.$$ 
Since $N_\eps = \lfloor T\eps^{-1} \rfloor$, therefore, $\lim_{\eps \to 0} \sup_{k\in[N_\eps]} \Was{2}\of{\bar\rho(k\eps), \SB{\rho}{\eps}{k}} = 0$, proving the uniform convergence of $\of{\SB{\rho}{\eps}{k}, k\in[N_\eps]}$ to the curve $\of{\bar \rho(t), t\in[0,T]}$ as $\eps\to 0+$.

%\SP{This is a particularly important example. It's significance in the context of generative modeling should be clearly explained.}\MA{Done in the paragraph below.}

Finally, we consider the case of time reversed case of gradient flow of $\cF = \frac{1}{2}\kl{\cdot}{\nu}$ with $\nu = e^{-g} = \cN(0, I_d)$. Since this case is very similar to the previous examples on the gradient flow of KL divergence and time reversed gradient flow of entropy, we limit explicit calculations. The velocity field $\bar v(t) = \frac{1}{2}\of{\Id + \nabla \log \bar\rho(t)}$. For an initial choice $\rho(0) = \cN(0, \eta^2)$, the time reversed gradient flow, starting from time $t=T$, is $\bar\rho(t) = \cN\of{0, \of{\eta^2 e^{-(T-t)} + (1-e^{-(T-t)})} I_d}$. For brevity of notation, let $\rho = \bar\rho(0) = \cN(0, \tau^2)$ where $\tau^2 := \eta^2 e^{-T} + (1 - e^{-T})$. Now we will calculate the \eqref{sb-step} step. Again, we choose $\theta$ such that the surrogate measure $\sigma = \sigma(0) = \of{\rho e^g}^{\theta}$ is integrable. Using the same logic as Section~\ref{sec:example_kl}, if $\eta^2 > 1$ (equivalently $\tau^2>1$), then $\theta$ may be taken to be $-1$ giving $\sigma = \rho^{-1} e^{-g} = \cN\of{0, (\tau^2/(\tau^2-1))I_d}$ and $\SB{\rho}{\eps}{1} = \of{2\Id - \BPbase{\sigma}{\eps}}_\#\rho$. Whereas if $\eta^2 < 1$ (equivalently $\tau^2 < 1$), then $\theta$ can be taken to be $1$ giving $\sigma = \cN(0, (\tau^2/(1- \tau^2)) I_d)$ and $\SB{\rho}{\eps}{1} = \of{\BPbase{\sigma}{\eps}}_\#\rho$. We will present calculations for $\eta^2<1$ case. Similar calculations follow for $\eta^2>1$ and have been skipped. 

 Now we prove that the \eqref{sb-step}  step is an $\cO(\eps^2)$ approximation of $\bar\rho(\eps)$. Note that for $\tau^2< 1$ and $\eps < \log \left((1-\tau^2)^{-1}\right)$, $\SB{\rho}{\eps}{1} = \cN\of{0, \of{C^\eps_{\tau^2/(1-\tau^2)}}^2 \tau^2}$ and $\bar\rho(\eps) = \cN\of{0, \tau^2 + (\tau^2 -1)(e^\eps -1)}$. Then
 \begin{align*}
     \Was{2}\of{\bar\rho(\eps), \SB{\rho}{\eps}{1}} &= d \tau \abs{ \sqrt{1 + \frac{\eps^2 (1-\tau^2)^2}{4\tau^4}} - \frac{\eps (1-\tau^2)}{2\tau^2} -  \sqrt{1 - \frac{1-\tau^2}{\tau^2}(e^\eps-1)}}.
 \end{align*}
 Denote $x = \frac{\eps(1-\tau^2)}{\tau^2}$, then the term containing $\eps$ above is $c_\eps = \sqrt{1 + \frac{x^2}{4}} - \frac{x}{2} - \sqrt{1 - x\of{\frac{e^\eps - 1}{\eps}}}$. 
 %Using the identity $\sqrt{a} - \sqrt{b} = \frac{a-b}{\sqrt{a}+\sqrt{b}}$ with $a = 1 + \frac{x^2}{4}$ and $b = \of{\frac{x}{2} + \sqrt{ 1 - x \of{\frac{e^\eps - 1}{\eps}}}}^2$, 
 It is easy to see that $c_\eps > 0$. Further, using the identity $\sqrt{1+a} < 1 + a/2$ for $a>0$ and $\sqrt{1-b} > 1-b/2$ for $0<b<1$, we have that $c_\eps < \frac{x}{2}\of{\frac{e^\eps-1}{\eps}} + \frac{x^2}{8}$. Therefore, we have the consistency \eqref{eq:consistency} result
 \begin{equation}\label{eq:reverse_kl_consistency}
     \Was{2}\of{\bar\rho(\eps), \SB{\rho}{\eps}{1}} < d\eps^2\of{\frac{(1-\tau^2)^2}{8\tau^3} + \frac{C(1-\tau^2)}{2\tau}} \quad \text{for some constant }C.
 \end{equation}

 Again since the \eqref{sb-step} scheme evolves via linear pushforwards, $\SB{\rho}{\eps}{k}$ is a Gaussian measure for all $k\in[N_\eps]$ and let $\SB{\rho}{\eps}{k} = \cN(0, \alpha_k^2 I_d)$. Then $\alpha_k^2$ evolves via the recursive relationship $\alpha_{k+1}^2 = \of{C^\eps_{\alpha_k^2/(1-\alpha_k^2 )}}^2 \alpha_k^2$. We prove uniform convergence of the SB scheme in the same way as the time reversed flow of entropy functional. The consistency is show above and now we show that the \eqref{sb-step} is a contraction for Gaussian measures. Take $\rho_1 = \cN(0, \sigma_1^2 I_d)$ and $\rho_2 = \cN(0, \sigma_2^2 I_d)$. Then, exactly as in the case of time-reversed gradient flow of entropy, %$\Was{2}\of{\SB{\rho_1}{\eps}{1}, \SB{\rho_2}{\eps}{1}} $ can be upper bounded by 
 %the following string of inequalities
\begin{align*}
 %   &= d\abs{\sigma_1 \sqrt{1 + \frac{\eps^2 (1-\sigma_1^2)^2}{4\sigma_1^4}} - \frac{\eps (1-\sigma_1^2)}{2\sigma_1} - \sigma_2 \sqrt{1 + \frac{\eps^2 (1-\sigma_2^2)^2}{4\sigma_2^4}} + \frac{\eps (1-\sigma_2^2)}{2\sigma_2}} &\\
%    &\leq \frac{d\eps}{2}\of{1 + \frac{1}{\sigma_1\sigma_2}} \abs{\sigma_1 -\sigma_2} +  d\abs{\sqrt{\sigma_1^2 + \frac{\eps^2 (1-\sigma_1^2)^2}{4\sigma_1^2}} - \sqrt{\sigma_2^2 + \frac{\eps^2 (1-\sigma_2^2)^2}{4\sigma_2^2}}} & \of{\text{triangle inequality}}\\
%    &\leq \frac{d\eps}{2}\of{1 + \frac{1}{\sigma_1\sigma_2}} \abs{\sigma_1 -\sigma_2} +  d\frac{\abs{{\sigma_1^2 + \frac{\eps^2 (1-\sigma_1^2)^2}{4\sigma_1^2}} - {\sigma_2^2 - \frac{\eps^2 (1-\sigma_2^2)^2}{4\sigma_2^2}}}}{{\sqrt{\sigma_1^2 + \frac{\eps^2 (1-\sigma_1^2)^2}{4\sigma_1^2}} + \sqrt{\sigma_2^2 + \frac{\eps^2 (1-\sigma_2^2)^2}{4\sigma_2^2}}}} & \of{\sqrt{a}-\sqrt{b} = \frac{a-b}{\sqrt{a}+\sqrt{b}}}\\
    % &\leq \frac{d\eps}{2}\of{1 + \frac{1}{\sigma_1\sigma_2}} \abs{\sigma_1 -\sigma_2} +  d\frac{(\sigma_1^2 - \sigma_2^2)\abs{1 + \frac{\eps^2}{4} - \frac{\eps^2}{4\sigma_1^2\sigma_2^2}}}{{\sqrt{\sigma_1^2 + \frac{\eps^2 (1-\sigma_1^2)^2}{4\sigma_1^2}} + \sqrt{\sigma_2^2 + \frac{\eps^2 (1-\sigma_2^2)^2}{4\sigma_2^2}}}} & \of{}\\
 \Was{2}\of{\SB{\rho_1}{\eps}{1}, \SB{\rho_2}{\eps}{1}}   &< \frac{d\eps}{2}\of{1 + \frac{1}{\sigma_1\sigma_2}} \abs{\sigma_1 -\sigma_2} +  d \abs{1 - \frac{\eps^2}{4} \of{\frac{1-\sigma_1^2\sigma_2^2}{\sigma_1^2\sigma_2^2}}}  \abs{\sigma_1 -\sigma_2}. 
 %&\of{\sigma_i^2 + \frac{\eps^2 (1-\sigma_i^2)}{4\sigma_i^2} > \sigma_i^2} &\,.
\end{align*}
For $\eps < \frac{2\sigma_1 \sigma_2}{\sqrt{1 - \sigma_1^2 \sigma_2^2}}$, $\of{1 - \frac{\eps^2}{4} \of{\frac{1 - \sigma_1^2 \sigma_2^2}{\sigma_1^2 \sigma_2^2}}}>0$ and hence we have the contraction
\begin{equation}\label{eq:reverse_kl_contraction}
    \Was{2}\of{\SB{\rho_1}{\eps}{1}, \SB{\rho_2}{\eps}{1}} < \of{1 + \frac{\eps}{2}\of{1 + \frac{1}{\sigma_1 \sigma_2}}}\Was{2}\of{\rho_1, \rho_2}\,.
\end{equation}

Because $C^\eps_{\eta^2} \leq 1$ for all $\eta^2>0$, the sequence $\of{\alpha_k^2; k\in[N_\eps]}$ is monotonically decreasing and bounded from above by its initial value $\tau^2$. 
%Further, $\of{\alpha_k^2; k\in[N_\eps]}$ is bounded from below because 
%\[
%\alpha_k^2 = \of{C^\eps_{\alpha_{k-1}^2/(1 - \alpha_{k-1}^2)}}^2\alpha_{k-1}^2 = \alpha_{k-1}^2 +\frac{\eps^2 (1-\alpha_{k-1}^2)^2}{2\alpha_{k-1}^2} - \eps (1-\alpha_{k-1}^2)\sqrt{1 + \frac{\eps^2 (1-\alpha_{k-1}^2)^2}{4\alpha_{k-1}^4}}\,.
%\]
%Using $\sqrt{1+x} < 1+x/2$ for any $x>0$,
%\[
%\alpha_k^2 > \alpha_{k-1}^2 - \eps(1 - \alpha_{k-1}^2) + \frac{\eps^2 (1-\alpha_{k-1}^2)^2}{2\alpha_{k-1}^2} \of{1 - \frac{\eps(1-\alpha_{k-1}^2)}{4 \alpha_{k-1}^2}}\,.
%\]
%If $\eps$ is uniformly less than $\frac{4\alpha_{k-1}^2}{1-\alpha_{k-1}^2}$ for all $k\in[N_\eps]$, then $\alpha_{k}^2 > \alpha_{k-1}^2 - \eps(1-\alpha_{k-1}^2) = (1+\eps)\alpha_{k-1}^2 - \eps$. Recursively, $$\alpha_{N_\eps}^2 > (1+\eps)^{N_\eps}\tau^2 -\eps \sum_{k=0}^{N_\eps - 1} (1+\eps)^j = (1+\eps)^{N_\eps}\tau^2 - \of{(1+\eps)^{N_\eps}-1} = 1 - (1+\eps)^{N_\eps}(1-\tau^2)\,.$$
%Using the identity $(1+\eps)^{T/\eps} < \eps^T$ for all $\eps>0$, we finally have that $\alpha_{N_\eps}^2 > 1 - e^T(1-\tau^2)  = \eta^2$. 
Similar to the reversed time gradient flow of entropy, we get that the variance $(\alpha_k^2; k\in[N_\eps])$ is bounded from below if $\eps$ is less than $\frac{4\eta^2}{1-\eta^2}$, which depends on the variance at time $t=T$.
Using the triangular argument from Theorem~\ref{thm:S_uniform_convergence}, for any $k\in[N_\eps]$, 
%\begin{align*}
%    \Was{2}\of{\bar\rho((k+1)\eps), \SB{\rho}{\eps}{k+1}} &\leq \Was{2}\of{\bar\rho((k+1)\eps), \SB{\bar\rho(k\eps)}{\eps}{1}} + \Was{2}\of{\SB{\bar\rho(k\eps)}{\eps}{1}, \SB{\rho}{\eps}{k+1}}\,.
%\end{align*}
%Using consistency \eqref{eq:reverse_kl_consistency}, 
%$$\Was{2}\of{\bar\rho((k+1)\eps), \SB{\bar\rho(k\eps)}{\eps}{1}} < d\eps^2\of{\frac{(1-\beta_{k}^2)^2}{8\beta_{k}^3} + \frac{C(1-\beta_{k}^2)}{2\beta_{k}}}\,,$$ 
%where $\beta_k^2 = \eta^2 e^{-(T-k\eps)} + (1 - e^{-(T-k\eps)})$ and using contraction \eqref{eq:reverse_kl_contraction}
%\[
%\Was{2}\of{\SB{\bar\rho(k\eps)}{\eps}{1}, \SB{\rho}{\eps}{k+1}} < \of{1 + \frac{\eps}{2}\of{1 + \frac{1}{\alpha_k \beta_k}}} \Was{2}\of{\bar\rho(k\eps), \SB{\rho}{\eps}{k}}\,.
%\]
%Since for all $k\in[N_\eps]$, $\alpha_k^2 \geq \eta^2$,
%\[
%\Was{2}\of{\bar\rho((k+1)\eps), \SB{\rho}{\eps}{k+1}} < \frac{d\eps^2}{4\eta^3} + \of{1 + \frac{\eps}{2\eta^2}}\Was{2}\of{\bar\rho(k\eps), \SB{\rho}{\eps}{k}}\,.
%\]
%Recursively, 
$$\Was{2}\of{\bar\rho(k\eps), \SB{\rho}{\eps}{k}} \leq d\eps\frac{(1 + 4\eta^2 C)}{4\eta(\eta^2+1)} \of{1 + \frac{\of{1 + \eta^{-2}}\eps}{2}}^{N_\eps} \,.$$ 
Therefore, $\lim_{\eps \to 0} \sup_{k\in[N_\eps]} \Was{2}\of{\bar\rho(k\eps), \SB{\rho}{\eps}{k}} = 0$, proving the uniform convergence of $\of{\SB{\rho}{\eps}{k}; k\in[N_\eps]}$ to the curve $\of{\bar \rho(t); t\in[0,T]}$ as $\eps\to0+$.

\bibliographystyle{alpha}
\bibliography{sample}
%\newpage

\section*{Notations}

\begin{table}[!ht]
\centering
\begin{tabular}{ c  c } 
 \hline
 \textbf{Symbol} & \textbf{Explanation}  \\ [0.5ex] 
 \hline
 $C^{d}[0,\infty)$ & Collection of continuous paths $\omega: [0,\infty) \to \mathbb{R}^{d}$ \\
 \hline
 $C^{k}(\mathbb{R}^{d})$ & $k$-times continuously differentiable functions $\mathbb{R}^{d} \to \mathbb{R}$ \\
  \hline
 $\cP_2(\R^d)$ & Space of twice integrable probability measures on $\R^d$\\
  \hline
 $\cP_2^{ac}(\R^d)$ & \makecell{Space of twice integrable and absolutely continuous probability \\ measures (with respect to Lebesgue measure) on $\R^d$}\\
  \hline
  % $T_\mu$ & Pushforward of measure $\mu$ by map $T$\\
  % \hline
  $C_c^\infty(\R^d)$ & Space of smooth and compactly supported functions in $\R^d$\\
   \hline
   $\Was{2}(\mu, \nu)$ & Wasserstein-2 distance between measures $\mu$ and $\nu$ in $\cP_2(\R^d)$\\
   \hline
   % $\Was{\boldsymbol{\mu}}(\mu_i, \mu_j)$ & Pseudo 2-Wasserstein distance induced by $\boldsymbol{\mu} \in \Gamma(\mu_1, \dots, \mu_n)$ where $i,j \in [n]$\\
   % \hline
   $T_\mu^\nu$ & Optimal transport map from $\mu$ to $\nu$ (if it exists)\\
   \hline
 $(\pi_t, t \geq 0)$ & Coordinatewise projections on $C^{d}[0,\infty)$, $\pi_s(\omega) = \omega_s$ \\
  \hline
 $\BMstatic{\rho}{\eps}$ & \makecell{$\BMstatic{\rho}{\eps}(dx,dy) = (2\pi\eps)^{-d/2}\rho(x)\exp\left(-\|x-y\|^2/2\eps\right)dxdy$} \\
  \hline
 $(W_x, x \in \mathbb{R}^{d})$ & Wiener measure on $C^{d}[0,\infty)$ with initial condition $\delta_x$ \\
  \hline
 $W$ & Reversible Wiener measure on $C^{d}[0,\infty)$ \\
  \hline
  $R_{\eps}$ & Law of reversible Brownian motion with diffusion $\eps$ \\ \hline
 $(P_t, t \geq 0)$ & Brownian semigroup \\
  \hline
 $(p_{t}(\cdot,\cdot), t > 0)$ & Brownian transition densities \\
  \hline
 $Q$ & Law of stationary Langevin diffusion on $C^{d}[0,\infty)$ \\
  \hline
 $\LDstatic{\rho}{\eps}$ & Joint law of Langevin diffusion at time 
$0$ and $\eps$, i.e.\ $(\pi_0,\pi_{\eps})_{\#}Q$\\
  \hline
 $(q_{t}(\cdot,\cdot), t > 0)$ & Transition densities for Langevin diffusion \\
  \hline
 $L$ & Extended generator for the Langevin diffusion \\
  \hline
 $(G_t, t \geq 0)$ & Langevin semigroup \\
  \hline
 $\SBstatic{\rho}{\eps}$ & $\eps$-static $\Schro$ bridge with marginals equal to $\rho$ \\
  \hline
  $(f_\eps,\eps > 0)$ & Entropic potentials associated to $(\SBstatic{\rho}{\eps},\eps > 0)$ \\ \hline
  $\BPbase{\rho}{\eps}$ & Barycentric projection associated with $\SBstatic{\rho}{\eps}$, $\BPbase{\rho}{\eps}(x) = \Exp{\SBstatic{\rho}{\eps}}[Y|X=x]$\\ \hline
 $\SBdynam{\rho}{\eps}$ & \makecell{Law of $\eps$-dynamic $\Schro$ bridge on $C^{d}[0,1]$\\ with initial and terminal distribution $\rho$} \\
  \hline
 % $(L_{t}^{\eps}, t \in [0,1])$ & Family of extended generators for $\eps$-dynamic $\Schro$ bridge \\ 
 % \hline
 $(\rho_t^{\eps}, t \in [0,1])$ & Entropic interpolation from $\rho$ to itself, $\rho_{t}^{\eps} = (\pi_t)_{\#}\SBdynam{\rho}{\eps}$ \\ \hline
 $\|\cdot\|_{\psi_{1}}$ & Subexponential norm of a random variable, \cite[Definition 2.7.5]{vershynin-hdp}\\ \hline
 $\text{Tan}(\mu)$ & Tangent bundle of $\cP_2(\R^d)$ at $\mu$\\
\hline
 $\Ent(\mu)$ & Entropy of $\mu\in\cP_2^{ac}(\R^d)$\\
 \hline
 $H(\mu|\nu)$& Kullback-Leibler divergence between $\mu$ and $\nu$\\
 \hline
 $I(\rho)$ & Fisher information of $\rho$, $I(\rho) = \Exp{\rho}\|\nabla \log \rho\|^{2}$ for $\rho \in \cP_{2}^{ac}(\mathbb{R}^{d})$ \\ \hline
 % $\abs{\partial \cF}(\mu)$ & Metric slope of $\cF$ at $\mu$\\
 % \hline
 % $D(\cF)$ & Domain of functional $\cF$\\
 % \hline
 % $D(\abs{\partial \cF})$ & Subset of $D(\cF)$ for which $\abs{\partial \cF}$ exists\\
 % \hline 
 % $ \boldsymbol{\partial}\cF(\mu)$ & General Fréchet subdifferential of $\cF$ at $\mu$\\
 % \hline
 % $\tilde{\boldsymbol{\partial}} \cF(\mu)$ & General strong Fréchet subdifferential of $\cF$ at $\mu$\\
 % \hline
 % $\partial^\circ \cF(\mu)$ & Minimal selection subdifferential of $\cF$ at $\mu$\\
 % \hline
% $\eps$ & Step size/regularization parameter/temperature\\
% \hline
 $\SB{\mu_0}{\eps}{k}$ & $k$th iterate of Schrödinger bridge scheme starting from $\mu_0$ with stepsize $\eps$\\
 \hline
$\opt{\eps}{k}{\mu_0}$ & \makecell{$k$th iterate of pushforward of form $(\Id + \eps v_{\eps})_{\#}\mu_{0}$, where \\ each $v_{\eps}: \mathbb{R}^{d} \to \mathbb{R}^{d}$ is a vector field specified in context} \\
% Euler iterate (explicit or implicit) starting from $\mu_0$ with stepsize $\eps$\\
\hline
% $\PP{\mu_0}{\eps}{k}$ & $k$th iterate of proximal point scheme starting from $\mu_0$ with stepsize $\eps$\\
% \hline
 % $\SB{t}{\eps}{}$ & Piecewise constant interpolation of SB scheme $(\SB{\mu_0}{\eps}{k}; k\in[T\eps^{-1}])$ at time $t\in[0,T]$\\
 % \hline
% $\GA{t}{\eps}{}$ & Piecewise constant interpolation of GA scheme $(\GA{\mu_0}{\eps}{k}; k\in[T\eps^{-1}])$ at time $t\in[0,T]$\\
% \hline
% $\PP{t}{\eps}{}$ & Piecewise constant interpolation of PP scheme $(\PP{\mu_0}{\eps}{k}; k\in[T\eps^{-1}])$ at time $t\in[0,T]$\\
% \hline
\end{tabular}

\label{table:notation-sumary}
\end{table}

%\newpage

%\section*{Appendix}
%\input{arxiv_version/appendix}

%\newpage
%\section*{Itemized Tasks}
%\input{arxiv_version/UnfinishedTasks}

\end{document}